\newcommand{\setpath}{}
\RequirePackage{fix-cm}
 \documentclass[smallextended,onecolumn]{\setpath%
svjour_simple} \smartqed  

\usepackage{graphicx}

\usepackage{amsmath} 
\usepackage{amssymb}  
\usepackage{subfigure}  
\usepackage{enumerate}%
\usepackage{multirow}

\usepackage{stackengine}

\renewcommand{\qed}{\nobreak \ifvmode \relax \else
      \ifdim\lastskip<1.5em \hskip-\lastskip
      \hskip1.5em plus0em minus0.5em \fi \nobreak
      \vrule height0.75em width0.5em depth0.25em\fi}

\usepackage{amsthm}

\newtheoremstyle{plainbigtitle}  
  {0.5\topsep}   
  {0.5\topsep}   
  {\itshape}  
  {0pt}       
  {\bfseries} 
  {.}         
  {5pt plus 1pt minus 1pt}  
  {\bfseries\boldmath#1 #2\thmnote{ (#3)}}          

\newtheoremstyle{definitionbigtitle}  
  {0.5\topsep}   
  {0.5\topsep}   
  {\upshape}  
  {0pt}       
  {\bfseries} 
  {.}         
  {5pt plus 1pt minus 1pt}  
  {\bfseries\boldmath#1 #2\thmnote{ (#3)}}          

\newcommand{\newtheoremwithoutindent}[2]{\newtheorem{#1}{#2}\AfterEndEnvironment{#1}{\noindent\ignorespaces}}

\theoremstyle{plainbigtitle}
\newtheoremwithoutindent{theorem}{Theorem}
\newtheoremwithoutindent{lemma}{Lemma}
\newtheoremwithoutindent{proposition}{Proposition}
\newtheoremwithoutindent{algorithm}{Algorithm}
\newtheoremwithoutindent{corollary}{Corollary}
\newtheoremwithoutindent{problem}{Problem}
\newtheoremwithoutindent{assumption}{Assumption}

\theoremstyle{definitionbigtitle}
\newtheoremwithoutindent{definition}{Definition}
\newtheoremwithoutindent{remark}{Remark}
\newtheoremwithoutindent{example}{Example}

\begin{document}
%
%
%
%

\newcommand{\function}{g} 
\newcommand{\Si}[1]{s_{#1}} 
\input{\setpath%
tubmacros}



\newcommand{\newPa}[2]{#1}
\newcommand{\usuk}[2]{{#1}}
\renewcommand{\minimisefunction}{\text{\textup{minimi\usuk{z}{s}e}}}
\renewcommand{\maximisefunction}{\text{\textup{maximi\usuk{z}{s}e}}}
\newcommand{\showplots}[2]{#1}

\newcommand{\explanation}[1]{}

\newcommand{\extendednotation}[2]{#1}

\newcommand{\dualitygapresidual}[2]{#1}

\newcommand{\removemapW}[2]{#1}

\newcommand{\simplifySB}[2]{#1}

\newcommand{\locusfunction}{Z}
\newcommand{\locus}[1]{\locusfunction({#1})}

\newcommand{\functionalsetname}{F}
\newcommand{\makeparallel}[1]{\bar{#1}}
\newcommand{\makeall}[1]{{#1}}

\newcommand{\CCDfunction}{\makeparallel{\functionalsetname}}
\newcommand{\CCDallfunction}{\makeall{\functionalsetname}}
\newcommand{\LSCfunction}{\textnormal{Lsc}}
\renewcommand{\lipschitz}{l}
\newcommand{\CCLfunction}{\CCDfunction^L}
\renewcommand{\CCLfunction}{\CCDfunction^\lipschitz}
\newcommand{\CCLallfunction}{\bar\CCDfunction^L}
\renewcommand{\CCLallfunction}{\CCDallfunction^\lipschitz}

\newcommand{\LIPfunction}{\textnormal{Lip}}
\newcommand{\CONfunction}{\textnormal{Con}}
\renewcommand{\LIP}[1]{\LIPfunction({#1})}
\renewcommand{\CON}[1]{\CONfunction({#1})}
\newcommand{\CCD}[1]{\CCDfunction({#1})}
\newcommand{\CCDall}[1]{\CCDallfunction({#1})}
\renewcommand{\CCL}[1]{\CCLfunction({#1})}
\newcommand{\CCLall}[1]{\CCLallfunction({#1})}
\newcommand{\LSC}[1]{\LSCfunction({#1})}

\newcommand{\notxconv}{\notx^{\star}}

\newcommand{\thefunction}{g}
\renewcommand{\functionk}[1]{\thefunction^{#1}}
\renewcommand{\distrifunction}{P}

\newcommand{\smoothfunction}{\objectivefunction}
\newcommand{\Pfunction}{h}
\newcommand{\Qfunction}{q}
\newcommand{\Amatrix}{A}
\newcommand{\bvec}{b}
\newcommand{\descentalpha}{a}
\newcommand{\descentdirection}{d}

\newcommand{\smooth}[1]{\smoothfunction({#1})}
\newcommand{\smoothtwo}[1]{\smoothtwofunction({#1})}
\newcommand{\smoothtwofunction}{g}
\newcommand{\smoothkfunction}[1]{\smoothfunction^{#1}}
\newcommand{\smoothkx}[2]{\smoothkfunction{#1}({#2})}
\newcommand{\Px}[1]{\Pfunction({#1})}
\newcommand{\Pifunction}[1]{\Pfunction_{#1}}
\newcommand{\Pix}[2]{\Pifunction{#1}({#2})}
\newcommand{\Qx}[1]{\Qfunction({#1})}
\newcommand{\descentdirectioni}[1]{\descentdirection_{#1}}

\newcommand{\diracfunction}{1}
\newcommand{\dirac}[1]{\indicatorfunction_{#1}}
\newcommand{\diracx}[2]{\indicator{#1}({#2})}

\newcommand{\Stochi}[1]{\Stoch_{#1}}

\newcommand{\acceptednorm}[2]{\norm{#2}_{#1}}
\newcommand{\uniformnorm}[2]{\norm{#2}_{1,{#1}}}
\newcommand{\uniformsecondderivativenorm}[2]{\norm{#2}_{2,{#1}}}
\newcommand{\family}[2]{\{{#1}\}_{#2}}

\newcommand{\intXxFf}[4]{\int\nolimits_{#1} {#4} \,  {#3}(d {#2})}
\newcommand{\IntXxFf}[4]{\int_{#1} {#4} \,  {#3}(d {#2})}
\newcommand{\cardinal}[1]{\singlenorm{#1}}
\newcommand{\notS}{S}

\newcommand{\xxgen}{\xx}
\newcommand{\ijk}{q}

\newcommand{\mmuki}[2]{\mmu^{#1}_{#2}}
\newcommand{\llambbis}{\hat{\llamb}}
\newcommand{\llambbisk}[1]{\llambbis^{#1}}
\newcommand{\llambbisi}[1]{\llambbis_{#1}}
\newcommand{\llambbiski}[2]{\llambbisk{#1}_{#2}}
\newcommand{\mmubis}{\hat{\mmu}}
\newcommand{\mmubisk}[1]{\mmubis^{#1}}
\newcommand{\mmubisi}[1]{\mmubis_{#1}}
\newcommand{\mmubiski}[2]{\mmubisk{#1}_{#2}}

\newcommand{\stufffunction}{\theta}
\newcommand{\stuff}[1]{\stufffunction({#1})}
\newcommand{\maxradius}{\rho}
\newcommand{\neibradius}[1]{\setoflambdas^{\star}({#1})}

\newcommand{\drawi}{j}
\newcommand{\drawilast}{t}
\newcommand{\drawikl}[2]{\drawi^{#1}_{#2}}
\newcommand{\probadistrifunction}{\pi}
\newcommand{\probadistri}[1]{\probadistrifunction_{#1}}
\newcommand{\Probadistrifunction}{\Pi}
\newcommand{\Probadistri}[1]{\Probadistrifunction_{#1}}
\newcommand{\probadiag}{\Pi}
\newcommand{\genericfunctionk}{\hat \thefunction}
\newcommand{\genericfunctionkx}[1]{\genericfunctionk({#1})}
\renewcommand{\genericnormscale}{M}
\newcommand{\genericnormscaletwo}{Q}
\newcommand{\genericnormscaleij}[2]{\genericnormscale_{{#1}{#2}}}
\newcommand{\genericdiag}{D}
\newcommand{\genericdiagi}[1]{\genericdiag_{{#1}}}

\newcommand{\SWifunction}{I}
\newcommand{\SWi}[2]{\SWifunction({#1},{#2})}
\newcommand{\SWIfunction}{I}
\newcommand{\SWI}[2]{\SWIfunction({#1},{#2})}
\newcommand{\exSWifunction}{J}
\newcommand{\exSWi}[2]{\exSWifunction({#1},{#2})}
\newcommand{\exSWIfunction}{I}
\newcommand{\exSWI}[2]{\exSWIfunction({#1},{#2})}
\newcommand{\standbyfunction}{K}
\newcommand{\standby}[4]{\standbyfunction({#1},{#2},{#3},{#4})}
\newcommand{\standbyfx}[2]{\standbyfunction({#1},{#2})}
\newcommand{\standbyfxvw}[4]{\standbyfunction\plusvectfxvw{#1}{#2}{#3}{#4}}
\newcommand{\standbyifunction}[1]{\bar\standbyfunction_{#1}}
\newcommand{\standbyi}[5]{\standbyifunction{#1}({#2},{#3},{#4},{#5})}
\newcommand{\standbyifx}[3]{\standbyifunction{#1}({#2},{#3})}
\newcommand{\standbyifxvw}[5]{\standbyifunction{#1}\plusvectfxvw{#2}{#3}{#4}{#5}}
\newcommand{\plusstandbyfunction}{\makeplus{\SWifunction}}
\newcommand{\plusstandby}[4]{\plusstandbyfunction({#1},{#2},{#3},{#4})}
\newcommand{\plusstandbyfx}[2]{\plusstandbyfunction({#1},{#2})}
\newcommand{\plusstandbyfxvw}[4]{\plusstandbyfunction\plusvectfxvw{#1}{#2}{#3}{#4}}

\newcommand{\plusnotf}{\makeplus{\notf}}
\newcommand{\plusllamb}{\makeplus{\llamb}}

\newcommand{\descentset}{I}
\newcommand{\descentsetfx}[2]{\descentset({#1},{#2})}

\newcommand{\probabilityspace}[3]{({#1},{#2},{#3})}

\newcommand{\search}{\bar\notx}
\newcommand{\searchfxT}[3]{\search_{{#1},{#2},{#3}}}
\newcommand{\searchfxTa}[4]{\search_{{#1},{#2},{#3}}({#4})}
\newcommand{\searchfxa}[3]{\search({#3})}
\renewcommand{\mapAfx}[2]{\mapAfunction({#1},{#2})}

\newcommand{\mapJfx}[2]{\mapJfunction({#1},{#2})}

\obsolete{
\DontPrintSemicolon
\SetKwInput{Data}{Initial data}
\SetKwInput{Result}{Result}
\SetKw{Ret}{Return}
\newcommand{\Give}[1]{\Ret\  {#1}}
\SetKw{Wait}{Standby}
\SetKw{Observe}{Observe}
\SetKw{Draw}{Draw }
\SetKwFor{While}{until}{do}{endw}
\SetKwRepeat{newRepeat}{Repeat}{until}
\SetKwProg{cyclicFor}{For}{ do successively}{end}
\newcommand{\inlineIfElse}[3]{\lIf{#1}{{#2}\ \lElse{#3}}}
\SetKwIF{Def}{ElseIf}{Else}{Consider}{as in}{else if}{else}{endif}
\SetKwProg{Funct}{Function}{:}{}
\newcommand{\Function}[2]{\Funct{#1}{#2}}
\newcommand{\Define}[2]{\lDef{#1}{#2}}
\newcommand{\routinefont}[1]{\FuncSty{#1}}
\newcommand{\parfont}[1]{\FuncSty{#1}}
\newcommand{\stoppingcriterionfunction}{\routinefont{stopping\_criterion}}
\newcommand{\stoppingcriterion}[3]{\stoppingcriterionfunction\parfont{(}{#1},{#2},{#3}\parfont{)}}
\newcommand{\dualstoppingcriterionfunction}{\routinefont{dual\_stopping\_criterion}}
\newcommand{\dualstoppingcriterion}[4]{\dualstoppingcriterionfunction\parfont{(}{#1},{#2},{#3},{#4}\parfont{)}}
\newcommand{\effectivemappingfunction}{\routinefont{conditional\_mapping}}
\newcommand{\effectivemapping}[4]{\effectivemappingfunction\parfont{(}{#1},{#2},{#3},{#4}\parfont{)}}
\newcommand{\standbycriterionfunction}{\routinefont{standby\_criterion}}
\newcommand{\standbycriterion}[3]{\standbycriterionfunction\parfont{(}{#1},{#2},{#3}\parfont{)}}
\newcommand{\dualstandbycriterionfunction}{\routinefont{dual\_standby\_criterion}}
\newcommand{\dualstandbycriterion}[4]{\dualstandbycriterionfunction\parfont{(}{#1},{#2},{#3},{#4}\parfont{)}}
\SetKwFor{While}{Until}{do}{endw}
\newcommand{\makedistribution}[1]{distribution({#1})}
\newcommand{\learn}[2]{model({#1},{#2})}
\newcommand{\duallearn}[3]{dual\_model({#1},{#2},{#3})}
\newcommand{\observe}[1]{\Observe{\textup{:}}\ {#1}}
}

\newcommand{\Conjunction}[2]{\bigwedge_{#1}{#2}}
\newcommand{\Conjunctionnolimits}[2]{\bigwedge\nolimits_{#1}{#2}}

\newcommand{\iset}{\imath}
\newcommand{\iseti}[1]{\iset_{#1}}
\newcommand{\isetk}[1]{\iset^{#1}}
\newcommand{\jset}{\jmath}
\newcommand{\jseti}[1]{\jset_{#1}}
\newcommand{\jsetk}[1]{\jset^{#1}}
\newcommand{\setofnodesets}{\mathcal{N}}
\newcommand{\mapi}[1]{\mapfunction_{#1}}
\newcommand{\localmapifunction}[1]{{\mapfunction}_{#1}}
\newcommand{\localmapifx}[3]{\localmapifunction{#1}({#2},{#3})}

\newcommand{\localmapGifunction}[1]{{\mapCfunction}^{\mapGfunction}_{#1}}
\newcommand{\localmapGifx}[3]{\localmapGifunction{#1}({#2},{#3})}
\newcommand{\mapGCfunction}{\mapCfunction^{\mapGfunction}}
\newcommand{\mapGC}[2]{\mapGCfunction({#1},{#2})}
\newcommand{\mapGSBfunction}{\mapSBfunction^{\mapGfunction}}
\newcommand{\mapGSB}[2]{\mapGSBfunction({#1},{#2})}
\newcommand{\mapGSBfxvw}[4]{\mapGSBfunction\plusvectfxvw{#1}{#2}{#3}{#4}}
\newcommand{\pluslocalmapGifunction}[1]{\plusmapGfunction_{#1}}
\newcommand{\pluslocalmapGifx}[3]{\pluslocalmapGifunction{#1}({#2},{#3})}
\newcommand{\pluslocalmapGifxvw}[5]{\pluslocalmapGifunction{#1}\plusvectfxvw{#2}{#3}{#4}{#5}}
\newcommand{\plusmapGfunction}{\tilde{\mapCfunction}^{\mapGfunction}}
\newcommand{\plusmapG}[4]{\plusmapGfunction({#1},{#2},{#3},{#4})}
\newcommand{\plusmapGfxvw}[4]{\plusmapGfunction\plusvectfxvw{#1}{#2}{#3}{#4}}

\newcommand{\mapifx}[3]{\mapi{#1}({#2},{#3})}

\newcommand{\solset}{S}
\newcommand{\solsetf}[1]{\solset^{#1}}
\newcommand{\solsetfX}[2]{\solsetf{#1}({#2})}

\newcommand{\local}[1]{\dot{#1}}
\newcommand{\localprimal}[1]{\ddot{#1}}
\newcommand{\subspace}{S}
\newcommand{\localsubspacei}[1]{\local{\subspace}_{#1}}
\newcommand{\localproji}[1]{{P}_{#1}}
\newcommand{\localNNi}[1]{\localprimal{\setofneighbours}_{#1}}
\newcommand{\localxx}{\localprimal{\xx}}
\newcommand{\localstochxfunction}{\localprimal\xx}
\newcommand{\localstochx}[1]{\localstochxfunction({#1})}
\newcommand{\localstochobjectivefunction}{\localprimal{\stochobjectivefunction}}
\newcommand{\localstochconstraintfunction}{\localprimal{\stochconstraintfunction}}
\newcommand{\localstochconstraintifunction}[1]{\localstochconstraintfunction_{#1}}
\newcommand{\localstochconstrainti}[3]{\localstochconstraintifunction{#1}({#2},{#3})}
\newcommand{\localstochconstraintineqfunction}{\localprimal{\stochconstraintineqfunction}}
\newcommand{\localstochconstraintineqifunction}[1]{\localstochconstraintineqfunction_{#1}}
\newcommand{\localstochconstraintineqi}[3]{\localstochconstraintineqifunction{#1}({#2},{#3})}
\newcommand{\localstochconstrainteqfunction}{\localprimal{\stochconstrainteqfunction}}
\newcommand{\localstochconstrainteqifunction}[1]{\localstochconstrainteqfunction_{#1}}
\newcommand{\localstochconstrainteqi}[3]{\localstochconstrainteqifunction{#1}({#2},{#3})}
\newcommand{\localstochobjectiveillambfunction}[2]{\localstochobjectivefunction_{#1,#2}}
\newcommand{\localstochobjectiveillamb}[4]{\localstochobjectiveillambfunction{#1}{#2}({#3},{#4})}
\newcommand{\localthefunction}{\local{\thefunction}}
\newcommand{\localfunctionk}[1]{\localthefunction^{#1}}
\newcommand{\localfunctionillamb}[2]{\localthefunction_{#1,#2}}
\newcommand{\localfunctionkillamb}[3]{\localfunctionk{#1}_{#2,#3}}
\newcommand{\localfunctionkillambmmu}[4]{\localfunctionkillamb{#1}{#2}{#3}({#4})}
\newcommand{\noth}{h}
\newcommand{\nothx}[1]{\noth({#1})}
\newcommand{\localnoth}{\local{\noth}}
\newcommand{\localnothix}[2]{\localnoth_{#1,#2}}
\newcommand{\localnothixy}[3]{\localnothix{#1}{#2}({#3})}

\newcommand{\maplocal}[1]{#1}
\newcommand{\maplocalix}[3]{\maplocal{#1}_{{#2}:{#3}}}
\newcommand{\maplocalixy}[4]{\maplocalix{#1}{#2}{#3}({#4})}
\newcommand{\localWdescentfunction}[1]{\local{\Delta}^{#1}}
\newcommand{\localWdescentfunctionix}[3]{\localWdescentfunction{#1}_{{#2},{#3}}}

\newcommand{\qk}[1]{q({#1})}
\newcommand{\stochkq}[2]{\stochk{{#1},{#2}}}

\renewcommand{\stochgfunction}{\hat \thefunction}

\newcommand{\genericseti}[1]{\genericset_{#1}}

\newcommand{\notq}{\tilde{\notx}}
\newcommand{\notqi}[1]{\notq_{#1}}

\newcommand{\deltasymbol}{s}
\newcommand{\covarsymbol}{\Sigma}
\newcommand{\fullcovarsymbol}{\Gamma}
\newcommand{\deltafunction}{\deltasymbol}
\newcommand{\deltaf}[1]{\deltafunction^{#1}}
\newcommand{\deltaifunction}[1]{\deltafunction_{#1}}
\newcommand{\deltafx}[2]{\deltaf{#1}({#2})}
\newcommand{\deltafi}[2]{\deltaf{#1}_{#2}}
\newcommand{\deltafix}[3]{\deltafi{#1}{#2}({#3})}
\newcommand{\stddeltaf}[1]{\sigma_{\deltaf{#1}}}
\newcommand{\vardeltaf}[1]{\sigma^2_{\deltaf{#1}}}
\newcommand{\covardeltaf}[1]{\covarsymbol_{\deltaf{#1}}}
\newcommand{\stddeltafx}[2]{\stddeltaf{#1}({#2})}
\newcommand{\vardeltafx}[2]{\vardeltaf{#1}({#2})}
\newcommand{\covardeltafx}[2]{\covardeltaf{#1}({#2})}
\newcommand{\estimdeltafk}[2]{\hat\deltasymbol^{#2}}
\newcommand{\estimdeltafkx}[3]{\estimdeltafk{#1}{#2}({#3})}
\newcommand{\estimdeltafki}[3]{\estimdeltafk{#1}{#2}_{#3}}
\newcommand{\estimdeltafkix}[4]{\estimdeltafki{#1}{#2}{#3}({#4})}
\newcommand{\estimstddeltafk}[2]{\hat\sigma^{#2}}
\newcommand{\estimvardeltafk}[2]{(\hat\sigma^{#2})^2}
\newcommand{\estimcovardeltafk}[2]{\hat\covarsymbol^{#2}}
\newcommand{\estimstddeltafkx}[3]{\estimstddeltafk{#1}{#2}({#3})}
\newcommand{\estimvardeltafkx}[3]{(\estimstddeltafk{#1}{#2}({#3}))^2}
\newcommand{\estimcovardeltafkx}[3]{\estimcovardeltafk{#1}{#2}({#3})}
\newcommand{\estimstddeltafki}[3]{\estimstddeltafk{#1}{#2}_{#3}}
\newcommand{\estimvardeltafki}[3]{(\estimstddeltafk{#1}{#2}_{#3})^2}
\newcommand{\estimcovardeltafki}[3]{\estimcovardeltafk{#1}{#2}_{#3}}
\newcommand{\estimstddeltafkix}[4]{\estimstddeltafki{#1}{#2}{#3}({#4})}
\newcommand{\estimvardeltafkix}[4]{(\estimstddeltafki{#1}{#2}{#3}({#4}))^2}
\newcommand{\estimcovardeltafkix}[4]{\estimcovardeltafki{#1}{#2}{#3}({#4})}
\newcommand{\truecovarfunction}{\covarsymbol}
\newcommand{\truecovarifunction}[1]{\truecovarfunction_{#1}}
\newcommand{\truecovar}[2]{\truecovarfunction({#2})}
\newcommand{\truecovari}[3]{\truecovarifunction{#1}({#3})}
\newcommand{\estimcovarfunction}{\hat\truecovarfunction}
\newcommand{\estimcovarifunction}[1]{\estimcovarfunction_{#1}}
\newcommand{\estimcovar}[2]{\estimcovarfunction({#2})}
\newcommand{\estimcovari}[3]{\estimcovarifunction{#1}({#3})}
\newcommand{\truefullcovarfunction}{\fullcovarsymbol}
\newcommand{\truefullcovarifunction}[1]{\truefullcovarfunction_{#1}}
\newcommand{\truefullcovari}[3]{\truefullcovarifunction{#1}({#3})}
\newcommand{\estimfullcovardeltafk}[2]{\hat\fullcovarsymbol^{#2}}
\newcommand{\estimfullcovardeltafkx}[3]{\estimfullcovardeltafk{#1}{#2}({#3})}
\newcommand{\estimfullcovardeltafki}[3]{\estimfullcovardeltafk{#1}{#2}_{#3}}
\newcommand{\estimfullcovardeltafkix}[4]{\estimfullcovardeltafki{#1}{#2}{#3}({#4})}

\newcommand{\rank}[1]{\textup{rk}({#1})}
\newcommand{\Rank}[1]{\textup{rk}\left({#1}\right)}
\newcommand{\pinv}{\dagger}

\newcommand{\decreasingfunction}{d}
\newcommand{\decreasing}[1]{\decreasingfunction({#1})}

 \newcommand{\individualfunction}{h}
\newcommand{\individualklfunction}[2]{\individualfunction^{{#1},{#2}}}
\newcommand{\individualklifunction}[3]{\individualklfunction{#1}{#2}_{#3}}
\newcommand{\individualkli}[4]{\individualklifunction{#1}{#2}{#3}({#4})}
\newcommand{\dualitygap}{\delta}
\newcommand{\dualitygapstd}{\sigma_{\dualitygap}}
\newcommand{\estimdualitygap}{\hat\delta}
\newcommand{\estimdualitygapstd}{\hat\sigma_{\dualitygap}}
\newcommand{\Variance}[1]{\textup{Var\left({#1}\right)}}
\newcommand{\variance}[1]{\textup{Var}({#1})}

\newcommand{\xsolllambfunction}[1]{\xx^{\star}}
\newcommand{\xsolllamb}[2]{\xsolllambfunction{#1}({#2})}

\newcommand{\indexneighbour}[2]{{#1}_{#2}}
\newcommand{\setofneighboursplus}{\dot\setofneighbours}
\renewcommand{\Ni}[1]{\indexneighbour{\setofneighbours}{#1}}
\newcommand{\Nplusi}[1]{\indexneighbour{\setofneighboursplus}{#1}}

\newcommand{\notproba}{\pi}
\newcommand{\notfreedom}{d}
\newcommand{\decisionthreshold}{\beta}
\newcommand{\decisionthresholdr}[1]{\decisionthreshold_{#1}}
\newcommand{\decisionthresholdrp}[2]{\decisionthresholdr{#1}({#2})}
\newcommand{\true}{\emph{\text{true}}}
\newcommand{\false}{\emph{\text{false}}}
\newcommand{\functionto}{\mapsto}


\newcommand{\shortlong}[2]{#1}
\renewcommand{\doteq}{:=}

\newcommand{\llambsol}{\bar\llamb}

\newtheorem{condition}{Condition}
\newcommand{\Statement}{Result{}}
\newtheorem{fact}{\Statement}

{\theoremstyle{plain}
\newtheorem{algo}{Algorithm}}

\newcommand{\reducedball}{B}
\newcommand{\reducedzz}{\tilde \newzz}
\newcommand{\reducedyy}{\tilde \newyy}

 \newcommand{\emproblem}{\em}
\newcommand{\emdefinition}{\em}

\newcommand{\removequeue}[2]{#2}
\newcommand{\removerobust}[2]{#1}
\newcommand{\removeothermethods}[2]{#2}

\newcommand{\pointtoset}[2]{#1}
\newcommand{\thinpm}{  \!  \pm  \! }
\newcommand{\thintimes}{  \!  \times  \! }
\renewcommand{\mpmean}[1]{{\messagepassingfunction}({#1})}


\newcommand{\gls}[1]{\emph{#1}}
\newcommand{\Gls}[1]{\emph{#1}}
\newcommand{\glspl}[1]{\emph{#1}s}
\newcommand{\glslink}[2]{{#2}}
\newcommand{\Glspl}[1]{\emph{#1}s}

\renewcommand{\compact}{C}

\newcommand{\changeZZ}[2]{#1}
\newcommand{\nondifferentiable}[2]{#1}
\newcommand{\differentiable}[2]{#2}

\renewcommand{\llambconv}{\bar \llamb}
\newcommand{\stochxconvfunction}{\bar \xx}
\newcommand{\stochxconvifunction}[1]{\stochxconvfunction_{#1}}
\newcommand{\stochxconv}[1]{\stochxconvfunction({#1})}
\newcommand{\stochxconvi}[2]{\stochxconvifunction{#1}({#2})}

\renewcommand{\buff}{\hat{\mmu}}
\renewcommand{\buffj}[1]{\buff^{{#1}}}
\renewcommand{\buffji}[2]{\buff_{#2}^{{#1}}}

\renewcommand{\nn}{n}
\renewcommand{\NN}{p}
\renewcommand{\MM}{m}
\renewcommand{\setofnodes}{N}
\renewcommand{\setofneighbours}{\setofnodes}

\renewcommand{\alphamapping}{\hat a}
\renewcommand{\mapAfunction}{\alphamapping}
\renewcommand{\projectedllamb}{\hat \llamb}
\renewcommand{\violatingsetsymbol}{\mathcal{A}}
\renewcommand{\nonviolatingsymbol}{E}
\renewcommand{\violatingsymbol}{\bar \nonviolatingsymbol}
\renewcommand{\Acorthogbase}{F}
\renewcommand{\LCostrowD}[1]{\hat D}
\renewcommand{\LCostrowE}[1]{\hat E}

\renewcommand{\LClambconv}{0}

\newcommand{\setofScalesn}[1]{\mathcal{\Scale}({#1})}
\renewcommand{\setofScales}{\setofScalesn{\MM}}
\renewcommand{\setofScalesi}[1]{\setofScalesn{\MMi{#1}}}

\renewcommand{\freelagrangianarg}[3]{({#1},{#3})}
\renewcommand{\freexsol}[2]{\freexsolfunction({#2})}
\renewcommand{\freexsoli}[3]{\freexsolfunction({#3}) }

\renewcommand{\Rconst}{\incidence}

\renewcommand{\gradij}[2]{(\grad_{\! {#1}})_{#2}}

\renewcommand{\mmudy}[2]{\lamb_i ({#1},{#2})}



\newcommand*{\QEDA}{\hfill\ensuremath{\square}}
\newcommand{\titleenv}[1]{\emph{[{#1}]}}
\newcommand{\defemph}{\em{}}
\newcommand{\probemph}{\em{}}

\newcommand{\jotanew}[2]{#1}


\renewcommand{\doPROJ}[1]{\tilde{#1}}
\renewcommand{\nonPROJnu}{\xx}
\renewcommand{\subderiv}{s}
\renewcommand{\dgnom}{\subderiv}
\renewcommand{\dnom}{\nu}
\renewcommand{\snom}{\sigma}
\renewcommand{\dgv}{\bar \dgnom}
\renewcommand{\dgnv}{\dgnom}
\renewcommand{\dv}{\bar \dnom}
\renewcommand{\dnv}{\dnom}
\renewcommand{\sv}{\bar \snom}
\renewcommand{\snv}{\snom}

\newcommand{\newNN}{\NN}
\newcommand{\newxx}{\xx}
\newcommand{\newyy}{\yy}
\newcommand{\newzz}{\zz}
\newcommand{\newnotf}{\notf}
\newcommand{\newnotfx}[1]{\newnotf({#1})}
\newcommand{\newgenericset}{\genericset}
\newcommand{\newxxdest}{\newzz}
\newcommand{\newzzk}[1]{\hat\newzz[{#1}]}
\newcommand{\newScale}{D}
\newcommand{\newScalefunction}{\Scale}
\newcommand{\newScalefx}[2]{\newScalefunction({#1},{#2})}
\renewcommand{\reducedspace}{Z}

\renewcommand{\Scaleki}[2]{\Scalek{#1}_{#2}}


\renewcommand{\zz}{z}

\renewcommand{\constraintineqfunction}{d}
\renewcommand{\constrainteqfunction}{h}
\newcommand{\constraintineqlocalfunction}{\zeta}
\renewcommand{\constraintineqlocalijfunction}[2]{\constraintineqlocalfunction_{{#1}{#2}}}

\renewcommand{\convexcave}[2]{#1}
\renewcommand{\Qset}{\sideconstraintset}


\newcommand{\Wrho}{\Zrho}
\newcommand{\Wd}{d}
\newcommand{\Weps}{\epsilon}
\newcommand{\Wk}{\Zk}
\newcommand{\Wtheta}{\theta}
\newcommand{\Wthetax}[1]{\Wtheta({#1})}
\newcommand{\Wdescentfunctional}{\Delta}
\newcommand{\Wdescentfunction}[1]{\Wdescentfunctional^{#1}}
\newcommand{\Wdescent}[2]{\Wdescentfunction{#1}({#2})}
\newcommand{\Wn}{\bar{n}}
\newcommand{\Wni}[1]{\Wn({#1})}

\newcommand{\mapifunction}[1]{\mapfunction_{#1}}

\newcommand{\makeplus}[1]{\tilde{#1}}

\newcommand{\plusfunction}{\makeplus{\thefunction}}
\newcommand{\plusfunctionxw}[2]{\plusfunction({#1},{#2})}
\newcommand{\plusWdescentfunction}[1]{\makeplus{\Wdescentfunctional}^{#1}}
\newcommand{\plusWdescent}[3]{\plusWdescentfunction{#1}({#2},{#3})}
\newcommand{\plusCCDfunction}{\makeplus{\functionalsetname}}
\newcommand{\plusCCD}[1]{\plusCCDfunction({#1})}

\newcommand{\plusv}{v}
\newcommand{\plusvmax}{\bar{\plusv}}
\newcommand{\plusvk}[1]{\plusv^{#1}}
\newcommand{\plusvki}[2]{\plusvk{#1}_{#2}}
\newcommand{\plusvx}[1]{\plusv({#1})}
\newcommand{\plusvi}[1]{\plusv_{#1}}
\newcommand{\plusvkx}[2]{\plusvk{#1}({#2})}
\newcommand{\plusvix}[2]{\plusvi{#1}({#2})}
\newcommand{\plusvkix}[3]{\plusvki{#1}{#2}({#3})}
\newcommand{\plusVfunction}{V}
\newcommand{\plusV}[1]{\plusVfunction({#1})}
\newcommand{\plusVifunction}[1]{\plusVfunction_{#1}}
\newcommand{\plusVi}[2]{\plusVifunction{#1}({#2})}
\newcommand{\maxplusv}{\bar\plusv}
\newcommand{\plusvnorm}[2]{\norm{{#2}}_{#1}^{\plusVfunction}}
\newcommand{\plusw}{w}
\newcommand{\pluswk}[1]{\plusw^{#1}}
\newcommand{\pluswbis}{u}
\newcommand{\pluswbisk}[1]{\pluswbis^{#1}}
\newcommand{\plusW}{W}

\newcommand{\pluslocalmapifunction}[1]{\plusmapfunction_{#1}}
\newcommand{\pluslocalmapifx}[3]{\pluslocalmapifunction{#1}({#2},{#3})}
\newcommand{\pluslocalmapifxvw}[5]{\pluslocalmapifunction{#1}\plusvectfxvw{#2}{#3}{#4}{#5}}

\newcommand{\plusmapfunction}{\tilde{\mapfunction}}
\newcommand{\plusmap}[4]{\plusmapfunction({#1},{#2},{#3},{#4})}
\newcommand{\plusmapfxvw}[4]{\plusmapfunction\plusvectfxvw{#1}{#2}{#3}{#4}}
\newcommand{\plusvectfxvw}[4]{\vect{\vect{{#1},{#3}},\vect{{#2},{#4}}}}
\newcommand{\plusmapfx}[2]{\plusmapfunction({#1},{#2})}
\newcommand{\mapSWfunction}{\mathcal{S}}
\newcommand{\mapSW}[2]{\mapSWfunction({#1},{#2})}
\newcommand{\mapSBfunction}{\mathcal{K}}
\newcommand{\mapSB}[2]{\mapSBfunction({#1},{#2})}
\newcommand{\mapSBfx}[2]{\mapSBfunction({#1},{#2})}
\newcommand{\mapSBfxvw}[4]{\mapSBfunction\plusvectfxvw{#1}{#2}{#3}{#4}}
\newcommand{\mapVfunction}{\mathcal{V}}
\newcommand{\mapV}[2]{\mapVfunction({#1},{#2})}
\newcommand{\mapJfunction}{\mathcal{J}}
\newcommand{\mapJ}[2]{\mapJfunction({#1},{#2})}
\newcommand{\mapCfunction}{\mathcal{C}}
\newcommand{\mapC}[2]{\mapCfunction({#1},{#2})}
\newcommand{\mapNfunction}{\mathcal{N}}
\newcommand{\mapN}[2]{\mapNfunction({#1},{#2})}
\newcommand{\mapNifunction}[1]{\mapNfunction_{#1}}
\newcommand{\mapNi}[3]{\mapNifunction{#1}({#2},{#3})}
\newcommand{\mapRfunction}{\mathcal{R}}
\newcommand{\mapRkfunction}[1]{\mapRfunction^{#1}}
\newcommand{\mapRkfx}[3]{\mapRkfunction{#1}({#2},{#3})}
\newcommand{\mapS}[2]{\mapSfunction({#1},{#2})}
\newcommand{\mapcombifunction}{\mathcal{C}}
\newcommand{\mapcombicfunction}[1]{\mapcombifunction^{#1}}
\newcommand{\mapcombic}[3]{\mapcombicfunction{#1}({#2},{#3})}
\newcommand{\combi}{c}
\newcommand{\Combi}{R}
\newcommand{\CombiS}[1]{\Combi({#1})}
\newcommand{\Seq}{S}
\newcommand{\SeqS}[1]{\Seq({#1})}
\newcommand{\setvary}{J}
\newcommand{\setvarynot}{I}
\newcommand{\notzi}[1]{\notz_{#1}}

\newcommand{\criterion}{\theta}
\newcommand{\criterionk}[1]{\criterion^{#1}}
\newcommand{\criterionkfx}[3]{\criterionk{#1}({#2},{#3})}
\newcommand{\criterionki}[2]{\criterionk{#1}_{#2}}
\newcommand{\criterionkifx}[4]{\criterionki{#1}{#2}({#3},{#4})}

\newcommand{\subkfunction}{\kappa}
\newcommand{\subk}[1]{\subkfunction({#1})}

\newcommand{\lowerbound}{l}
\newcommand{\upperbound}{u}

\newcommand{\genericsolset}{\solset}
\newcommand{\ballfunction}{B}
\newcommand{\ballX}[1]{{#1}}
\newcommand{\ballXeps}[2]{\ballX{#1}^{#2}}
\newcommand{\ballXepsx}[3]{\ballXeps{#1}{#2}({#3})}
\newcommand{\HYPfunction}{H}
\newcommand{\HYP}[2]{\HYPfunction({#1},{#2})}
\newcommand{\polarHYPfunction}{\HYPfunction^*}
\newcommand{\polarHYPifunction}[1]{\polarHYPfunction_{#1}}
\newcommand{\polarHYP}[2]{\polarHYPfunction({#1},{#2})}
\newcommand{\polarHYPi}[3]{\polarHYPifunction{#3}({#1},{#2})}
\newcommand{\projfunction}[1]{\textup{P}^{\text{\tiny{\ensuremath{\perp}}}}_{#1}}
\renewcommand{\projx}[2]{\projfunction{#1}({#2})}
\newcommand{\jacobiansymbol}{\textup{J}}
\renewcommand{\jacobianfunction}[1]{\jacobiansymbol{#1}}
\renewcommand{\jacobian}[2]{\jacobianfunction{#1}({#2})}

\newcommand{\dualitygapfunction}{\delta}
\newcommand{\dualitygapx}[1]{\dualitygapfunction({#1})}



\newcommand{\setofinterest}{\bar{\setoflambdas}}
\newcommand{\plussetofinterest}{\makeplus{\setoflambdas}}
\newcommand{\setofinteresti}[1]{\setofinterest_{#1}}

\renewcommand{\iset}{I}
\renewcommand{\jset}{J}
\renewcommand{\SWifunction}{K}
\renewcommand{\exSWifunction}{L}
\renewcommand{\standbyfunction}{Z}

\renewcommand{\mapSWfunction}{\mathcal{K}}
\renewcommand{\mapSBfunction}{\mathcal{Z}}
\newcommand{\notn}{p}
\newcommand{\notni}[1]{\notn_{#1}}
\renewcommand{\genericn}{\notn}

\renewcommand{\deltafx}[2]{\deltaf{#1}({#2})}
\renewcommand{\deltafi}[2]{\deltaf{#1}_{#2}}
\renewcommand{\deltafix}[3]{\deltafi{#1}{#2}({#3})}

\newcommand{\incidencestochfunction}{\hat\incidence}
\renewcommand{\incidencestoch}[1]{\incidencestochfunction({#1})}
\newcommand{\incidenceistochfunction}[1]{\incidencestochfunction_{#1}}
\newcommand{\incidenceijstochfunction}[2]{\incidencestochfunction_{{#1}{#2}}}
\renewcommand{\incidenceistoch}[2]{\incidenceistochfunction{#1}({#2})}
\renewcommand{\incidenceijstoch}[3]{\incidenceijstochfunction{#1}{#2}({#3})}

\newcommand{\stochcapafunction}{\hat{\kappa}}
\newcommand{\stochcapaijfunction}[2]{\stochcapafunction_{{#1}{#2}}}
\newcommand{\stochcapaij}[4]{\stochcapaijfunction{#1}{#2}({#3},{#4})}
\renewcommand{\positive}[1]{#1^{+}}

\newcommand{\nbtwohop}[1]{\bar\nn_{i}}
\newcommand{\average}{\mu}
\newcommand{\averagei}[1]{\average_{#1}}
\newcommand{\estimaveragei}[1]{\hat\average_{#1}}
\newcommand{\estimaverageik}[2]{{\estimaveragei{#1}^{#2}}}

\newcommand{\cumuldescentsfunction}{\varphi}
\newcommand{\cumuldescents}[1]{\cumuldescentsfunction({#1})}
\newcommand{\cumuldescentstepsfunction}{\bar\tableaufunction}
\newcommand{\cumuldescentsteps}[1]{\cumuldescentstepsfunction({#1})}
\newcommand{\SAaveraging}{SA}

\newcommand{\SA}[2]{#1}

\newcommand{\moveproof}[2]{#1}

\renewcommand{\setofneighboursplus}{\setofneighbours}

\renewcommand{\expectationsymbol}{\textup{E}}

\newcommand{\Var}[1]{\textup{Var}\left[{#1}\right]}
\newcommand{\Varsmall}[1]{\textup{Var}[{#1}]}


\newcommand{\zeroP}[1]{o_P({#1})}
\newcommand{\magnitudeP}[1]{O_P({#1})}
\newcommand{\distrito}{\xrightarrow{d}}
\newcommand{\asto}{\xrightarrow{a.s.}}
\newcommand{\dir}{\delta}
\newcommand{\dirx}[1]{\dir({#1})}
\newcommand{\variation}{h}
\newcommand{\LCvariation}{\tilde\variation}
\newcommand{\secondderiv}{\bar{\variation}}
\newcommand{\secondderivdir}[1]{\secondderiv({#1})}
\newcommand{\NRV}{\nu}
\newcommand{\NRVx}[1]{\NRV({#1})}
\newcommand{\accu}{\mmu}
\newcommand{\accuk}[1]{\accu^{#1}}
\newcommand{\LCaccu}{\tilde\mmu}
\newcommand{\LCaccuk}[1]{\tilde\accu^{#1}}
\renewcommand{\LCMk}[2]{\LCMfunction(\functionk{#1},{#2})}
\newcommand{\LCMinf}[1]{\LCMfunction(\function,{#1})}
\newcommand{\remainderfunction}{\varrho}
\newcommand{\fullremainderfunction}{\rho}
\newcommand{\remainder}[1]{\remainderfunction(\function,{#1})}
\newcommand{\fullremainder}[1]{\fullremainderfunction(\function,{#1})}
\newcommand{\remainderk}[2]{\remainderfunction(\functionk{#1},{#2})}
\newcommand{\fullremainderk}[2]{\fullremainderfunction(\functionk{#1},{#2})}
\newcommand{\displacementfunction}{d}
\newcommand{\displacement}[1]{\displacementfunction({#1})}
\renewcommand{\matrixnonviolatingx}[1]{\nonviolatingsymbol}
\newcommand{\varstochgfunction}{\sigma^2}
\newcommand{\varstochg}[1]{\varstochgfunction({#1})}
\newcommand{\markovA}{\tilde{A}}
\newcommand{\markovB}{\tilde{B}}
\renewcommand{\markovA}{{A}}
\renewcommand{\markovB}{{B}}
\newcommand{\markovAk}[1]{\markovA^{#1}}
\newcommand{\markovBk}[1]{\markovB^{#1}}
\newcommand{\criticalcone}{C}
\newcommand{\critical}[1]{\criticalcone({#1})}


\newcommand{\affinea}{a}
\newcommand{\affineb}{b}
\newcommand{\affineaj}[1]{\affinea_{#1}}
\newcommand{\affinebj}[1]{\affineb_{#1}}
\newcommand{\affineai}[1]{\affinea}
\newcommand{\affinebi}[1]{\affineb}
\newcommand{\affineaij}[2]{\affineai{#1}_{#2}}
\newcommand{\affinebij}[2]{\affinebi{#1}_{#2}}
\newcommand{\affinen}{p}
\newcommand{\affineni}[1]{\affinen_{#1}}
\newcommand{\aaffinea}{\bar\affinea}
\newcommand{\aaffineb}{\bar\affineb}
\newcommand{\aaffineaj}[1]{\aaffinea_{#1}}
\newcommand{\aaffinebj}[1]{\aaffineb_{#1}}
\newcommand{\aaffineai}[1]{\aaffinea}
\newcommand{\aaffinebi}[1]{\aaffineb}
\newcommand{\aaffineaij}[2]{\aaffineai{#1}_{#2}}
\newcommand{\aaffinebij}[2]{\aaffinebi{#1}_{#2}}
\newcommand{\aaffinen}{c}
\newcommand{\aaffineni}[1]{\aaffinen_{#1}}
\newcommand{\activeaffine}{\mathcal{A}}
\newcommand{\activeaffinei}[1]{\activeaffine_{#1}}
\newcommand{\activeaffineix}[2]{\activeaffinei{#1}({#2})}
\newcommand{\activeaffinex}[1]{\activeaffine({#1})}
\newcommand{\activeaffinespecial}{\mathcal{B}}
\newcommand{\LCScalefx}[2]{\LCScale({#1},{#2})}
\newcommand{\LCScaleifx}[3]{\LCScalei{#1}({#2},{#3})}
\renewcommand{\gradftilde}{\tilde \grad \smoothfunction}
\newcommand{\bijectionfx}[1]{\bijectionf({#1})}
\renewcommand{\bijectionf}{h}
\renewcommand{\matrixnonviolatingx}[1]{\nonviolatingsymbol(#1)}

\newcommand{\refeq}[2]{\overset{\text{\tiny{#1}}}{#2}}

\newcommand{\notA}{A}
\newcommand{\notAk}[1]{\notA^{#1}}

\renewcommand{\llamb}{\notx}
\renewcommand{\mmu}{\noty}
\renewcommand{\setoflambdas}{\genericset}
\newcommand{\GPfunction}{\function}
\newcommand{\GPfunctionx}[1]{\GPfunction({#1})}

\renewcommand{\llambconv}{\llamb^\star}
\newcommand{\nbhd}{\setoflambdas^{\star}}
\newcommand{\Paok}{\hat\Pa}
\newcommand{\Paokk}[1]{\Paok^{#1}}

\newcommand{\mapGTfunction}[1]{\mapGfunction^{#1}}
\newcommand{\mapGTfx}[3]{\mapGTfunction{#1}({#2},{#3})}
\newcommand{\mapGTXfunction}[2]{\mapGfunction^{(#1,#2)}}
\newcommand{\mapGTXfx}[4]{\mapGTXfunction{#1}{#2}({#3},{#4})}

\newcommand{\componentfunction}{\mathcal{C}}
\newcommand{\componentfx}[2]{\componentfunction({#1},{#2})}
\newcommand{\componentifunction}[1]{\componentfunction_{#1}}
\newcommand{\componentifx}[3]{\componentifunction{#1}({#2},{#3})}

\renewcommand{\LCmu}{\tilde \mmu}
\renewcommand{\LCMSfunction}{\LCMfunction^{\mapSfunction}}

\newcommand{\mmudyi}[3]{\lamb_{#1} ({#2},{#3})}
 \newcommand{\LCMfx}[2]{\LCMfunction({#1},{#2})}
 \newcommand{\LCMkfunction}[1]{\LCMfunction^{\lbrack{#1}\rbrack}}
 \newcommand{\LCMkfx}[3]{\LCMkfunction{#1}({#2},{#3})}
\newcommand{\switchfunction}{\psi}
\newcommand{\Switch}{\Psi}
\newcommand{\switch}[1]{\switchfunction({#1})}

\newcommand{\identity}{\identitysymbol}
\newcommand{\tildeidentity}{\tilde{\identitysymbol}}

\newcommand{\mapZfunction}{\mathcal{Z}}
\newcommand{\mapZqfunction}[1]{\mapZfunction^{[{#1}]}}
\newcommand{\mapZqfx}[3]{\mapZqfunction{#1}({#2},{#3})}
\renewcommand{\LCMfunction}{\tilde{S}}
\renewcommand{\LCMJfunction}{\tilde{J}}
 \newcommand{\LCMJfx}[2]{\LCMJfunction({#1},{#2})}
\newcommand{\LCMRfunction}{\tilde{R}}
 \newcommand{\LCMRfx}[2]{\LCMRfunction({#1},{#2})}
\renewcommand{\LCMqfunction}[1]{\tilde{Z}^{[{#1}]}}
\newcommand{\LCMqfx}[3]{\LCMqfunction{#1}({#2},{#3})}
\newcommand{\LCMGifunction}[1]{\tilde{G}_{{#1}}}
\newcommand{\LCMGifx}[3]{\LCMGifunction{#1}({#2},{#3})}
\newcommand{\LCMgenfunction}{\tilde{M}}
\newcommand{\LCMgenfx}[2]{\LCMgenfunction({#1},{#2})}

\renewcommand{\globalMqfunction}[1]{Z^{[{#1}]}}
\newcommand{\globalMqfx}[3]{\globalMqfunction{#1}({#2},{#3})}

\newcommand{\Scaletermfunction}{Q}
\newcommand{\Scaleterm}[2]{\Scaletermfunction({#1},{#2})}


\def\stackalignment{l}
\newcommand{\makenewD}[1]{%
\bottominset{\scalebox{0.7}[0.5]{{$\backslash$}}}{${#1}$}{}{-0.5pt}
}
\newcommand{\makenewL}[1]{%
\bottominset{\scalebox{0.7}[1.4]{${\llcorner}$}}{${#1}$}{}{-0.1pt}
}
\newcommand{\newDsymbol}{\makenewD{\nabla^2}}
\newcommand{\newtildeDsymbol}{\makenewD{\tilde\nabla^2}}
\newcommand{\newDfunction}[1]{\newDsymbol{#1}}
\newcommand{\newtildeDfunction}[1]{\newtildeDsymbol{#1}}
\newcommand{\newD}[2]{\newDfunction{#1}({#2})}
\newcommand{\newtildeD}[2]{\newtildeDfunction{#1}({#2})}

\newcommand{\newLsymbol}{\makenewL{\nabla^2}}
\newcommand{\newtildeLsymbol}{\makenewL{\tilde\nabla^2}}
\newcommand{\newLfunction}[1]{\newLsymbol{#1}}
\newcommand{\newtildeLfunction}[1]{\newtildeLsymbol{#1}}
\newcommand{\newL}[2]{\newLfunction{#1}({#2})}
\newcommand{\newtildeL}[2]{\newtildeLfunction{#1}({#2})}

\newcommand{\genericblock}{U}
\newcommand{\genericL}{V}
\newcommand{\genericD}{W}
\newcommand{\genericblockij}[2]{\genericblock_{{#1}{#2}}}

\newcommand{\LCMgenericfunction}{\tilde\sigma}
\newcommand{\LCMgeneric}[1]{\LCMgenericfunction({#1})}
\newcommand{\LCMRgenericfunction}{\tilde\rho}
\newcommand{\LCMRgeneric}[1]{\LCMRgenericfunction({#1})}

\newcommand{\makeconstantD}[1]{{#1}^{(\backslash)}}
\newcommand{\makeconstantL}[1]{{#1}^{(\llcorner)}}

\newcommand{\boundK}{U}
\newcommand{\minK}{\textup{\b{$u$}}}
\newcommand{\boundKD}{\makeconstantD{\boundK}}
\newcommand{\boundKL}{\makeconstantL{\boundK}}
\newcommand{\LCboundK}{\tilde{\boundK}}
\newcommand{\LCboundKD}{\makeconstantD{\LCboundK}}
\newcommand{\LCboundKL}{\makeconstantL{\LCboundK}}
\newcommand{\boundKij}[2]{\boundK_{{#1},{#2}}}
\newcommand{\LCboundKij}[2]{\LCboundK_{{#1},{#2}}}
\newcommand{\minboundK}{\textup{\b{$\boundK$}}}
\newcommand{\minKi}[1]{\minK_{#1}}
\newcommand{\newboundK}{\hat{\boundK}}
\newcommand{\newboundKi}[1]{\newboundK_{{#1}}}
\newcommand{\Lipschitz}{L}
\newcommand{\Lipschitzi}[1]{\Lipschitz_{#1}}
\newcommand{\boundL}{L}
\newcommand{\boundLD}{\makeconstantD{\boundL}}
\newcommand{\boundLL}{\makeconstantL{\boundL}}
\newcommand{\LCboundL}{\tilde{\boundL}}
\newcommand{\LCboundLD}{\makeconstantD{\LCboundL}}
\newcommand{\LCboundLL}{\makeconstantL{\LCboundL}}
\newcommand{\boundLij}[2]{\boundL_{{#1},{#2}}}
\newcommand{\LCboundLij}[2]{\LCboundL_{{#1},{#2}}}
\newcommand{\maxL}{\bar{l}}
\newcommand{\maxboundL}{\bar{\boundL}}
\newcommand{\maxLi}[1]{\maxL_{#1}}

\newcommand{\supermmu}{\hat\mmu}
\newcommand{\superllamb}{\hat\llamb}
\newcommand{\superalpha}{\alpha}

\newcommand{\boundKs}{U^{\text{{\tiny$(\smoothfunction)$}}}}
\newcommand{\minKs}{\textup{\b{$u$}}^{\text{{\tiny$(\smoothfunction)$}}}}
\newcommand{\boundKsij}[2]{\boundKs_{{#1},{#2}}}
\newcommand{\minboundKs}{\textup{\b{$\boundK$}}^{(\smoothfunction)}}
\newcommand{\minKsi}[1]{\minKs_{#1}}
\newcommand{\boundKc}{U^{\text{{\tiny$(\convexfunction)$}}}}
\newcommand{\minKc}{\textup{\b{$u$}}^{\text{{\tiny$(\convexfunction)$}}}}
\newcommand{\boundKcij}[2]{\boundKc_{{#1},{#2}}}
\newcommand{\minboundKc}{\textup{\b{$\boundK$}}^{(\smoothfunction)}}
\newcommand{\minKci}[1]{\minKc_{#1}}

\newcommand{\correctproba}[2]{#2}
\newcommand{\probamin}{\textup{\b{$\probadistrifunction$}}}
\newcommand{\coefineq}{\mu}

\renewcommand{\constantScale}{\hat{\Scale}}
\renewcommand{\constantScalei}[1]{\constantScale_{#1}}
\newcommand{\supernormfunction}{\Psi}
\newcommand{\supernorm}[1]{\supernormfunction({#1})}

\newcommand{\smartscale}{V} 
\newcommand{\smartscalek}[1]{\smartscale^{#1}}
\newcommand{\tildesmartscale}{\tilde{\smartscale}} 

\newcommand{\expectationcond}[2]{\expectation{\left.{#1}\right\rvert{#2}}}
\newcommand{\smallexpectationcond}[2]{\smallexpectation{{#1} | {#2}}}

\renewcommand{\accu}{\llambconv}
\renewcommand{\accuk}[1]{\mmu^{#1}}
\renewcommand{\LCaccuk}[1]{\tilde{\mmu}^{#1}}

\newcommand{\Pseqfunction}{\phi}
\renewcommand{\Pseqk}[1]{\Pseqfunction^{#1}}

\newcommand{\Nlamb}{\hat \llamb}
\newcommand{\Nlambi}[1]{\Nlamb_{#1}}
\newcommand{\Nlambk}[1]{\Nlamb^{#1}}
\newcommand{\Nlambki}[2]{\Nlamb^{#1}_{#2}}
\newcommand{\NMfunction}{\hat{S}}
\newcommand{\NMfx}[2]{\NMfunction({#1},{#2})}
\newcommand{\NMRfunction}{\hat{R}}
\newcommand{\NMRfx}[2]{\NMRfunction({#1},{#2})}
\newcommand{\NMGifunction}[1]{\hat{G}_{#1}}
\newcommand{\NMGifx}[3]{\NMGifunction{#1}({#2},{#3})}
\newcommand{\NMRFfunction}{\NMRfunction^{[\smoothfunction]}}
\newcommand{\NMRFfx}[2]{\NMRFfunction({#1},{#2})}
\newcommand{\NMRSfunction}{\NMRfunction^{[\supernormfunction]}}
\newcommand{\NMRSfx}[2]{\NMRSfunction({#1},{#2})}

\newcommand{\smoothmin}{\smoothfunction^{\ballradius}}

\renewcommand{\notxconv}{\notx^\star}
\newcommand{\notyconv}{\noty^\star}
\newcommand{\notyk}[1]{\noty^{#1}}

\renewcommand{\Probafunction}{\textup{Pr}}
\newcommand{\eventin}{\theta}
\newcommand{\eventout}{\neg\eventin}
\newcommand{\eventink}[1]{\eventin^{#1}}
\newcommand{\eventoutk}[1]{\eventout^{#1}}
\newcommand{\difference}{v}
\newcommand{\differencek}[1]{\difference^{#1}}
\newcommand{\differencetwo}{w}
\newcommand{\differencetwok}[1]{\difference^{#1}}
\newcommand{\genericquantity}{r}
\newcommand{\genericquantityk}[1]{\genericquantity^{#1}}
\newcommand{\genericresidualfunction}{h}
\newcommand{\genericresidual}[1]{\genericresidualfunction({#1})}
\newcommand{\LCMHfunction}{H}
\newcommand{\LCMHfx}[2]{\LCMHfunction}

\newcommand{\strongconvexity}{c}

\newcommand{\setoffunctions}{F}
\newcommand{\setoffunctionsn}[1]{\setoffunctions({#1})}

\newcommand{\NDG}{G}
\newcommand{\NDGk}[1]{\NDG^{#1}}
\newcommand{\NDGki}[2]{\NDGk{#1}_{#2}}
\newcommand{\NDMset}{\Sigma}
\newcommand{\NDM}{\genericnormscale}
\newcommand{\NDMk}[1]{\NDM^{#1}}
\newcommand{\NDMki}[2]{\NDMk{#1}}
\newcommand{\NDH}{H}
\newcommand{\NDHx}[1]{\NDH({#1})}
\newcommand{\NDHmax}{\bar\NDH}

\renewcommand{\GPfunction}{\smoothfunction}
\renewcommand{\function}{\smoothfunction}
\renewcommand{\functionk}[1]{\smoothfunction^{#1}}

\newcommand{\smoothRfunction}{\mathcal{R}}
\newcommand{\smoothRi}[1]{\smoothRfunction_{#1}}
\newcommand{\smoothRix}[2]{\smoothRi{#1}({#2})}
\newcommand{\smoothVfunction}{{R}}
\newcommand{\smoothVi}[1]{\smoothVfunction_{#1}}
\newcommand{\smoothVix}[2]{\smoothVi{#1}({#2})}

\newcommand{\compositeRfunction}{\mathcal{C}}
\newcommand{\compositeRi}[1]{\compositeRfunction_{#1}}
\newcommand{\compositeRix}[2]{\compositeRi{#1}({#2})}
\newcommand{\compositeVfunction}{{C}}
\newcommand{\compositeVi}[1]{\compositeVfunction_{#1}}
\newcommand{\compositeVix}[2]{\compositeVi{#1}({#2})}
\newcommand{\compositevfunction}{{c}}
\newcommand{\compositevi}[1]{\compositevfunction_{#1}}
\newcommand{\compositevixy}[3]{\compositevi{#1}({#2},{#3})}

\newcommand{\convexfunction}{g}
\newcommand{\convexifunction}[1]{\convexfunction_{#1}}
\newcommand{\convex}[1]{\convexfunction({#1})}
\newcommand{\convexi}[2]{\convexifunction{#1}({#2})}
\newcommand{\compositefunction}{h}
\newcommand{\composite}[1]{\compositefunction({#1})}

\newcommand{\subdifferential}{S}
\newcommand{\subdiff}[1]{\subdifferential({#1})}
\newcommand{\subdifferentiali}[1]{\subdifferential_{#1}}
\newcommand{\subdiffi}[2]{\subdifferentiali{#1}({#2})}
\newcommand{\subgrad}{s}
\newcommand{\subgradi}[1]{\subgrad_{#1}}

\newcommand{\functionHfunction}{H}
\newcommand{\functionHxT}[2]{\functionHfunction({#1},{#2})}
\newcommand{\vecT}{t}
\newcommand{\vecTi}[1]{\vecT_{#1}}
\newcommand{\optT}{\bar{\vecT}}
\newcommand{\optTi}[1]{\optT_{#1}}
\newcommand{\optTx}[1]{\optT({#1})}
\newcommand{\optTix}[2]{\optTi{#1}({#2})}

\newcommand{\makems}[1]{\bar{#1}}
\newcommand{\makeps}[1]{\textup{\b{${#1}$}}}
\newcommand{\LmT}{T}
\newcommand{\LmH}{H}
\newcommand{\LmL}{L}
\newcommand{\Lml}{l}
\newcommand{\Lmlamb}{\lambda}
\newcommand{\LmD}{D}
\newcommand{\LmG}{G}
\newcommand{\LmM}{M}
\newcommand{\LmZ}{Z}

\newcommand{\identityps}{\makeps{\identity}}
\newcommand{\identityms}{\makems{\identity}}
\newcommand{\identitypsi}[1]{\identityps_{#1}}
\newcommand{\identitymsi}[1]{\identityms_{#1}}
\newcommand{\LmTps}{\makeps{\LmT}}
\newcommand{\LmLps}{\makeps{\LmL}}
\newcommand{\LmDps}{\makeps{\LmD}}
\newcommand{\Lmlps}{\makeps{\Lml}}
\newcommand{\LmTms}{\makems{\LmT}}
\newcommand{\LmLms}{\makems{\LmL}}
\newcommand{\LmDms}{\makems{\LmD}}
\newcommand{\LmZms}{\makems{\LmZ}}
\newcommand{\LmTi}[1]{\LmT_{#1}}
\newcommand{\LmTpsi}[1]{\LmTps_{#1}}
\newcommand{\LmTmsi}[1]{\LmTms_{#1}}
\newcommand{\LmDi}[1]{\LmD_{#1}}
\newcommand{\LmDpsi}[1]{\LmDps_{#1}}
\newcommand{\LmDmsi}[1]{\LmDms_{#1}}
\newcommand{\LmLi}[1]{\LmL_{#1}}
\newcommand{\LmLpsi}[1]{\LmLps_{#1}}
\newcommand{\LmLmsi}[1]{\LmLms_{#1}}
\newcommand{\Lmli}[1]{\Lml_{#1}}
\newcommand{\Lmlpsi}[1]{\Lmlps_{#1}}
\newcommand{\Lmlambi}[1]{\Lmlamb_{#1}}
\newcommand{\LmGi}[1]{\LmG_{#1}}
\newcommand{\LmZmsi}[1]{\LmZms_{#1}}
\newcommand{\LmMi}[1]{\LmM_{#1}}
\newcommand{\LmHij}[2]{\LmH_{{#1}{#2}}}

\newcommand{\matrixsep}{\, & \,}
\newcommand{\makeblockmatrixnine}[9]{\left(\begin{array}{rrr}{#1}\matrixsep{#2}\matrixsep{#3}\\{#4}\matrixsep{#5}\matrixsep{#6}\\{#7}\matrixsep{#8}\matrixsep{#9}\end{array}\right)}
\newcommand{\makeblockmatrixfour}[4]{\left(\begin{array}{rr}{#1}\matrixsep{#2}\\{#3}\matrixsep{#4}\end{array}\right)}

\newcommand{\scalarp}[2]{\langle{#1},{#2}\rangle}
\newcommand{\Scalarp}[2]{\left\langle{#1},{#2}\right\rangle}

\newcommand{\jotaIEEE}[2]{#1}
\renewcommand*{\QEDA}{}

\title{Asymptotic convergence rates for coordinate descent in polyhedral sets} 

\author{
Olivier Bilenne
\jotaIEEE{}{
\thanks{O. Bilenne is with the Control Systems Group, Technical University of Berlin, Germany.}
}
}

\jotaIEEE{
\institute{%
Olivier Bilenne\\
Control Systems Group, Technical University of Berlin, Germany\\
\email{bilenne@control.tu-berlin.de}
}

}{}

%
%

\markboth{}{} 
%



\maketitle

\vskip0mm

\begin{abstract}
We consider a family of parallel methods for constrained optimization based on projected gradient descents along individual coordinate directions. In the case of polyhedral feasible sets, local convergence towards a regular solution occurs unconstrained in a reduced space, allowing for the computation of tight asymptotic convergence rates by sensitivity analysis, this even when global convergence rates are unavailable or too conservative. We derive linear asymptotic rates of convergence in polyhedra for variants of the coordinate descent approach, including cyclic, synchronous, and random modes of implementation. Our results find application in stochastic optimization, and with recently proposed optimization algorithms based on Taylor approximations of the Newton step.
\end{abstract}
\obsolete{
\begin{abstract}
We consider a family of parallel  methods for constrained optimization based on projected gradient descents along individual coordinate directions. 
Coordinate descent methods typically prove to converge linearly for strongly convex functions. In the case of polyhedral feasible sets, asymptotic convergence towards a regular solution 
occurs unconstrained in a reduced space,
allowing for the computation of exact asymptotic  convergence rates by sensitivity analysis.
This study derives linear asymptotic rates of convergence in polyhedra for variants of the coordinate descent approach,  including cyclic, synchronous, and random modes of implementation. 
Our results find  application in stochastic optimization, and with recently proposed optimization algorithms based on Taylor approximations of the Newton step.
\end{abstract}
}%

\obsolete{
\jotaIEEE{%
}{
\begin{IEEEkeywords}
IEEEtran, journal, \LaTeX, paper, template.
\end{IEEEkeywords}
}%
}%

%
\jotaIEEE{}{\IEEEpeerreviewmaketitle}

\section{Introduction}

The interest for the coordinate descent methods lies in their simplicity of implementation and  flexibility~\cite{bertsekas97,bertsekas99}.
Yet their performances in terms of speed of convergence are generally modest compared to their centralized counterparts and still subject to active research.
In this work we derive the asymptotic convergence rates of parallel implementations of the gradient projection algorithm~\cite{bertsekas76} in the context of the constrained minimization of  a strictly convex, continuously differentiable function 
over a polyhedral feasible set---this class of problems  is met for instance in bound-constrained {optimization} or in dual optimization. 
Our developments rely on the property of projected gradient methods to asymptotically behave, when applied in a polyhedral feasible set specified by a collection of affine inequality constraints, like unconstrained gradient descents on the surface of the polyhedron, provided that the gradient of the cost function at the point of convergence be a negative combination of the normal vectors of the active constraints.
This property facilitates 
the derivation of rates of convergence in the form of matrices 
playing roles analogous to those of system matrices in linear state space models, thus reducing the question of the convergence  to the spectral analysis of  matrices.


\emph{Outline ---}
Section~\ref{sectiongradientprojection} formulates the gradient projection algorithm and identifies certain properties enjoyed by the  method  in polyhedral  sets. From the initial algorithm we derive parallelized implementations operating gradient descents along coordinate directions, each of them characterized by the way these operations are organized (synchronously, cyclically, randomly, etc.). In Section~\ref{sectionconvergencerates} we compute asymptotic convergence rates for the parallel algorithms under the hypothesis of twice continuous differentiability 
at the point of convergence. Our developments are then reconsidered for non-twice differentiable cost functions and from the perspective  of stochastic optimization settings.

\emph{Notation ---}
In this paper vectors are column vectors and denoted by $\notx =  \vect{\notxi{1},...,\notxi{\nn}}$, where $\notxi{1},...,\notxi{\nn}$ are the coordinates of~$\notx$. 
Subscripts are reserved for vector coordinates.
The transpose of a  vector $\notx\in\REAL{\notn}$ is denoted by~$\notx^\trsp$ and its Euclidean norm by~$\norm{\notx}$;  for any~$\genericnormscale\in\REAL{\notn\times \notn}$ symmetric, positive definite, we define the scaled norm 
$ \specialnorm{\xx}{\genericnormscale}\doteq (\xx^\trsp \genericnormscale \xx)^{-\frac{1}{2}}$.
%
%
%
Let~$\notS$ be a finite set, $\{\Pak{k}\}$  a sequence in~$\notS$, and $\Pa\in\notS$. We write $\Pak{k}\to\Pa$ iff there is a~$\ksc$ such that $\Pak{k}=\Pa$ for  $k>\ksc$. Similarly, for $\notA\subset \notS$ and a sequence $\{\notAk{k}\}$ of subsets of~$\notS$, we write~$\notAk{k}\to\notA$ iff there is a~$\ksc$ such that $\notAk{k}\equiv\notA$ for  $k>\ksc$.

\section{The gradient projection algorithm}\label{sectiongradientprojection}

\subsection{Formulation}

Consider a closed convex subset~$\setoflambdas$ of a real vector space~$\RMM$ and a function $\function\in\setoffunctionsn{\MM}$, where, for any space~$\REAL{\genericn}$, $\setoffunctionsn{\genericn}$ denotes the set of the functions $\REAL{\genericn}\functionto\REALone$ strictly convex, continuously differentiable with gradient~$\grad\function$ Lipschitz continuous. 
Lipschitz continuity of~$\grad\function$ can be understood as the existence of  a   symmetric, positive definite matrix $\Lipschitz\in\REAL{\genericn\times\genericn} $ satisfying 
$
[\grad\functionx{\llamb}-\grad\functionx{\mmu}]^\trsp(\llamb-\mmu)
\leq
\specialnorm{\llamb-\mmu}{\boundL}^2
$ for any $ \llamb,\mmu\in\REAL{\genericn}$. It follows from this condition that\footnotemark{}
%
%
%
\begin{align}
&
\grad\smooth{\llamb}^\trsp(\mmu-\llamb)
\geq \label{Lequation}
\smooth{\mmu}-\smooth{\llamb}-\frac{1}{2}\specialnorm{\llamb-\mmu}{\Lipschitz}^2
,
&
\forall \llamb,\mmu\in\REAL{\genericn}.
\end{align}
%
%
\footnotetext{
To show~(\ref{Lequation}), use for instance
$
\smooth{\mmu}
=
\smooth{\llamb}
+
\grad\smooth{\mmu}^\trsp(\mmu-\llamb)
+
\int_{0}^{1} [ 
\grad\smooth{\llamb+\xi(\mmu-\llamb)}
-
\grad\smooth{\mmu}
   ]^\trsp 
(\mmu-\llamb)  \, d\xi
$.
}%
\obsolete{
Since the paper is concerned with the asymptotic properties of convergent gradient descent algorithms at a supposedly unique point of convergence~$\llambconv$,  it is assumed---at least in the vicinity of the convergence point----that the function~$\function$ is continuously differentiable, strictly convex with positive curvature (strong convexity), and that~$\grad\function$ is Lipschitz continuous with constant $\lipschitz>0$, i.e.
$
\norm{\grad\functionx{\llamb}-\grad\functionx{\mmu}}
\leq
\lipschitz
\norm{\llamb-\mmu}
$
holds near~$\llambconv$, which is a common assumption for gradient methods.

-----------

}%
\obsolete{
Let~$\CCLall{\genericn}$ denote the class of 
the functions~$\GPfunction:\REAL{\notn}\functionto\REALonepinfty$ proper and lower semicontinuous,
continuously differentiable on their domain, and such that~$\grad\GPfunction$ is Lipschitz continuous on~$\dom{\GPfunction}$ with Lipschitz constant~$\lipschitz$.
}%
Let~$\eigentbound$ and~$\eigenTbound$ be two positive scalar constants such that~$0 < \eigentbound \leq\eigenTbound < \infty $.
For any real space~$\REAL{\genericn}$, we let~$\setofScalesn{\genericn}  $ define the set of the symmetric, positive definite scaling matrices in~$\REAL{\genericn\times\genericn}$ with eigenvalues 
bounded by~$\eigentbound$ and~$\eigenTbound$, i.e.
$ 
\setofScalesn{\genericn}  = \{ \Scale \in  \REAL{\genericn\times\genericn} : \eigentbound I \matrixleq \Scale \matrixleq \eigenTbound 	I \} 
$. 
We consider the following algorithm. 

\begin{algorithm}[Scaled gradient projection~$\mapGTXfunction{\Scale}{\genericset}$]
\label{scaledgradientprojectionmapping} 
%
%
%
Consider a closed, convex set $\genericset\subset\REAL{\genericn}$, a function $\function\in\setoffunctionsn{\genericn}$,
 a scaling mapping
%
$\Scale:\setoffunctionsn{\genericn}\times\genericset\functionto\setofScalesn{\genericn}$,
 fixed scalar parameters $  \Pbeta, \Psig \in (0,1)$,
and an initial point~$\notxk{0}\in\genericset$.
A scaled gradient projection algorithm is given by 
\begin{equation}\label{gradientprojectionalgorithm}
\notxk{k+1}=\mapGTXfx{\Scale}{\genericset}{\GPfunction}{\notxk{k}},
\hspace{5mm}k=0,1,2,...,
\end{equation}
with~$\mapGTXfunction{\Scale}{\genericset}$ 
defined
for
$\notx\in\genericset$
by
$\mapGTXfx{\Scale}{\genericset}{\GPfunction}{\notx}\doteq\searchfxa{\GPfunction}{\notx}{
\newPa{\Paok}{\mapAfx{\GPfunction}{\notx}}
}$,
where
\begin{equation} \label{globalequation}
\searchfxa{\GPfunction}{\notx}{\Pa}
\in
\arg\min\nolimits
_{\noty\in\genericset}
\grad\GPfunctionx{\notx}^\trsp(\noty-\notx)+\frac{1}{2}
\specialnorm{\noty-\notx}{[\Pa \Scalefx{\GPfunction}{\notx}]^{-1}}^2
,
\ \forall \Pa>0,
\end{equation}
and
\newPa{$\Paok$ is an appropriate step size  bounded above~$0$.}{$\mapAfunction:\CCLall{\genericn}\times\genericset\functionto (0,1]$ is an appropriate step size rule.}
\end{algorithm}
%

%
Any  point $\notx\in\genericset$  such that $\mapGTXfx{\Scale}{\genericset}{\GPfunction}{\notx}=\notx$ is called stationary.
Since by assumption~$\GPfunction$ is convex, the stationary points coincide with the solutions of the  minimization of~$\GPfunction$. The first-order optimality condition of a point~$\notx\in\genericset$ is therefore given by
\begin{equation}\label{firstorderoptimalitycondition}
\grad\GPfunctionx{\notx}^\trsp(\noty-\notx)
\geq
0
,\hspace{5mm}\forall \noty\in\genericset
.
\end{equation}
%
If~$\GPfunction$ is strictly convex, (\ref{firstorderoptimalitycondition}) holds for at most one point and there is at most one solution. Notice that condition~(\ref{firstorderoptimalitycondition})
reduces, for the subproblem~(\ref{globalequation}) and any step size $\Pa>0$, to
\begin{equation}\label{gradientdescentoptimalitycondition}
\left[\grad\GPfunctionx{\notx}+[\Pa \Scalefx{\GPfunction}{\notx}]^{-1}(\searchfxa{\GPfunction}{\notx}{\Pa}-\notx) \right]^\trsp(\noty-\searchfxa{\GPfunction}{\notx}{\Pa})
\geq
0
,\hspace{5mm}\forall \noty\in\REAL{\genericn}
.
\end{equation}

The notion of \lq{}gradient projection\rq{} in Algorithm~\ref{scaledgradientprojectionmapping} can be explained by the observation that~$\searchfxa{\GPfunction}{\notx}{\Pa}$ in~(\ref{globalequation}) coincides with the scaled projection on~$\genericset$,
\begin{equation}\label{equationprojection}
\searchfxa{\GPfunction}{\notx}{\Pa}\in
\arg\min\nolimits_{\noty\in\genericset}
\specialnorm{\noty-\notz}{\Scalefx{\GPfunction}{\notx}^{-1}}^2
,
\end{equation}
of the vector 
$
\notz=\notx-\Pa\Scalefx{\GPfunction}{\notx}\grad\GPfunctionx{\notx}
$ 
obtained by scaled gradient descent from~$\notx$.
It follows from the convexity of~$\genericset$ and from the  projection theorem~\cite[Proposition~3.7 in Section~3.3]{bertsekas97} that~$\searchfxa{\GPfunction}{\notx}{\Pa}$ is uniquely defined in~(\ref{equationprojection}) and~(\ref{globalequation}).

Global convergence of~(\ref{gradientprojectionalgorithm}) is commonly guaranteed 
by
using an approximate line search rule 
of the type {Armijo}~\cite{armijo66},
which consists of 
\newPa{setting $\Paok\equiv\mapAfx{\GPfunction}{\notxk{k}}$ where, for $\notx\in\genericset$, $\mapAfx{\GPfunction}{\notx}$ is defined as}{setting~$\mapAfx{\GPfunction}{\notx}$ to} 
the largest $\Pa\in\{ \Pbeta^{\Pm} \}_{\Pm=0}^{\infty}$ satisfying
\begin{equation} \label{stocharmijoconditionglobalone} 
%
\GPfunctionx{ \notx} - \GPfunctionx{\searchfxa{\GPfunction}{\notx}{\Pa}} \geq {\Psig}  
\specialnorm{\searchfxa{\GPfunction}{\notx}{\Pa}-\notx}{[\Pa \Scalefx{\GPfunction}{\notx}]^{-1}}^2
.
\end{equation}
From~\cite{bertsekas76} we know that the step-sizes computed by~(\ref{stocharmijoconditionglobalone}) are restricted to a set~$[\Pamin,1]$, where $\Pamin>0$ is a function of the Lipschitz constant of~$\grad\function$. 
\obsolete{
Since the present study is focused on local convergence issues, we will not extend on the technique of approximate line search  considering that the range of the step-sizes computed by line search ideally reduces to~$\{1\}$ when approaching the point of convergence, and step-size selection rules such as~(\ref{stocharmijoconditionglobalone}) become asymptotically trivial, as illustrated by the next result.
}%
When the algorithm is appropriately designed, 
(\ref{stocharmijoconditionglobalone}) becomes asymptotically trivial. This is illustrated by the next result, shown in the Appendix, where~$\Lipschitz$ denotes the Lipschitz constant in the sense of~(\ref{Lequation}).
\begin{proposition}[Line search efficiency]\label{propositionefficiency}
Suppose that Algorithm~\ref{scaledgradientprojectionmapping} is implemented with the step-size selection rule~(\ref{stocharmijoconditionglobalone}) and generates a sequence~$\{\notxk{k}\}$ converging to a stationary point~$\notxconv$. 
If
\begin{equation} \label{conditionefficiencynonlinear}
2 (1-\Psig) \Scalefx{\GPfunction}{\notx}^{-1}  \matrixgeq \Lipschitz,
\hspace{5mm}
\forall \notx\in\genericset,
\end{equation}
then $\mapAfx{\GPfunction}{\notxk{k}}= 1$ for all $k$.
If~$\Scalefx{\GPfunction}{\cdot}$ is continuous, $\GPfunction$ is twice continuously differentiable in a neighborhood of~$\notxconv$, and
\begin{equation} \label{conditionefficiency}
2 (1-\Psig) \Scalefx{\GPfunction}{\notxconv}^{-1}  \matrixg \hessian{\GPfunction}{\notxconv},
\end{equation}
then $\mapAfx{\GPfunction}{\notxk{k}}\to 1$.
\end{proposition}

\obsolete{

\begin{proposition}[Asymptotic line search efficiency]\label{propositionefficiency}
Suppose that Algorithm~\ref{scaledgradientprojectionmapping} is implemented with the step-size selection rule~(\ref{stocharmijoconditionglobalone}) and generates a sequence~$\{\notxk{k}\}$ converging to a solution~$\llambconv$ in a neighborhood of which~$\Scalefx{\GPfunction}{\cdot}$ is continuous and~$\GPfunction$ is twice continuously differentiable.
If
\begin{equation} \label{conditionefficiency}
2 (1-\Psig) \Scalefx{\GPfunction}{\notxconv}^{-1}  \matrixg \hessian{\GPfunction}{\notxconv},
\end{equation}
then $\mapAfx{\GPfunction}{\notxk{k}}\to 1$ as $k\to\infty$.
\end{proposition}
\begin{proof}
Let~$\notx\in\genericset$ and $\Pa\in(0,1]$, and define 
\begin{equation}
\locus{\noty}=(\noty-  \notx)^\trsp[\grad\GPfunctionx{\notx}+[\Pa \Scalefx{\GPfunction}{\notx}]^{-1}(\noty-\notx)],
\end{equation} 
where the second factor in the right member is the gradient of the cost approximation in~(\ref{globalequation}).
The locus of all the points~$\noty$ where the level sets of the cost approximation  are tangent to $\noty-\notx$ is thus given by $\{y\in\RMM\setst\locus{\noty}=0\}$.
%

From~(\ref{equationprojection}) we know that~$\searchfxa{\GPfunction}{\notx}{\Pa}$ is the unique projection of the point~$\notz=\notx-\Pa\Scalefx{\GPfunction}{\notx}$ on~$\genericset$.
Hence all the possible locations for~$\searchfxa{\GPfunction}{\notx}{\Pa}$ are included in the set $\{y\in\RMM\setst\locus{\noty}\leq 0\}$. Indeed, if we suppose that $\locus{\searchfxa{\GPfunction}{\notx}{\Pa}}> 0$, then by convexity of~$\genericset$ it is possible to find a point  on the open segment $(\notx, \searchfxa{\GPfunction}{\notx}{\Pa}  )$ that belongs to~$\genericset$ and minimizes~(\ref{globalequation}), which contradicts the uniqueness of~$\searchfxa{\GPfunction}{\notx}{\Pa}$.
%

For any~$\noty$ in the vicinity of~$\notx$ we find,  by Taylor's theorem,
\begin{equation}\label{taylortheorem}
\begin{array}{c}
\GPfunctionx{\noty}=\GPfunctionx{\notx}+\grad\GPfunctionx{\notx}^\trsp(\noty-\notx)+\frac{1}{2}(\noty-\notx)^\trsp\hessian{\GPfunction}{\notx}(\noty-\notx)+\zero{\norm{\noty-\notx}^2}.
\end{array}
\end{equation}
It follows from~(\ref{taylortheorem})  
that~(\ref{stocharmijoconditionglobalone}) is satisfied when $\locus{\searchfxa{\GPfunction}{\notx}{\Pa}}\leq 0$
if
\begin{equation}\label{equationefficiency}
\begin{array}{c}
(\searchfxa{\GPfunction}{\notx}{\Pa}-\notx)^\trsp
\left[
\frac{1}{2}\hessian{\GPfunction}{\notx}
-
(1-{\Psig})[\Pa \Scalefx{\GPfunction}{\notx}]^{-1}
\right]
(\searchfxa{\GPfunction}{\notx}{\Pa}-\notx)+\zero{\norm{\searchfxa{\GPfunction}{\notx}{\Pa}-\notx}^2} \leq 0.
\end{array}
\end{equation}
By continuity of~$\hessianfunction{\GPfunction}$ and~$\Scalefx{\GPfunction}{\cdot}$ near~$\notxconv$, we find that~(\ref{equationefficiency}) is satisfied for $\Pa=1$ when~(\ref{conditionefficiency}) holds. That 
$\norm{\mapGTXfx{\Scale}{\genericset}{\GPfunction}{\notxk{k}}-\notxk{k}}\to 0$ completes the proof.
\QEDA
\end{proof}

}%

\subsection{Descent in polyhedral sets}\label{sectionpolyhedra}

%
Throughout the paper we consider the following problem.
\begin{problem}\label{initialproblem}
Solve
\begin{equation}
\min\nolimits_{\llamb\in\setoflambdas}{\smooth{\llamb}}
\end{equation}
where $\smoothfunction\in\setoffunctionsn{\MM}$, $\grad\function$ satisfies~(\ref{Lequation}) with Lipschitz constant~$\Lipschitz$,
and~$\setoflambdas$ is  the nonempty polyhedron 
$
\setoflambdas \doteq \{ \llamb \in\RMM \setst 
\affineaj{1}^\trsp\llamb \leq \affinebj{1}  ,
\,
\affineaj{2}^\trsp\llamb \leq \affinebj{2}  ,
...
,
\,
\affineaj{\affinen}^\trsp\llamb \leq \affinebj{\affinen}   
\}$,
with
$\affineaj{1},...,\affineaj{\affinen}\in\REAL{\MM}$
 and
$\affinebj{1},...,\affinebj{\affinen}\in\REALone$.
\obsolete{Minimize $\functionx{\llamb}=\smooth{\llamb}+
\characteristicx{\setoflambdas}{\llamb}
$, where  $\smoothfunction:\RMM\to\REALone$ is  convex, continuously differentiable and such that~$\grad\smoothfunction$ is Lispschitz continuous with Lipschitz  constant~$\lipschitz$, }
\obsolete{
Minimize a function $\smoothfunction\in\setoffunctionsn{\MM}$ such that~$\grad\function$ satisfies~(\ref{Lequation}) with Lipschitz constant~$\Lipschitz$ 
over the nonempty polyhedron 
$
\setoflambdas \doteq \{ \llamb \in\RMM \setst 
\affineaj{1}^\trsp\llamb \leq \affinebj{1}  ,
\,
\affineaj{2}^\trsp\llamb \leq \affinebj{2}  ,
...
,
\,
\affineaj{\affinen}^\trsp\llamb \leq \affinebj{\affinen}   
\}$,
where
$\affineaj{1},...,\affineaj{\affinen}\in\REAL{\MM}$
 and
$\affinebj{1},...,\affinebj{\affinen}\in\REALone$.
}%
\obsolete{
\begin{equation}\label{definitionpolyhedron}
\setoflambdas = \{ \llamb \in\RMM \setst 
\affineaj{1}^\trsp\llamb \leq \affinebj{1}  ,
...
,
\,
\affineaj{\affinen}^\trsp\llamb \leq \affinebj{\affinen}   
\}
,\hspace{3mm}
\affineaj{1},...,\affineaj{\affinen}\in\REAL{\MM}
,\ 
\affinebj{1},...,\affinebj{\affinen}\in\REALone
.
\end{equation}
}%
%
\obsolete{
rewrites as
\begin{equation}
\minimise{\llamb\in\setoflambdas} \smooth{\llamb}.
\end{equation}
}%
\end{problem}
\obsolete{
Since it considers a function of the type~(\ref{differentiablebox}) with~$\smoothfunction$  convex, $\Pfunction= \characteristic{\setoflambdas}$,  $\dom{\function}=\setoflambdas$, and $\nn=1$,
Problem~\ref{initialproblem} qualifies for Algorithm~\ref{scaledgradientprojectionmapping}.
Although the problem is first treated globally ($\nn=1$), the  results of this section extend naturally to the parallel optimization setting of Section~\ref{sectionparalleloptimization}.
}%
\obsolete{
We consider the following problem.
\begin{problem}\label{initialproblem}
\begin{equation}
\minimise{\llamb\in\setoflambdas} \smooth{\llamb}
\end{equation}
where~$\smoothfunction:\RMM\to\REALone$ is  convex, continuously differentiable and such that~$\grad\smoothfunction$ is Lispschitz continuous with Lipschitz  constant~$\lipschitz$, and~$\setoflambdas$ is the polyhedron 
\begin{equation}\label{definitionpolyhedron}
\setoflambdas = \{ \llamb \in\RMM \setst 
\affineaj{1}^\trsp\llamb \leq \affinebj{1}  ,
...
,
\,
\affineaj{\affinen}^\trsp\llamb \leq \affinebj{\affinen}   
\}
,\hspace{3mm}
\affineaj{1},...,\affineaj{\affinen}\in\REAL{\MM}
,\ 
\affinebj{1},...,\affinebj{\affinen}\in\REALone
.
\end{equation}
\obsolete{
Consider the minimization 
over the polyhedron 
\begin{equation}\label{definitionpolyhedron}
\setoflambdas = \{ \llamb \in\RMM \setst 
\affineaj{1}^\trsp\llamb \leq \affinebj{1}  ,
...
,
\,
\affineaj{\affinen}^\trsp\llamb \leq \affinebj{\affinen}   
\}
,\hspace{3mm}
\affineaj{1},...,\affineaj{\affinen}\in\REAL{\MM}
,\ 
\affinebj{1},...,\affinebj{\affinen}\in\REALone
\end{equation}
of a  
function~$\smoothfunction:\RMM\to\REALone$ convex, continuously differentiable and such that~$\grad\smoothfunction$ is Lipschitz continuous with Lipschitz  constant~$\lipschitz$.
}
\end{problem}
Problem~\ref{initialproblem} qualifies for Algorithm~\ref{scaledgradientprojectionmapping} since it is equivalent to minimizing the function $\functionx{\llamb}=\smooth{\llamb}+\Px{\llamb}$
 given by~(\ref{differentiablebox}) with~$\smoothfunction$  convex, $\Pfunction= \characteristic{\setoflambdas}$, where $\characteristicx{\setoflambdas}{\llamb}=0$ if $\llamb\in\setoflambdas$ and $\characteristicx{\setoflambdas}{\llamb}=+\infty$ otherwise, and thus $\dom{\function}=\setoflambdas$.
Although the problem is first treated globally ($\nn=1$), the  results of this section extend naturally to the parallel optimization setting of Section~\ref{sectionparalleloptimization}.
}%
\obsolete{
We minimize the function $\functionx{\llamb}=\smooth{\llamb}+\Px{\llamb}$
 given by~(\ref{differentiablebox}), 
in which we assume that~$\smoothfunction$ is convex and  set $\Pfunction = \characteristic{\setoflambdas}$,
where~$\setoflambdas$ is a polyhedron of the type
\begin{equation}\label{definitionpolyhedron}
\setoflambdas = \{ \llamb \in\RMM \setst 
\affineaj{1}^\trsp\llamb \leq \affinebj{1}  ,
...
,
\,
\affineaj{\affinen}^\trsp\llamb \leq \affinebj{\affinen}   
\}
,\hspace{3mm}
\affineaj{1},...,\affineaj{\affinen}\in\REAL{\MM}
,\ 
\affinebj{1},...,\affinebj{\affinen}\in\REALone
,
\end{equation}
and~$\characteristic{\setoflambdas}$ denotes the characteristic function of~$\setoflambdas$, i.e. $\characteristicx{\setoflambdas}{\llamb}=0$ if $\llamb\in\setoflambdas$ and $\characteristicx{\setoflambdas}{\llamb}=+\infty$ otherwise. This amounts to minimizing the convex, continuously differentiable function~$\smoothfunction$ on the polyhedron~$\setoflambdas$, and the problem is initially treated globally ($\nn=1$). The results of this section extend naturally to the parallel optimization setting of Section~\ref{sectionparalleloptimization}.
}
The affine constraint functions in Problem~\ref{initialproblem} can be rewritten  as $ \cstjx{j}{\llamb} \leq 0$, where $ \cstjx{j}{\llamb} \doteq \affineaj{j}^\trsp\llamb - \affinebj{j} $ and $\grad  \cstjx{j}{\llamb} = \affineaj{j} $ for~$\llamb\in\RMM$  ($j=1,...,\affinen$).
A constraint $ \cstjx{j}{\llamb}\leq0$ is said to be \emph{inactive} at a point~$\notx\in\RMM$ if $\cstjx{j}{\llamb}<0$, and \emph{active} if $\cstjx{j}{\llamb}=0$, in which case we write $j\in\activeaffinex{\notx}$, where $\activeaffinex{\notx}\subset\{1,...,\affinen\}$ denotes the index set of the active constraints  at~$\notx$.
If a constraint qualification holds for Problem~\ref{initialproblem} (e.g. Slater's condition~\cite{boyd04}), then
 the first-order optimality condition~(\ref{firstorderoptimalitycondition}) for a point~$\llambconv\in\setoflambdas$ translates, in accordance with the \emph{Karush-Kuhn-Tucker (KKT)} conditions, into the existence of nonnegative coefficients 
$\{\CLcoefj{j}\}_{j\in\activeaffinex{\llambconv}}$
satisfying
\begin{equation} \label{KKTsc}
\begin{array}{c}
\grad \smooth {\llambconv} = - \sum\nolimits_{j\in\activeaffinex{\notx}} \CLcoefj{j} \grad \cstjx{j}{\llambconv}.
\end{array}
\end{equation}
Frequently, a  solution~$\llambconv\in\setoflambdas$ of Problem~\ref{initialproblem} will meet the stronger condition that~(\ref{KKTsc}) holds for positive  coefficients
$\{\CLcoefj{j}\}_{j\in\activeaffinex{\llambconv}}$. 
In that case we say that \emph{strict complementarity} holds at~$\llambconv$---thus extending to polyhedra a notion discussed in~\cite{kelley99} in the context of bound-constrained optimization---, 
and it follows that~$\activeaffinex{\llambconv}$ is identified in finite time by the gradient projection algorithm.
%
%
\begin{proposition} [{Identification of the active  constraints}] \label{identificationofthesaturatednonnegativityconstraints}
Assume that Problem~\ref{initialproblem} admits a solution~$\llambconv$  where  strict complementarity holds.
Then one can find a~$ \ballradius>0$ such that $\activeaffinex{\mapGTXfx{\Scale}{\setoflambdas}{\function}{\llamb}}=\activeaffinex{\llambconv}$ for any $\llamb\in\setoflambdas$ satisfying $\norm{\llamb-\llambconv}<\ballradius$.
Moreover,  any sequence~$\{\llambk{k}\}$   generated by Algorithm~\ref{scaledgradientprojectionmapping} and converging  towards~$\llambconv$ is such that $\activeaffinex{\llambk{k}}\to\activeaffinex{\llambconv}$.
\obsolete{
Consider Problem~\ref{initialproblem} and let~$\{\llambk{k}\}$ be a sequence  generated by Algorithm~\ref{scaledgradientprojectionmapping} converging  to a solution~$\llambconv$  where  strict complementarity holds.
Then 
$\activeaffinex{\llambk{k}}\to\activeaffinex{\llambconv}$, and one can find a~$ \ballradius>0$ such that $\activeaffinex{\mapGTXfx{\Scale}{\setoflambdas}{\function}{\llamb}}=\activeaffinex{\llambconv}$ for any $\llamb\in\setoflambdas$ satisfying $\norm{\llamb-\llambconv}<\ballradius$.
}%
\end{proposition}
The proof 
is given in the Appendix.
%
%
%
%
%
\label{referencereducedspace}
A consequence of Proposition~\ref{identificationofthesaturatednonnegativityconstraints} is that, under strict stationarity at a point of convergence~$\llambconv$, local convergence occurs in a subspace $
\{\llamb\in\RMM\setst \affineaj{j}^\trsp\notx=\affinebj{j},\ j\in\activeaffinex{\llambconv}\}
$  with dimension $\tildeMM\leq\MM$, called the \emph{reduced space} at~$\llambconv$, and orthogonal to  the normal vectors of all the active constraints  at~$\llambconv$. 
By~$\matrixnonviolatingx{\llambconv}$ we denote any matrix whose columns  form an orthonormal basis of the reduced space at~$\llambconv$.
For any $\llamb \in \setoflambdas$ such that $\activeaffinex{\llamb} = \activeaffinex{\llambconv}$, there is a unique vector~$\LClamb \in \REAL{\tildeMM}$ satisfying $\llamb = \llambconv + \matrixnonviolatingx{\llambconv} \LClamb $.
%
%
%
%
%
\obsolete{%
The following result states that the gradient projections reduce to mere gradient descents in the vicinity of~$\llambconv$.
\begin{proposition} [Descent in the reduced space] \label{unconstraineddescent}
Consider Problem~\ref{initialproblem} and let~$\{\llambk{k}\}$ be a sequence  generated by Algorithm~\ref{scaledgradientprojectionmapping} which converges  to a solution~$\llambconv$  where  strict complementarity holds.
One can find a~$\ksc<\infty$ such that, for~$k>\ksc$, one has 
$\llambk{k} = \llambconv + \matrixnonviolatingx{\llambconv} \LClambk{k}  $ and $\llambk{k+1} = \llambconv + \matrixnonviolatingx{\llambconv} \LClambk{k+1}  $ for some
$  \LClambk{k}, \LClambk{k+1} \in \REAL{\tildeMM} $ satisfying
\begin{equation} \label{unconstrainedequationnew}
\LClambk{k+1} = 
\LClambk{k} - \newPa{\Paokk{k}}{\mapAfx{\function}{\llambk{k}}} \LCScalefx{\function}{\llambk{k}} \gradfxtilde{\llambk{k}}
,
\end{equation}
where 
 $\gradftilde \doteq \matrixnonviolatingx{\llambconv}^\trsp \gradftilde $,
$\LCScalefx{\function}{\llambk{k}} \doteq \lbrack \matrixnonviolatingx{\llambconv}^\trsp \Scalefx{\function}{\llambk{k}}^{-1}   \matrixnonviolatingx{\llambconv} \rbrack^{-1} $,
and~$\newPa{\Paokk{k}}{\mapAfx{\function}{\llambk{k}}}$ is the step size at step~$k$.
%
\end{proposition}
\begin{proof}
%
Since the proposition is trivial when $\tildeMM=0$,
we suppose that $\tildeMM>0$.
%
Let
$\manifoldsetoflambdas \doteq \{ \llamb \in \setoflambdas \setst \setviolatingx{\llamb} = \setviolatingx{\llambconv} \} $,
and
$\LCsetoflambdas \doteq \{ \matrixnonviolatingx{\llambconv}^\trsp ( \llamb - \llambconv ) \setst \llamb \in \manifoldsetoflambdas \}$,
which is an open subset of $ \REAL{\tildeMM}$ by  continuity of the constraint functions.
%
Moreover, the function
$\bijectionfx{\LClamb} \doteq \llambconv + \matrixnonviolatingx{\llambconv} \LClamb$ 
is a bijection between~$\LCsetoflambdas$ and~$\manifoldsetoflambdas$, i.e.  $\manifoldsetoflambdas = \{ \bijectionfx{\LCmu} \setst \LCmu \in \LCsetoflambdas \} $.
It follows from Proposition~\ref{identificationofthesaturatednonnegativityconstraints}
that one can find a $\ksc<\infty$ such that, for all $k\geq\ksc$, 
we have
$\llambk{k}, \llambk{k+1} \in \manifoldsetoflambdas$,
and thus
$\llambk{k} =  \bijectionfx{\LClambk{k}}  $ and $\llambk{k+1} = \bijectionfx{\LClambk{k+1}}  $ for some points $  \LClambk{k}, \LClambk{k+1} \in \LCsetoflambdas $.
Based on 
$\llambk{k+1}=\mapGTXfx{\Scale}{\genericset}{\function}{\llambk{k}}$
with step size~$\Pa$ 
and~(\ref{globalequation}), it is straightforward to show that
%
\begin{align} 
\LClambk{k+1} 
 & 
\begin{array}{c}
=
\arg \min\nolimits_{\LCmu \in \LCsetoflambdas} \left\{  \gradfxtilde{\llambk{k}}^\trsp ({\LCmu} - {\LClambk{k} } )  
+ \frac{1 }{2}  
\specialnorm{ {\LCmu} - {\LClambk{k} } }{ [
\newPa{\Paokk{k}}{\mapAfx{\function}{\llambk{k}}}
 \LCScalefx{\function}{\llambk{k}} ]^{-1}}^2 \right\} 
\end{array}
\notag
\\    & =  \LClambk{k}  - 
\newPa{\Paokk{k}}{\mapAfx{\function}{\llambk{k} }} \LCScalefx{\function}{\llambk{k}} 
 \gradfxtilde{\llambk{k}}
\label{manifoldthree}
\end{align}
%
where~(\ref{manifoldthree}) follows from~(\ref{equationprojection}), and from the fact  that~$\LCsetoflambdas$ is 
an open set containing~$\LClambk{k+1}$, thus~$\LClambk{k+1}$ is the projection  on~$\LCsetoflambdas$ of only one point: itself.
%
%
\QEDA\end{proof}

---
}%
The following result states that the gradient projections reduce to mere gradient descents in the vicinity of~$\llambconv$ and derives asymptotics for~$\mapGTXfunction{\Scale}{\genericset}$.
We denote by~$\tildeidentity$ the identity matrix in~$\REAL{\tildeMM\times\tildeMM}$.
\begin{proposition} [Descent in the reduced space] \label{unconstraineddescent}
Let~$\llambconv$  be a solution of Problem~\ref{initialproblem} where strict complementarity holds, and consider  Algorithm~\ref{scaledgradientprojectionmapping}.
Any vectors $\llamb,\mmu\in\setoflambdas$ such that $\mmu=\mapGTXfx{\Scale}{\setoflambdas}{\function}{\llamb}$ with step size~$\newPa{\Paok}{\mapAfx{\function}{\llamb}}$,
and $  \LClamb, \LCmu \in \REAL{\tildeMM} $
such that
 $\llamb = \llambconv + \matrixnonviolatingx{\llambconv} \LClamb  $ and $\mmu = \llambconv + \matrixnonviolatingx{\llambconv} \LCmu  $,
satisfy
\begin{equation} \label{unconstrainedequationnew}
\LCmu = 
\LClamb - \newPa{\Paokk{k}}{\mapAfx{\function}{\llambk{k}}} \LCScalefx{\function}{\llamb} \gradfxtilde{\llamb}
,
\end{equation}
where
 $\gradftilde \doteq \matrixnonviolatingx{\llambconv}^\trsp \gradftilde $ and
$\LCScalefx{\function}{\llamb} \doteq \lbrack \matrixnonviolatingx{\llambconv}^\trsp \Scalefx{\function}{\llamb}^{-1}   \matrixnonviolatingx{\llambconv} \rbrack^{-1} $.

Further, if~$\function$ is twice continuously differentiable and~$\Scalefx{\function}{\cdot}$ is continuous  at~$\llambconv$,
%
%
then
\begin{equation} \label{globalrate}
\LCmu = 
\left[
\tildeidentity - \newPa{\Paok}{\mapAfx{\function}{\llambk{k}}} \LCScalefx{\function}{\llamb} \hessiantilde{\function}{\llamb}
\right]
\LClamb
+
\fullremainder{\LClamb}
,
\end{equation}
where $\hessiantilde{\smoothfunction}{\llambconv}\doteq \matrixnonviolatingx{\llambconv}^\trsp   \hessian{\smoothfunction}{\llambconv} \matrixnonviolatingx{\llambconv}$
and
$\fullremainder{\LClamb}=\zero{\norm{\LClamb}}$.
%
If~$\function$ is smooth and~$\Scalefx{\function}{\cdot}$ is continuously differentiable at~$\llambconv$, then the remainder rewrites as $\fullremainder{\LClamb}=\remainder{\LClamb} (\LClamb-\LCaccu)(\LClamb-\LCaccu)^\trsp$, where~$\remainder{\LClamb}$ is a function of the derivatives at~$\LClamb$ of~$\hessiantildefunction{\function}$ and~$\LCScalefx{\function}{\cdot}$ uniformly bounded in a neighborhood of~$0$.
\end{proposition}
\begin{proof}
Since the proposition is trivial when $\tildeMM=0$,
we suppose that $\tildeMM>0$.
Let
$\manifoldsetoflambdas \doteq \{ \llamb \in \setoflambdas \setst \setviolatingx{\llamb} = \setviolatingx{\llambconv} \} $,
and
$\LCsetoflambdas \doteq \{ \matrixnonviolatingx{\llambconv}^\trsp ( \llamb - \llambconv ) \setst \llamb \in \manifoldsetoflambdas \}$,
which is an open subset of $ \REAL{\tildeMM}$ by  continuity of the constraint functions.
Moreover, the function
$\bijectionfx{\LClamb} \doteq \llambconv + \matrixnonviolatingx{\llambconv} \LClamb$ 
is a bijection between~$\LCsetoflambdas$ and~$\manifoldsetoflambdas$, i.e.  $\manifoldsetoflambdas \equiv \{ \bijectionfx{\LCmu} \setst \LCmu \in \LCsetoflambdas \} $.
%
%
It follows 
 the assumptions on~$\llamb$ and~$\mmu$ that
\begin{align} 
\LCmu 
 & 
\refeq{(\ref{globalequation})}{=} 
\begin{array}{c}
\arg \min\nolimits_{\xi \in \LCsetoflambdas} \left\{  \gradfxtilde{\llamb}^\trsp ({\xi} - {\LClamb } )  
+ \frac{1 }{2}  
\specialnorm{ {\xi} - {\LClamb } }{ [
\newPa{\Paok}{\mapAfx{\function}{\llambk{k}}}
 \LCScalefx{\function}{\llamb} ]^{-1}}^2 \right\} 
\end{array}
\\    & 
\refeq{(\ref{equationprojection})}{=} 
   \LClamb - 
\newPa{\Paok}{\mapAfx{\function}{\llambk{k}}} \LCScalefx{\function}{\llamb} 
 \gradfxtilde{\llamb}
\label{manifoldthree}
\end{align}
where~(\ref{manifoldthree}) follows 
 from the fact  that~$\LCsetoflambdas$ is  an open set containing~$\LCmu$, thus~$\LCmu$ is the projection  on~$\LCsetoflambdas$ of only one point: itself.

The remaining statements follow directly  from~(\ref{unconstrainedequationnew}) and Taylor's theorem, by linear approximation at~$ \llambconv$ of the displacement $\displacement{\lamb}\doteq\mapGTXfx{\Scale}{\setoflambdas}{\function}{\llamb} - \llamb$.
\QEDA\end{proof}

\obsolete{
-----

 \begin{corollary}[Asymptotics of~$\mapGTXfunction{\Scale}{\genericset}$]\label{globalasymptotics}
\obsolete{
Let~$\llambconv$  be a solution of Problem~\ref{initialproblem} where~$\function$ is twice continuously differentiable and strict complementarity holds.
If a sequence~$\{\llambk{k}\}$  generated by Algorithm~\ref{scaledgradientprojectionmapping} with scaling~$\Scalefx{\function}{\cdot}$  continuous at~$\llamb$
converges towards~$\llambconv$, then
for large~$k$ one has 
$\llambk{k} = \llambconv + \matrixnonviolatingx{\llambconv} \LClambk{k}  $ and $\llambk{k+1} = \llambconv + \matrixnonviolatingx{\llambconv} \LClambk{k+1}  $ for some
$  \LClambk{k}, \LClambk{k+1} \in \REAL{\tildeMM} $, and 
\begin{equation} \label{globalrate}
\LClambk{k+1} = 
[\tildeidentity - \newPa{\Paokk{k}}{\mapAfx{\function}{\llambk{k}}} \LCScalefx{\function}{\llambk{k}} \hessiantilde{\function}{\llambk{k}}]
\LClambk{k}
+
\remainder{\llambconv} (\LClambk{k}-\LCaccuk{k})(\LClambk{k}-\LCaccuk{k})^\trsp
,
\end{equation}
where~$\newPa{\Paokk{k}}{\mapAfx{\function}{\llambk{k}}}$ is the step size at step~$k$, $\hessiantilde{\smoothfunction}{\llambconv}\doteq \matrixnonviolatingx{\llambconv}^\trsp   \hessian{\smoothfunction}{\llambconv} \matrixnonviolatingx{\llambconv}$, and he remainder~$\remainder{\llambconv}$ is finite.
}%
Let~$\llambconv$  be a solution of Problem~\ref{initialproblem} where~$\function$ is twice continuously differentiable and strict complementarity holds. Consider  Algorithm~\ref{scaledgradientprojectionmapping} with scaling~$\Scalefx{\function}{\cdot}$  continuous  at~$\llambconv$.
%
Then any vectors $\llamb,\mmu\in\setoflambdas$ such that $\mmu=\mapGTXfx{\Scale}{\setoflambdas}{\function}{\llamb}$ with step size~$\newPa{\Paok}{\mapAfx{\function}{\llamb}}$ 
and $\llamb = \llambconv + \matrixnonviolatingx{\llambconv} \LClamb  $, $\mmu = \llambconv + \matrixnonviolatingx{\llambconv} \LCmu  $ 
for some
$  \LClamb, \LCmu \in \REAL{\tildeMM} $,
satisfy
\begin{equation} \label{globalrate}
\LCmu = 
[\tildeidentity - \newPa{\Paok}{\mapAfx{\function}{\llambk{k}}} \LCScalefx{\function}{\llamb} \hessiantilde{\function}{\llamb}]
\LClamb
+
\fullremainder{\LClamb}
,
\end{equation}
where $\hessiantilde{\smoothfunction}{\llambconv}\doteq \matrixnonviolatingx{\llambconv}^\trsp   \hessian{\smoothfunction}{\llambconv} \matrixnonviolatingx{\llambconv}$
and
$\fullremainder{\LClamb}=\zero{\norm{\LClamb}^2}$.
%
Further, if~$\function$ is smooth and~$\Scalefx{\function}{\cdot}$  continuously differentiable at~$\llamb$, then the remainder rewrites as $\fullremainder{\LClamb}=\remainder{\LClamb} (\LClamb-\LCaccu)(\LClamb-\LCaccu)^\trsp$, where~$\remainder{\LClamb}$ is a function of the derivatives at~$\LClamb$ of~$\hessiantildefunction{\function}$ and~$\LCScalefx{\function}{\cdot}$ and uniformly bounded in a neighborhood of~$0$.
\end{corollary}
\begin{proof}
This result follows immediately  from~(\ref{unconstrainedequationnew}) and Taylor's theorem by linear approximation at~$ \llambconv$ of the displacement $\mapGTXfx{\Scale}{\setoflambdas}{\function}{\llamb} - \llamb$.
\QEDA\end{proof}

}%

\subsection{Parallel analysis and coordinate descent}\label{sectionparalleloptimization}

This section considers parallel implementations of Algorithm~\ref{scaledgradientprojectionmapping} for Problem~\ref{initialproblem}, where assumption is made that~$\setoflambdas$ is a Cartesian product set.
\begin{assumption}[Parallel analysis]\label{assumptionparallel}
The feasible set of Problem~\ref{initialproblem} is given by $\setoflambdas=\setoflambdasi{1}\times...\times\setoflambdasi{\nn}$,
where each~$\setoflambdasi{i}$ is a polyhedron in~$\RMMi{i}$ and $\MMi{1}+...+\MMi{\nn}=\MM$.
The Lipschitz continuity of~$\grad\function$ is considered  coordinate-wise and~(\ref{Lequation}) holds for
$\Lipschitz=\diag{\Lipschitzi{1},...,\Lipschitzi{\nn}}$.  
\end{assumption}
%
%
Assumption~\ref{assumptionparallel} implicitly defines a set  $\setofnodes=\{1,...,\nn\}$ of coordinate directions with respective dimensions $\MMi{1},...,\MMi{\nn}$, suggesting  parallel optimization by coordinate descent.

\subsubsection{Coordinate descent}
In this study the {optimization} of  $\GPfunction\in\setoffunctionsn{\MM}$
at a point
$\llamb\in\setoflambdas$ 
along a particular coordinate direction $i\in \setofnodes$ 
is symbolized by the function
$\maplocalix{\GPfunction}{i}{\llamb}
\in\setoffunctionsn{\MMi{i}}$ 
obtained from~$\GPfunctionx{\llamb}$ by fixing the other coordinates, i.e.
\begin{equation}\label{coordinatefunction}
\maplocalixy{\GPfunction}{i}{\llamb}{\mmu}
\doteq
\GPfunctionx{\llambi{1},...,\llambi{i-1},\mmu,\llambi{i+1},...,\llambi{\nn}}
,\hspace{5mm} \forall \mmu\in
\RMMi{i}, \  i\in\setofnodes
.
\end{equation}
Using~(\ref{coordinatefunction}),
\newPa{we   consider in each direction~$i$ a scaling mapping~$\Scalei{i}:\setoffunctionsn{\MMi{i}}\times\setoflambdasi{i}\functionto\setofScalesn{\MMi{i}}$ and}{we consider along every coordinate~$i$ a scaling mapping~$\Scalei{i}:\setoffunctionsn{\MMi{i}}\times\setoflambdasi{i}\functionto\setofScalesn{\MMi{i}}$
a step size rule~$\mapAi{i}:\setoffunctionsn{\MMi{i}}\times\setoflambdasi{i}\functionto (0,1]$, and}
define the associated coordinate gradient projection mapping
\begin{equation}\label{coordinategradientprojection}
\mapGifx{i}{\GPfunction}{\llamb}\doteq\mapGTXfx{\Scalei{i}}{\setoflambdasi{i}}{\maplocalix{\GPfunction}{i}{\llamb}}{\llambi{i}},
\hspace{5mm}\forall\llamb\in\setoflambdas,\ i\in\setofnodes.
\end{equation}
Based on~(\ref{coordinategradientprojection}), we formulate  a \emph{synchronous} coordinate descent algorithm, modeled on the \emph{Jacobi method}, as $\llambk{k+1}=\mapJfx{\function}{\llambk{k}}$%
, where
\begin{equation}\label{mapJfunction}
\mapJfx{\GPfunction}{\llamb}\doteq\vect{\mapGi{1},\mapGi{2},...,\mapGi{\nn}}\vect{\GPfunction,\llamb}
\hspace{5mm}\forall\llamb\in\setoflambdas.
\end{equation}
Notice in~(\ref{mapJfunction}) that the coordinate descents $\mapGi{1},...,\mapGi{\nn}$ are  applied simultaneously along the~$\nn$ directions.
Global convergence of~$\mapJfunction$ is conditioned by its quality to guarantee sufficient descent along the  considered function at each step, which, in the general case, requires not only synchronism for the  applications of $\mapGi{1},...,\mapGi{\nn}$, but also consensus at the global level on the choice of scaling matrices and step-sizes in each direction~\cite{bilennethesis,bertsekas97}.
An alternative is to process the coordinate descents sequentially, using the directional mappings
\begin{equation}
\maphatGifx{i}{\GPfunction}{\llamb}\doteq\vect{\llambi{1},...,\llambi{i-1},\mapGifx{i}{\GPfunction}{\llamb},\llambi{i+1},...,\llambi{\nn}},
\hspace{5mm}\forall \llamb\in\setoflambdas,\  i\in\setofnodes.
\end{equation}
A \emph{cyclic} coordinate descent algorithm can then be designed  by applying the coordinate descent mappings  in a predefined order as in the \emph{Gauss-Seidel method}, i.e. $\llambk{k+1}=\mapSfx{\function}{\llambk{k}}$, where
\begin{equation}
\mapSfx{\GPfunction}{\llamb}\doteq
(\maphatGi{\nn}\mappingcompo\cdots\mappingcompo\maphatGi{2}\mappingcompo\maphatGi{1})\vect{\GPfunction,\llamb},
\hspace{5mm}\forall\llamb\in\setoflambdas,
\end{equation}
and~$\mappingcompo$ denotes the composition operator, defined for any $i,j\in\setofnodes$ by
$(\maphatGi{i}\mappingcompo\maphatGi{j})\vect{\GPfunction,\llamb}\doteq\maphatGifx{i}{\GPfunction}{\maphatGifx{j}{\GPfunction}{\llamb}}$.
The global convergence of~$\mapSfunction$ is guaranteed by  approximate line search in each coordinate direction, i.e. 
\newPa{%
using, for  $i\in\setofnodes$ and  at every $\llamb\in\setoflambdas$, the step sizes $\mapAfx{\maplocalix{\function}{i}{\llamb}}{\llambi{i}}$ 
 specified by~(\ref{stocharmijoconditionglobalone}).
}{%
by setting $\mapAifx{i}{\function}{\llamb}\doteq\mapAfx{\maplocalix{\function}{i}{\llamb}}{\llambi{i}}$ for $\llamb\in\setoflambdas$ and $i\in\setofnodes$, where~$\mapAfunction$ is given by the Armijo rule~(\ref{stocharmijoconditionglobalone}),
}%
One shows that convergence is conserved when the mappings~$\mapGi{i}$ are applied in random order provided that each coordinate direction is visited an infinite number of times~\cite{bilennethesis}.
A \emph{random} coordinate descent algorithm is given by
$\llambk{k+1}=\mapRkfx{k}{\function}{\llambk{k}}$, where
\begin{equation}\label{randomimplementation}
\mapRkfx{k}{\GPfunction}{\llamb}\doteq
\maphatGifx{\Pseqk{k}}{\GPfunction}{\llamb},
\hspace{5mm}\forall\llamb\in\setoflambdas, \ k=0,1,2,...,
\end{equation}
and~$\{\Pseqk{k}\}_{k=0}^{\infty}$ is a sequence of coordinate directions randomly selected in~$\setofnodes$, so that each $\Pseqk{k}$ is a realization of a discrete random variable defined on a probability space~$\probabilityspace{\setofnodes}{2^{\setofnodes}}{\probadistrifunction}$, with $\probadistrifunction=\vect{\probadistri{1},...,\probadistri{\nn}}\in{(0,1)}^{\nn}$.
%
More sophisticated parallel implementations of~$\mapGfunction$ involving  (block-) coordinate selection routines (e.g. \emph{Gauss-Southwell} methods~\cite{tseng09}) can also be devised.

For any parallel implementation of Algorithm~\ref{scaledgradientprojectionmapping} 
 designed with coordinate scaling~$\{\Scalei{i}\}_{i\in\setofnodes}$, 
%
we  let 
$ 
\Scalefx{\function}{\llamb}
\doteq
\diag{\Scaleifx{1}{\maplocalix{\function}{1}{\llamb}}{\llambi{1}},...,\Scaleifx{\nn}{\maplocalix{\function}{\nn}{\llamb}}{\llambi{\nn}}}
$
for $\llamb\in\setoflambdas$.
%
Proposition~\ref{propositionefficiency} extends to coordinate descent, and the step sizes computed by%
~(\ref{stocharmijoconditionglobalone}) in  each direction  reduce to~$1$ 
 if~(\ref{conditionefficiencynonlinear}) holds~\cite{bilennethesis}.
If the scaling mappings are continuous  and~$\smoothfunction$ is twice continuously differentiable, 
the Hessian $\hessianfunction{\smoothfunction}=(\hessianifunction{\smoothfunction}{ij})$ may be regarded as a block matrix and decomposed into
$ \hessianfunction{\smoothfunction} =  \newDfunction{\smoothfunction} - \newLfunction{\smoothfunction} - \newLfunction{\smoothfunction}^\trsp
 $,
where $\newDfunction{\smoothfunction} \doteq  \diag{ \hessianifunction{\smoothfunction}{11},..., \hessianifunction{\smoothfunction}{\nn\nn} }$ is block diagonal and~$\newLfunction{\smoothfunction}$ strictly lower triangular.
The step sizes of the parallel algorithms then reduce to~$1$ in the vicinity of%
~$\llambconv$ if
$ 2{(1-\Psig)} \Scaleifx{i}{\maplocalix{\function}{i}{\llambconv}}{\llambconv}^{-1}  \matrixg  \hessiani{\smoothfunction}{ii}{\llambconv}
$ 
%
holds for $i\in\setofnodes$,
i.e. if
\begin{equation} \label{coordinateconditionefficiency}
2 (1-\Psig) \Scalefx{\function}{\llambconv}^{-1}  \matrixg \newD{\smoothfunction}{\llambconv}.
\end{equation}
%
%

\subsubsection{Asymptotics}

By extension of Proposition~\ref{identificationofthesaturatednonnegativityconstraints}~(\cite{bilennethesis}), one shows the existence of a reduced space where every  sequence generated by a parallel implementation of Algorithm~\ref{scaledgradientprojectionmapping} and converging to the solution~$\llambconv$
accumulates under strict complementarity for~$\function$ 
at~$\llambconv$ or, equivalently, strict complementarity for~$\maplocalix{\function}{i}{\llambconv}$ at~$\llambconvi{i}$ in every direction~$i\in\setofnodes$.
The reduced space at $\llambconv=\vect{\llambconvi{1},...,\llambconvi{\nn}}$  is a Cartesian product and the matrix~$\matrixnonviolatingx{\llambconv}$, introduced in Section~\ref{referencereducedspace}, takes the block-diagonal form
%
$ 
\matrixnonviolatingx{\llambconv} 
= 
\diag{\matrixnonviolatingix{1}{\llambconvi{1}} ,..., \matrixnonviolatingix{\nn}{\llambconvi{\nn}} } 
$, 
where the columns of~$\matrixnonviolatingix{i}{\llambconvi{i}}$ form an orthonormal basis of the reduced space at~$\llambconvi{i}$ in the coordinate direction~$i$. 
%

Under Assumption~\ref{assumptionparallel}, $\grad\smoothfunction$ rewrites as the $\nn$-dimensio\-nal composite 
composite vector $\grad\smoothfunction=\vect{\gradi{1}\smoothfunction,...,\gradi{\nn}\smoothfunction}$.
For $\llamb\in\setoflambdas$, we 
 define 
$\graditilde{i}\smooth{\llamb} \doteq \matrixnonviolatingix{i}{\llambconvi{i}}^\trsp \gradi{i}\smooth{\llamb} $ ($i\in\setofnodes$).
Similarly, setting
$\LCScalefx{\function}{\llamb} \doteq \lbrack \matrixnonviolatingx{\llambconv}^\trsp \Scalefx{\function}{\llamb}^{-1}   \matrixnonviolatingx{\llambconv} \rbrack^{-1} $
yields the block-diagonal form
$ 
\LCScalefx{\function}{\llamb}=\diag{\LCScaleifx{1}{\maplocalix{\function}{1}{\llamb}}{\llambi{1}},...,\LCScaleifx{\nn}{\maplocalix{\function}{\nn}{\llamb}}{\llambi{\nn}}}
$,
where the  diagonal elements are given by $\LCScaleifx{i}{\maplocalix{\function}{i}{\llamb}}{\llambi{i}} \doteq  \lbrack \matrixnonviolatingix{i}{\llambconvi{i}}^\trsp \Scaleifx{i}{\maplocalix{\function}{i}{\llamb}}{\llambi{i}}^{-1}   \matrixnonviolatingix{i}{\llambconvi{i}} \rbrack^{-1} $. 
%
%

In the reduced space at~$\llambconv$, a gradient projection
$
\mmu = \mapGifx{i}{\function}{\llamb}
$
with step size~$\Paok$ in direction $i\in\setofnodes$ at a point $\llamb\in\setoflambdas$ near~$\llambconv$
 reduces, by translation of Proposition~\ref{unconstraineddescent} into the coordinate descent framework~(\cite{bilennethesis}),
to
$\llamb = \llambconv + \matrixnonviolatingx{\llambconv} \LClamb  $ and $\mmu = \llambconvi{i} + \matrixnonviolatingix{i}{\llambconvi{i}} \LCmu  $ for some
$  \LClamb\in \REAL{\tildeMM} $ and $\LCmu \in \REAL{\tildeMMi{i}} $ satisfying
\begin{equation} \label{coordinateunconstrainedequation}
\LCmu = 
\LClambi{i} - 
\Paok\LCScaleifx{i}{\maplocalix{\function}{i}{\llamb}}{\llambi{i}} \graditilde{i} \smooth{\llamb},
\hspace{5mm}\forall i\in\setofnodes.
\end{equation}

%

If~$\function$ is twice continuously differentiable and~$\Scalefx{\function}{\cdot}$ is continuous  at~$\llambconv$, then~(\ref{coordinateunconstrainedequation}) asymptotically reduces to 
\begin{equation}\label{coordinateasymptotic}
\LCmu = 
[\tildeidentityi{i}- 
 \vect{0,...,0,\Paok\tildeidentityi{i},0,...,0}^\trsp \LCScalefx{\function}{\llambconv} \hessiantilde{\smoothfunction}{\llambconv}]
\LClamb
+\zero{\norm{\LClamb}}
, 
\hspace{2mm}\forall i\in\setofnodes,
\end{equation}
%
where $\hessiantilde{\smoothfunction}{\llamb}\doteq \matrixnonviolatingx{\llambconv}^\trsp   \hessian{\smoothfunction}{\llamb} \matrixnonviolatingx{\llambconv}$ and~$\tildeidentityi{i}$ denotes the identity matrix in~$\REAL{\tildeMMi{i}\times\tildeMMi{i}}$.
Similarly,  
%
we write
\mbox{$ 
\hessiantildefunction{\smoothfunction} =  \newtildeDfunction{\smoothfunction} - \newtildeLfunction{\smoothfunction} - \newtildeLfunction{\smoothfunction}^\trsp 
$},
where $ \newtildeD{\smoothfunction}{\llamb} \doteq \matrixnonviolatingx{\llambconv}^\trsp \newD{\smoothfunction}{\llamb} \matrixnonviolatingx{\llambconv}$, and 
$ \newtildeL{\smoothfunction}{\llamb} \doteq \matrixnonviolatingx{\llambconv}^\trsp \newL{\smoothfunction}{\llamb} \matrixnonviolatingx{\llambconv}$.%

\section{Asymptotic convergence rates}\label{sectionconvergencerates}

The developments of this section rely on the following assumption.
\obsolete{
Linear convergence can often be shown for  gradient descent methods 
under certain conditions which typically include strong convexity of the considered function~\cite{luo93theorem}. 
%
%
In the case of polyhedral feasible sets, the condition for linear local convergence reduces to positive curvature of the function at the point of convergence and strict complementarity. We make the following assumption.
}
\begin{assumption}\label{assumptiondifferentiability}
Problem~\ref{initialproblem} has a unique solution~$\llambconv$ where strict complementarity holds 
and in the vicinity of which~$\smoothfunction$ is twice continuously differentiable.
\end{assumption}
\emph{Terminology ---}
We derive asymptotic rates of convergence for the algorithms of Section~\ref{sectionparalleloptimization} by first-order sensitivity analysis around~$\llambconv$. Our aim is to find a matrix~$\LCMHfx{\function}{\llambconv}$ which satisfies an equation of the type $\genericresidual{\llambk{k+1}} =\LCMHfx{\function}{\llambconv}\genericresidual{\llambk{k}}+\zero{\genericresidual{\llambk{k}}}$ for some continuous function~$\genericresidualfunction$ and any sequence~$\{\llambk{k}\}$ generated by the considered algorithm.
If this equation holds for $\LCMHfx{\function}{\llambconv}\leq 1$, we say that~$\{\genericresidual{\llambk{k}}\}$ converges 
towards~$\genericresidual{\llambconv}$ \emph{with asymptotic rate}~$\LCMHfx{\function}{\llambconv}$. Convergence is called \emph{sublinear} if~$\LCMHfx{\function}{\llambconv}= 1$, and \emph{linear} if~$\LCMHfx{\function}{\llambconv}< 1$.
 If~$\LCMHfx{\function}{\llambconv}$ satisfies the inequality
$\genericresidual{\llambk{k+1}} \leq \LCMHfx{\function}{\llambconv}\genericresidual{\llambk{k}}+\zero{\genericresidual{\llambk{k}}}$
for any generated sequence and  the algorithm may produce a sequence for which the latter inequality holds with equality sign, then we speak of  
convergence \emph{with limiting asymptotic rate}~$\LCMHfx{\function}{\llambconv}$.
%
%
%

\obsolete{
\paragraph{Notation ---} Under Assumptions~\ref{assumptionparallel} and~\ref{assumptiondifferentiability}, the gradient of~$\smoothfunction$ rewrites as the $\nn$-dimensio\-nal composite vector $\grad\smoothfunction=\vect{\gradi{1}\smoothfunction,...,\gradi{\nn}\smoothfunction}$, and the Hessian $\hessianfunction{\smoothfunction}=(\hessianifunction{\smoothfunction}{ij})$ is regarded as a block matrix and decomposed into
$ \hessianfunction{\smoothfunction} =  \newDfunction{\smoothfunction} - \newLfunction{\smoothfunction} - \newLfunction{\smoothfunction}^\trsp
 $,
where $\newDfunction{\smoothfunction} \doteq  \diag{ \hessianifunction{\smoothfunction}{11},..., \hessianifunction{\smoothfunction}{\nn\nn} }$ is block diagonal and~$\newLfunction{\smoothfunction}$ strictly lower triangular.
In the reduced space at~$\llambconv$, we successively consider
$\hessiantilde{\smoothfunction}{\llamb}\doteq \matrixnonviolatingx{\llambconv}^\trsp   \hessian{\smoothfunction}{\llamb} \matrixnonviolatingx{\llambconv}$,
%
$ \newtildeD{\smoothfunction}{\llamb} \doteq \matrixnonviolatingx{\llambconv}^\trsp \newD{\smoothfunction}{\llamb} \matrixnonviolatingx{\llambconv}$, and 
$ \newtildeL{\smoothfunction}{\llamb} \doteq \matrixnonviolatingx{\llambconv}^\trsp \newL{\smoothfunction}{\llamb} \matrixnonviolatingx{\llambconv}$%
. Hence,
\mbox{$ 
\hessiantildefunction{\smoothfunction} =  \newtildeDfunction{\smoothfunction} - \newtildeLfunction{\smoothfunction} - \newtildeLfunction{\smoothfunction}^\trsp 
$}. 

}%

Proposition~\ref{localconvergenceprop} extends to polyhedral sets and arbitrary scaling a result derived in~\cite{crouzeix92} for a  cyclic coordinate descent algorithm used in the non-negative orthant with coordinate-wise Newton scaling.
The proof provided in this paper is arguably simpler and the requirements less restrictive. The \emph{spectral radius} of any matrix $\genericnormscale\in\REAL{\notn\times \notn}$, defined as the supremum among the absolute values of the eigenvalues of~$\genericnormscale$, is denoted by~$\spectralradius{\genericnormscale}$. 
%
\begin{proposition} [{Asymptotic convergence rate of~$\mapSfunction$}]  
\label{localconvergenceprop}
Let Assumptions~\ref{assumptionparallel} and~\ref{assumptiondifferentiability} hold for Problem~\ref{initialproblem}.
%
%
Consider the cyclic algorithm $\llambk{k+1}=\mapSfx{\function}{\llambk{k}}$  implemented with 
the step-size selection rule~(\ref{stocharmijoconditionglobalone}) and with 
scaling~$\{\Scaleifx{i}{\maplocalix{\function}{i}{\llamb}}{\llambi{i}}\}_{i\in\setofnodes}$ 
continuous at~$\llambconv$, 
and satisfying  
condition~(\ref{coordinateconditionefficiency}) for the step sizes.
Any sequence~$\{\llambk{k}\}$ generated by the algorithm  converges towards~$\llambconv$ with  asymptotic rate~$\matrixnonviolatingx{\llambconv}\LCMfx{\function}{\llambconv}$, where
\obsolete{
Consider Problem~\ref{initialproblem} under Assumption~\ref{assumptionparallel} and the cyclic algorithm $\llambk{k+1}=\mapSfx{\function}{\llambk{k}}$  implemented with 
the step-size selection rule~(\ref{stocharmijoconditionglobalone}) and with 
scaling mappings~$\Scaleifx{i}{\maplocalix{\function}{i}{\llamb}}{\llambi{i}}$  continuously differentiable in~$\llamb$ in a neighborhood of~$\llambconv$ ($i\in\setofnodes$).
\obsolete{
and such that
\begin{equation} \label{ostrowcondition}
2 \LCScalefx{\function}{\llambconv}^{-1} \matrixg \newtildeD{\smoothfunction}{\llambconv}  .
\end{equation} 
}
%
%
%
%
Let~$\{\llambk{k}\}$ be a sequence  generated by the algorithm and converging  to a solution~$\llambconv$  where  strict complementarity holds,
$\smoothfunction$ is twice continuously differentiable
 $\hessianfunction{\smoothfunction}$ is positive definite and continuous, and  condition~(\ref{conditionefficiency}) for unit step sizes near~$\llambconv$ is met.
%
%
%
%
Then~$(\llambk{k})$ converges linearly to~$\llambconv$ with rate
}
\begin{equation} \label{LCMrate}
\LCMfx{\function}{\llambconv} = \left\lbrack \LCScalefx{\function}{\llambconv}^{-1} -  \newtildeL{\smoothfunction}{\llambconv} \right\rbrack^{-1}  \left\lbrack \LCScalefx{\function}{\llambconv}^{-1} - \newtildeD{\smoothfunction}{\llambconv} + \newtildeL{\smoothfunction}{\llambconv}^\trsp \right\rbrack,
\end{equation}
with~$\spectralradius{\LCMfx{\function}{\llambconv}}\leq1$%
%
, while~$\singlenorm{\smooth{\llambk{k}}-\smooth{\llambconv}}$ vanishes  with limiting asymptotic rate
$
%
\spectralradius{\LCMfx{\function}{\llambconv}}^2
$.
If~$\hessian{\smoothfunction}{\llambconv}$ is positive definite, then~$\spectralradius{\LCMfx{\function}{\llambconv}}<1$.
\end{proposition}
\begin{proof}
%
%
%
%
%
%
It follows from the assumptions, Proposition~\ref{propositionefficiency}  and~(\ref{coordinateunconstrainedequation}), that one can find a~$\ksc < \infty$ such that, for any $k\geq\ksc$,  we have
$\llambk{k}=\llambconv+ \matrixnonviolatingx{\llambconv} \LClambk{k}$ and
$\llambk{k+1}=\llambconv+ \matrixnonviolatingx{\llambconv} \LClambk{k+1}$
for some $\LClambk{k},\LClambk{k+1} \in \REAL{\tildeMM}$,
%
%
%
%
%
with
\begin{align}\label{taylorresult}
 \LClambk{k+1} 
=
\LCMGifx{\nn}{\function}{\llambconv}
\,
\LCMGifx{\nn-1}{\function}{\llambconv}
\cdots
\LCMGifx{1}{\function}{\llambconv}
\LClambk{k}  
+ \zero{\norm{ \LClambk{k}  }}
, 
\end{align}
%
where 
\begin{equation}\label{shrinkage}
\LCMGifx{i}{\function}{\llambconv}\doteq \tildeidentity - \diag{0,...,0,\tildeidentityi{i},0,...,0} \LCScalefx{\function}{\llambconv} \hessiantilde{\smoothfunction}{\llambconv}
,\hspace{5mm} \forall i\in\setofnodes,
\end{equation} 
embodies the effect of a  gradient projection along 
coordinate direction~$i$.
%
\obsolete{
It is easily seen that each 
eigenvalue of~$\LCMGifx{i}{\function}{\llambconv}$ is either~$1$ or  takes the form~$1-\lambda$ for some eigenvalue~$\lambda$  of~$\LCScaleifx{i}{\maplocalix{\function}{i}{\llambconv}}{\llambconvi{i}}
\hessiantildei{\smoothfunction}{ii}{\llambconv}$.
%
It follows from~(\ref{coordinateconditionefficiency}) that
$\norm{\LCMGifx{i}{\function}{\llambconv}\LClamb}\leq\norm{\LClamb}$ for all $\LClamb\in\REAL{\tildeMM}$, $i\in\setofnodes$.
%
}%
Applying  Lemma~\ref{lemmamatrices} from the Appendix yields
$
\LCMGifx{\nn}{\function}{\llambconv}
\cdots
\LCMGifx{1}{\function}{\llambconv}
=
\LCMfx{\function}{\llambconv}
$.
\obsolete{
By induction we find 
$\norm{\LCMfx{\function}{\llambconv}\LClamb}\leq\norm{\LClamb}$ for all $\LClamb\in\REAL{\tildeMM}$. Hence,
$\spectralradius{\LCMfx{\function}{\llambconv}} 
\leq1$~\cite{varga00}.
}%

\obsolete{
---

Let~$\tildeidentityi{i}$ denote the identity matrix in~$\REAL{\tildeMMi{i}}$ and define 
$ 
\tildematrixidentityi{i} \doteq \diag{\tildeidentityi{1},...,\tildeidentityi{i-1},0,...,0} 
$. 
Let $\LClamb$ be a point of $\REAL{\tildeMM}$,  $\LCdfunction\in\REAL{\tildeMM}$ 
symbolize a cyclic displacement from~$\LClamb$ in~$ \REAL{\tildeMM}$, 
and $\mmudyi{i}{\LCdfunction}{\LClamb} \doteq \llambconv + \matrixnonviolatingx{\llambconv} ( \LClamb + \tildematrixidentityi{i} \, \LCdfunction )
$
reflect the intermediary point obtained in~$\setoflambdas$ after  
 displacement in the first~$i-1$ directions ($1\leq i \leq \nn$). 
%
%
Consider the $\REAL{\tildeMM}\times\REAL{\tildeMM}\functionto\REAL{\tildeMM}$ function~$\LCHfunction=\vect{\LCHifunction{1},...,\LCHifunction{\nn}}$ such that 
\begin{equation}
\LCHi{i}{\LCdfunction}{\LClamb} = \lbrack \Pssi{i} 
\LCScaleifx{i}{\maplocalix{\function}{i}{\mmudyi{i}{\LCdfunction}{\LClamb}}}{\LClambi{i}}
\rbrack^{-1} \LCdifunction{i} +  \graditilde{i} \smooth{\mmudyi{i}{\LCdfunction}{\LClamb}} , \hspace{5mm} \forall i \in \setofnodes .
\end{equation}
The implicit function theorem claims the existence of a $\REAL{\tildeMM}\functionto\REAL{\tildeMM}$ function $\LCd{\LClamb}$ continuously differentiable in a neighborhood $\LChood\subset\REAL{\tildeMMi{i}}$ 
of~$\llambconv$ such that~$\LCd{\llambconv}=0$ and~$\LCH{\LCd{\LClamb}}{\LClamb}=0$ for all~$\LClamb\in\LChood$. 
In fact, 
$\LCd{\LClamb}$ is the displacement from~$ \LClamb$ in the subspace caused by $\mapSfx{\function}{\llambconv + \matrixnonviolatingx{\llambconv} \LClamb}$,
and such that 
\begin{equation}  
\LCdi{i}{\LClamb} = - \Pssi{i}  
\LCScaleifx{i}{\maplocalix{\function}{i}{\mmudyi{i}{\LCd{\LClamb}}{\LClamb}}}{\LClambi{i}}
 \graditilde{i} \smooth{\mmudyi{i}{\LCd{\LClamb}}{\LClamb}}  , \hspace{5mm} \forall \LClamb\in\LChood,\ i \in \setofnodes .
\end{equation}
Consider the function $\LCHbis{\LClamb} \doteq \LCH{\LCd{\LClamb}}{\LClamb}$. By differentiation of~$\LCHbisfunction$ at~$\llambconv$ we find
\begin{equation}
\jacobian{\LCHbisfunction}{\llambconv} = \lbrack  \LCScalefx{\function}{\llambconv}^{-1} - \newtildeL{\smoothfunction}{\llambconv} \rbrack \jacobian{\LCdfunction}{\llambconv} + \hessiantilde{\smoothfunction}{\llambconv} ,
\end{equation}
where~$\jacobianfunction{\LCHbisfunction}$ and~$\jacobianfunction{\LCdfunction}$ {denote} the  Jacobians of~$\LCHbisfunction$ and~$\LCdfunction$, respectively.
From~$\jacobian{\LCHbisfunction}{\llambconv} = 0$ 
follows
\begin{equation} \label{jacobiand}
\jacobian{\LCdfunction}{\llambconv} = - \lbrack \LCScalefx{\function}{\llambconv}^{-1} - \newtildeL{\smoothfunction}{\llambconv} \rbrack^{-1}  \hessiantilde{\smoothfunction}{\llambconv} .
\end{equation}

It follows from the assumptions, Proposition~\ref{propositionefficiency}  and~(\ref{coordinateunconstrainedequation}), that one can find a~$\ksc < \infty$ such that, for any $k\geq\ksc$,  we have
$\llambk{k}=\llambconv+ \matrixnonviolatingx{\llambconv} \LClambk{k}$ and
$\llambk{k+1}=\llambconv+ \matrixnonviolatingx{\llambconv} \LClambk{k+1}$
with
$\LClambk{k+1} = \LClambk{k} + \LCd{\LClambk{k}}$
for some $\LClambk{k},\LClambk{k+1} \in \REAL{\tildeMM}$.
%
%
%
%
%
Taylor's theorem and $\LCd{\llambconv}=0$ yield, for $k\geq\ksc$,
\begin{align}\label{taylorresult}
 \LClambk{k+1} 
 =
\LClambk{k} +  \jacobian{\LCdfunction}{\llambconv}  \LClambk{k}  
+ \zero{\norm{ \LClambk{k}  }} 
=
\LCMfx{\function}{\llambconv}  \LClambk{k}  
+ \zero{\norm{ \LClambk{k}  }}. 
\end{align}
%
%
%
%
%
%
%
}%

%
%
%
%
%
Consider the sequence of function values~$\{\functionx{\llambk{k}}\}$.
It follows from Proposition~\ref{identificationofthesaturatednonnegativityconstraints} and~(\ref{KKTsc}) that, for~$k$ large enough, we have $\grad\smooth{\llambconv}^\trsp(\llambk{k}-\llambconv)=0$. Setting
$
\Nlambk{k}
\doteq
\hessiantilde{\smoothfunction}{\llambconv}^{\frac{1}{2}}
\LClambk{k}
$
and
$
\NMfx{\function}{\llambconv}
\doteq
\hessiantilde{\smoothfunction}{\llambconv}^{\frac{1}{2}}
\LCMfx{\function}{\llambconv}
\hessiantilde{\smoothfunction}{\llambconv}^{-\frac{1}{2}}
$,
the Taylor theorem yields
\begin{align}
\notag
\begin{array}{c}
\smooth{\llambk{k+1}}
-
\smooth{\llambconv}
\end{array}
\hspace{-8mm}
&
\\
\label{equationbegin}
&
\begin{array}{c}
=
\frac{1}{2}(\llambk{k+1}-\llambconv)^\trsp\hessian{\smoothfunction}{\llambconv}(\llambk{k+1}-\llambconv)
+\zero{\norm{\llambk{k+1}-\llambconv}^2}
\end{array}
\\
&
\begin{array}{c}
=\label{intermediate}
\frac{1}{2}(\LClambk{k+1})^\trsp\hessiantilde{\smoothfunction}{\llambconv}\LClambk{k+1}
+\zero{\norm{\LClambk{k+1}}^2}
\end{array}
\\
&
\begin{array}{c}
\refeq{\!\!(\ref{taylorresult})\!\!}{=}
\frac{1}{2}[\LCMfx{\function}{\llambconv}\LClambk{k}]^\trsp\hessiantilde{\smoothfunction}{\llambconv}\LCMfx{\function}{\llambconv}\LClambk{k}
+\zero{\norm{\LClambk{k}}^2}
\end{array}
\\
&
\begin{array}{c}
=\label{lastequality}
\frac{1}{2}
\norm{\NMfx{\function}{\llambconv}\Nlambk{k}}^2
+\zero{\norm{\LClambk{k}}^2}
\end{array}
\\
&
\begin{array}{c}
\leq
\frac{1}{2}
\spectralradius{\NMfx{\function}{\llambconv}}^2
\norm{\Nlambk{k}}^2
+\zero{\norm{\LClambk{k}}^2}
\end{array}
\\
&
\begin{array}{c}
\refeq{\!\!(\ref{intermediate})\!\!}{=}
\spectralradius{\NMfx{\function}{\llambconv}}^2
(\functionx{\llambk{k}}-\functionx{\llambconv})
+\zero{\norm{\llambk{k}-\llambconv}^2}
.
\end{array}
\label{equationend}
\end{align}
%
\obsolete{
Noting that 
\obsolete{
$\spectralradius{\LCMfx{\function}{\llambconv}^\trsp
\hessiantilde{\smoothfunction}{\llambconv}
\LCMfx{\function}{\llambconv}
\hessiantilde{\smoothfunction}{\llambconv}^{-1}}
=
\spectralradius{\NMfx{\function}{\llambconv}^\trsp\NMfx{\function}{\llambconv}}\leq\spectralradius{\NMfx{\function}{\llambconv}}^2=\spectralradius{\LCMfx{\function}{\llambconv}}^2<1$,
}%
$\spectralradius{\NMfx{\function}{\llambconv}}=\spectralradius{\LCMfx{\function}{\llambconv}}<1$
completes the proof.
}%
%
%
\obsolete{%
Using~$\LCd{\llambconv}=0$ and~(\ref{jacobiand}), the Taylor theorem gives
\begin{align}
 \LClambk{k+1} 
   & =  \LClambk{k} +  \LCd{\LClambk{k}}
\\ & = \LClambk{k} +  \jacobian{\LCdfunction}{\llambconv} \LClambk{k} + \functionhx{\LClambk{k}} 
\\ & = \LCMfx{\function}{\llambconv}  \LClambk{k}  + \functionhx{\LClambk{k}} , \label{finalequationM}
\end{align}
with $\functionhx{\llamb} =  \zero{\norm{\llamb  }} $.
We can rewrite~$ \LCMfx{\function}{\llambconv} $  as
\begin{equation}
\LCMfx{\function}{\llambconv} = ( \LCostrowD{\llambconv} -  \LCostrowE{\llambconv} )^{-1}   \LCostrowE{\llambconv}^\trsp  ,
\end{equation}
where~$ \LCostrowD{\llambconv} =  2 \LCScalefx{\function}{\llambconv}^{-1} - \newtildeD{\smoothfunction}{\llambconv} $ and~$  \LCostrowE{\llambconv} = \LCScalefx{\function}{\llambconv}^{-1} - \newtildeD{\smoothfunction}{\llambconv} + \newtildeL{\smoothfunction}{\llambconv}$.
Noting that $ \LCostrowD{\llambconv} -  \LCostrowE{\llambconv} -  \LCostrowE{\llambconv}^\trsp =  \hessiantilde{\smoothfunction}{\llambconv}$ is positive definite and~$(\LCostrowD{\llambconv} - \LCostrowE{\llambconv})$ is nonsingular, the {Ostrowski-Reich theorem} (see \cite{reich49,ostrowski54} and~\cite[Theorem 3.12]{varga00}) states that~$\spectralradius{\LCMfx{\function}{\llambconv}} < 1$ if~$  \LCostrowD{\llambconv} \matrixg 0 $, i.e. if
\begin{equation} \label{ostrowcondition}
2 \LCScalefx{\function}{\llambconv}^{-1} - \newtildeD{\smoothfunction}{\llambconv} \matrixg 0 .
\end{equation} 
}%

We now characterize~$\spectralradius{\NMfx{\function}{\llambconv}}$.
First assume now that~$\hessian{\smoothfunction}{\llambconv}$ is positive definite.
Observe that
$ 
\LCMfx{\function}{\llambconv} = ( \LCostrowD{\llambconv} -  \LCostrowE{\llambconv} )^{-1}   \LCostrowE{\llambconv}^\trsp  
$, 
where~$ \LCostrowD{\llambconv} =  2 \LCScalefx{\function}{\llambconv}^{-1} - \newtildeD{\smoothfunction}{\llambconv} $ and~$  \LCostrowE{\llambconv} = \LCScalefx{\function}{\llambconv}^{-1} - \newtildeD{\smoothfunction}{\llambconv} + \newtildeL{\smoothfunction}{\llambconv}$.
Noting that $ \LCostrowD{\llambconv} -  \LCostrowE{\llambconv} -  \LCostrowE{\llambconv}^\trsp =  \hessiantilde{\smoothfunction}{\llambconv}$ is positive definite and~$(\LCostrowD{\llambconv} - \LCostrowE{\llambconv})$ is nonsingular, the \emph{Ostrowski-Reich theorem}~\cite[Theorem 3.12]{varga00} 
states that~$\spectralradius{\LCMfx{\function}{\llambconv}} < 1$ if~$  \LCostrowD{\llambconv} \matrixg 0 $, i.e. 
if
\begin{equation} \label{ostrowcondition}
2 \LCScalefx{\function}{\llambconv}^{-1} \matrixg \newtildeD{\smoothfunction}{\llambconv}  .
\end{equation} 
%
Because~(\ref{coordinateconditionefficiency}) implies~(\ref{ostrowcondition}), we infer  that $\spectralradius{\LCMfx{\function}{\llambconv}} < 1$, and the algorithm converges linearly.

If, however, $\hessian{\smoothfunction}{\llambconv}$ is only positive semi-definite, then by computing under~(\ref{coordinateconditionefficiency}) the asymptotic rate for the function $\smooth{\llamb}+\frac{\epsilon}{2}\norm{\llamb}^2$ with Hessian $\hessianfunction{\smoothfunction}+\epsilon$ and letting $\epsilon\to 0$, we find $\spectralradius{\NMfx{\function}{\llambconv}}\leq 1$ by continuity of the eigenvalues of~$\NMfx{\function}{\llambconv}$ with respect to~$\hessian{\smoothfunction}{\llambconv}$, which completes the proof.
\QEDA\end{proof}
\obsolete{
\begin{remark}
In the conditions of Proposition~\ref{localconvergenceprop}, the local convergence of $\llambk{k+1}=\mapSfx{\function}{\llambk{k}}$ 
resembles that of the procedure which consists of solving the equation
$\hessiantilde{\smoothfunction}{\llambconv} \LClamb = 0 $
using the \emph{Gauss-Seidel} method~$\LClambk{k+1} = \LCMfx{\function}{\llambconv} \LClambk{k} $.
\end{remark}
}%
\begin{remark}\label{remarklocalcheck}
Since condition~(\ref{stocharmijoconditionglobalone}) 
is the  conjunction of~$\nn$ conditions  verifiable along individual directions, i.e. $ 2 (1-\Psig) \Scaleifx{i}{\maplocalix{\function}{i}{\llambconv}}{\llambconvi{i}}^{-1} \matrixg \hessiani{\smoothfunction}{ii}{\llambconv}$ for $i=\in\setofnodes$, the cyclic algorithm~$\mapSfunction$ is an attractive candidate for distributed optimization.
\end{remark}
\obsolete{
\begin{remark}
Observe that~(\ref{ostrowcondition}) holds if~$\mapSfunction$ is implemented with the Armijo rule~(\ref{stocharmijoconditionglobalone}) and condition~(\ref{conditionefficiency}) for asymptotically trivial step-size selection is satisfied.
\end{remark}
}
%
%
%
\obsolete{
\begin{remark}\label{remarkcoordinaterates}
It follows from Lemma~\ref{lemmamatrices} in the Appendix
that~(\ref{LCMrate}) rewrites as
\begin{equation}\label{LCMratebis}
\LCMfx{\function}{\llambconv}
=
\LCMGifx{\nn}{\function}{\llambconv}
\,
\LCMGifx{\nn-1}{\function}{\llambconv}
\cdots
\LCMGifx{1}{\function}{\llambconv},
\end{equation}
where $\LCMGifx{i}{\function}{\llambconv}\doteq \tildeidentity - \diag{0,...,0,\tildeidentityi{i},0,...,0} \LCScalefx{\function}{\llambconv} \hessiantilde{\smoothfunction}{\llambconv}$---in which~$\tildeidentity$ and~$\tildeidentityi{i}$ respectively denote the identity matrices in~$\REAL{\tildeMM}$ and~$\REAL{\tildeMMi{i}}$---symbolizes the error shrinkage resulting from a gradient projection along coordinate direction~$i$, which can be computed from~(\ref{LCMrate}) by using the scaling strategy $\Scaleifx{j}{\function}{\llamb}\equiv\epsilon\identityi{j}$ for every $j\neq i$ and $\llamb\in\setoflambdas$, and letting~$\epsilon\downarrow 0$. It is easily seen that  each 
eigenvalue of~$\LCMGifx{i}{\function}{\llambconv}$ takes the form~$1-\lambda$ for some eigenvalue~$\lambda$  of~$\LCScaleifx{i}{\maplocalix{\function}{i}{\llambconv}}{\llambconvi{i}}
\hessiantildei{\smoothfunction}{ii}{\llambconv}$.
%
Observe from~(\ref{ostrowcondition}) that~$\spectralradius{\LCMfx{\function}{\llambconv}} < 1$ iff~$\spectralradius{\LCScaleifx{i}{\maplocalix{\function}{i}{\llambconv}}{\llambconvi{i}}\hessiantildei{\smoothfunction}{ii}{\llambconv}} < 1$ for $i\in\setofnodes$, which is the condition for the convergence of all the~$\nn$ directional optimization processes taken separately along their own coordinate direction (cf. Section~\ref{sectionJacobi}).
\end{remark}
}%
\begin{remark}[Coordinate-wise Newton scaling]\label{sectionnewton}
\shorten{}{%
Propositions~\ref{propositionefficiency} and~\ref{localconvergenceprop} show the point of scaling the gradient projections according to the curvature of the function. In particular, when 
}%
When Newton scaling is used in each direction, i.e. $\Scaleifx{i}{\maplocalix{\function}{i}{\llamb}}{\llambi{i}}
=\hessiani{\smoothfunction}{ii}{\llamb}^{-1}$  
for $i=1,...,\nn$  or, equivalently,   $\Scalefx{\function}{\llamb}=\newD{\smoothfunction}{\llamb}^{-1}$, (\ref{conditionefficiency}) holds at the point of convergence,  and the asymptotic convergence rate~(\ref{LCMrate}) reduces to
$
\LCMfx{\function}{\llambconv} = 
\lbrack \newtildeD{\smoothfunction}{\llambconv} -  \newtildeL{\smoothfunction}{\llambconv} 
\rbrack^{-1}  \newtildeL{\smoothfunction}{\llambconv}^\trsp 
$.
 \label{sectionnewton}
\shorten{}{
Since the inversion of the $\hessianifunction{\smoothfunction}{ii}$ matrices can be computationally demanding, it is sometimes preferable to consider approximations of the inverse Hessians, using for instance diagonal scaling~\cite{bertsekas76,bilennethesis} or quasi-Newton methods~\cite{nocedal99,bertsekas99}.
}%
\end{remark}
%
%
\begin{remark}\label{remarkswitchingsystem}
In the case when condition~(\ref{conditionefficiency}) is not met and~$\mapAfx{\GPfunction}{\cdot}$ is discontinuous at~$\llambconv$, then $\mapSfunction$ then proves to converge locally like a stable discrete-time switching system defined by a rate set~$\{\LCMkfx{\switchfunction}{\function}{\llambconv} \}_{\switchfunction\in\Switch}$ with $\spectralradius{\LCMkfx{\switchfunction}{\function}{\llambconv}} \leq 1$ (or $\spectralradius{\LCMkfx{\switchfunction}{\function}{\llambconv}} < 1$ if~$\hessian{\smoothfunction}{\llambconv}$  is positive definite) for all $\switchfunction\in\Switch$, 
reducing for large~$k$ to
$ 
\LClambk{k+1} 
=
\LCMkfx{\switch{\LClambk{k}}}{\function}{\llambconv}  \LClambk{k}  + \zero{\norm{\LClambk{k}   }} 
$, 
where~$\switch{\cdot}$ is a switching function.
 \end{remark}

\obsolete{
\begin{remark}[Extensions]
When condition~(\ref{conditionefficiency}) for asymptotic efficiency of the step size selection rule is not satisfied, it is not always possible to derive an asymptotic rate of convergence for~$\mapSfunction$ if~$\mapAfx{\GPfunction}{\cdot}$ is discontinuous at~$\llambconv$. The algorithm then proves to converge locally like a stable discrete-time switching system defined by a rate set~$\{\LCMkfx{\switchfunction}{\function}{\llambconv} \}_{\switchfunction\in\Switch}$ with $\spectralradius{\LCMkfx{\switchfunction}{\function}{\llambconv}} < 1$ for all $\switchfunction\in\Switch$, 
reducing for large~$k$ to
\begin{equation}\label{switchingsystem}
\LClambk{k+1} 
=
\LCMkfx{\switch{\LClambk{k}}}{\function}{\llambconv}  \LClambk{k}  + \zero{\norm{\LClambk{k}   }} ,
\end{equation}
where~$\switch{\cdot}$ is a switching function.
\end{remark}
 
Local convergence as a switching system of the type~(\ref{switchingsystem}) can also be derived
for the case when~$\smoothfunction$ is not twice continuously differentiable at the point of convergence by considering the directional derivatives of~$\grad\smoothfunction$.

}

\subsection{Synchronous implementations}\label{sectionJacobi}

When implemented with identical step sizes in all coordinate directions,  the synchronous algorithm~$\mapJfunction$ proves to be equivalent to~$\mapGfunction$ endowed with block-diagonal scaling. We directly infer from~(\ref{globalrate}) the asymptotic convergence rate of~$\mapJfunction$, or by setting~$\nn\equiv1$ in Proposition~\ref{LCMrate}.

\begin{proposition} [{Asymptotic convergence rate of~$\mapJfunction$}]  
\label{localconvergencepropJacobi}
Let Assumptions~\ref{assumptionparallel} and~\ref{assumptiondifferentiability} hold for Problem~\ref{initialproblem}.
%
%
Consider the synchronous algorithm $\llambk{k+1}=\mapSfx{\function}{\llambk{k}}$  implemented %
with step size~$1$
in every coordinate direction and with 
scaling~$\{\Scaleifx{i}{\maplocalix{\function}{i}{\llamb}}{\llambi{i}}\}_{i\in\setofnodes}$ 
continuous at~$\llambconv$.
If~$\{\llambk{k}\}$ in a sequence generated by the algorithm  converging  towards~$\llambconv$, then~$\{\llambk{k}\}$ converges with asymptotic rate
\obsolete{
Consider Problem~\ref{initialproblem} under Assumption~\ref{assumptionparallel} and 
the synchronous algorithm $\llambk{k+1}=\mapSfx{\function}{\llambk{k}}$  implemented 
with unit step size $\Paok=1$ in all coordinate directions.
\obsolete{
and such that
\begin{equation} \label{ostrowcondition}
2 \LCScalefx{\function}{\llambconv}^{-1} \matrixg \newtildeD{\smoothfunction}{\llambconv}  
\end{equation} 
}
%
%
%
%
Let~$\{\llambk{k}\}$ be a sequence  generated by the algorithm which converges  to a solution~$\llambconv$  where  strict complementarity holds and~$\hessianfunction{\smoothfunction}$ is positive definite and continuous.
%
%
%
%
Then~$(\llambk{k})$ converges 
to~$\llambconv$ at  rate~$\matrixnonviolatingx{\llambconv}\LCMJfx{\function}{\llambconv}$, where
}%
\begin{equation} \label{LCMJrate}
\LCMJfx{\function}{\llambconv} = 
\tildeidentitysymbol
- \LCScalefx{\function}{\llambconv}
\hessiantilde{\smoothfunction}{\llambconv} 
.\end{equation}
\end{proposition}
\begin{remark}
In contrast with Remark~\ref{remarklocalcheck},  global and linear convergence of~$\mapJfunction$ may be difficult to assess  by inspection along individual directions.
In the particular case when $\MMi{i}=1$ for $i\in\setofnodes$ 
and~(\ref{ostrowcondition}) is satisfied, then~$\spectralradius{\LCMJfx{\function}{\llambconv}}$ is jointly characterized by Proposition~\ref{localconvergenceprop} and the \emph{Stein-Rosenberg theorem}~\cite[Theorem~3.8]{varga00}, which claim   that either $\spectralradius{\LCMfx{\function}{\llambconv}}=\spectralradius{\LCMJfx{\function}{\llambconv}}=1$, or
 $\spectralradius{\LCMfx{\function}{\llambconv}}<\spectralradius{\LCMJfx{\function}{\llambconv}}<1$
when~$\hessian{\smoothfunction}{\llambconv}$ is positive definite, in which case
any convergent sequence~$\{\llambk{k}\}$ generated by the algorithm~$\mapJfunction$   converges linearly. 
%
Notice that convergence is then asymptotically faster for the cyclic implementation~$\mapSfunction$ than for the synchronous implementation~$\mapJfunction$.
\end{remark}

\paragraph{Synchronous algorithms based on approximations of the Newton method ---}

A particular  approach explored e.g. in~\cite{zargham12paper,zargham2013accelerated,mokhtari15} 
is to find a compromise between the computational and organizational attractiveness of the  coordinate descent methods, which set $\Scalefx{\smoothfunction}{\llamb}\equiv\newD{\smoothfunction}{\llamb}^{-1}$ for $\llamb\in\setoflambdas$ under strong convexity of~$\smoothfunction$ and  converge linearly, 
%
and the quadratic convergence of the centralized Newton method, for which $\Scalefx{\smoothfunction}{\llamb}\equiv\hessian{\smoothfunction}{\llamb}^{-1}$. 
In these studies
~$\hessian{\smoothfunction}{\llamb}$ is assumed to be sparse and such that  the quantity
$\Scaleterm{\smoothfunction}{\llamb}\doteq
\newD{\smoothfunction}{\llamb}^{-\frac{1}{2}}[\newL{\smoothfunction}{\llamb}+\newL{\smoothfunction}{\llamb}^\trsp]\newD{\smoothfunction}{\llamb}^{-\frac{1}{2}}
$
can be computed in a distributed manner,
while
the inverse of the Hessian of~$\smoothfunction$ rewrites as the series
\begin{equation}\label{taylorinversehessian}
\begin{array}{c}
\hessian{\smoothfunction}{\llamb}^{-1}
=
\newD{\smoothfunction}{\llamb}^{-\frac{1}{2}}
[
\sum\nolimits_{t=0}^{\infty}
\Scaleterm{\smoothfunction}{\llamb}^{t} ]
\newD{\smoothfunction}{\llamb}^{-\frac{1}{2}}
\end{array}
\end{equation}
provided that
$\spectralradius{\Scaleterm{\smoothfunction}{\llamb}}<1$,
which holds under a strict diagonal dominance condition for~$\newD{\smoothfunction}{\llamb}^{-\frac{1}{2}}\hessian{\smoothfunction}{\llamb}\newD{\smoothfunction}{\llamb}^{-\frac{1}{2}}$ in virtue of the \emph{Gershgorin circle theorem}~\cite{varga00}.
%
%
The  approach suggested by~(\ref{taylorinversehessian})  is to generate vector sequences such that
$\llambk{k+1}=\mapZqfx{q}{\function}{\llambk{k}}$, where $\mapZqfunction{q}\equiv\mapGTXfunction{\Scale}{\setoflambdas}$ with scaling strategy
%
\begin{equation}\label{zargham}
\begin{array}{c}
\Scalefx{\smoothfunction}{\llamb}
=
\newD{\smoothfunction}{\llamb}^{-\frac{1}{2}}
[
\sum\nolimits_{t=0}^{q}
\Scaleterm{\smoothfunction}{\llamb}^{t}
]
\newD{\smoothfunction}{\llamb}^{-\frac{1}{2}}
,
\end{array}
\end{equation}
and~$q$ is a parameter symbolizing the implementability vs. rapidity trade-off, 
and
directly proportional to the computational complexity of the algorithm.
By setting~(\ref{zargham}) in~(\ref{LCMJrate}), we obtain for~$\{\llambk{k}\}$ the asymptotic convergence rate~$\matrixnonviolatingx{\llambconv}\LCMqfx{q}{\smoothfunction}{\llambconv}$, where
\begin{equation}\label{ratezargham}
\LCMqfx{q}{\smoothfunction}{\llambconv}
=
\LCScalefx{\smoothfunction}{\llambconv}
\matrixnonviolatingx{\llambconv}^\trsp
\Scalefx{\smoothfunction}{\llambconv}^{-1}
\globalMqfx{q}{\smoothfunction}{\llambconv}
\matrixnonviolatingx{\llambconv}
\end{equation}
and
$\globalMqfx{q}{\smoothfunction}{\llambconv}\doteq
\newD{\smoothfunction}{\llamb}^{-\frac{1}{2}}
\Scaleterm{\smoothfunction}{\llambconv}^{q+1}
\newD{\smoothfunction}{\llamb}^{\frac{1}{2}}
$
is the asymptotic convergence rate for the unconstrained problem (i.e. $\setoflambdas\equiv\RMM$).
It can be seen
 that $\spectralradius{\LCMqfx{q}{\smoothfunction}{\llambconv}}$ vanishes with growing~$q$. When~$q=0$, (\ref{ratezargham}) reduces  to the rate of~$\mapJfunction$ with coordinate-wise Newton scaling.

\subsection{Random implementations}

We consider the asymptotic convergence of the random algorithm~$\{\mapRkfunction{k}\}$ given in~(\ref{randomimplementation}) and used with probabilities
$\Pseqk{k}\follows\probadistrifunction=\vect{\probadistri{1},...,\probadistri{\nn}}$
for the coordinate directions.
In this context we formulate a strong convexity assumption.
%
%
%
\begin{assumption}
\label{assumptionstrongconvexity}
The function~$\smoothfunction$ is strongly convex so that there exist a symmetric, positive definite block matrix $\boundK=(\boundKij{i}{j})$ 
with $\boundKij{i}{j}\in\REAL{\MMi{i}\times\MMi{j}} $
satisfying $
[\grad\smooth{\llamb}-\grad\smooth{\mmu}]^\trsp(\llamb-\mmu)
\geq 
\specialnorm{\llamb-\mmu}{\boundK}^2
$ for every $\llamb,\mmu\in\setoflambdas$.
\end{assumption}
%
%
%
%

In polyhedral feasible sets, the local convergence of a generated sequen\-{}ce~$\{\llambk{k}\}$  converging to a solution~$\llambconv$ where strict complementarity holds 
occurs in the reduced space at~$\llambconv$. 
In that case we find from~(\ref{coordinateasymptotic}) and for~$k$ large,
$\llambk{k}=\llambconv+ \matrixnonviolatingx{\llambconv} \LClambk{k}$ and
$\llambk{k+1}=\llambconv+ \matrixnonviolatingx{\llambconv} \LClambk{k+1}$, with $\LClambk{k},\LClambk{k+1} \in \REAL{\tildeMM}$ and
\begin{equation}\label{taylorresultrandom}
 \LClambk{k+1} 
\refeq{(\ref{shrinkage})}{=}
\LCMGifx{\Pseqk{k}}{\function}{\llambconv}  \LClambk{k}  
+ \zero{\norm{ \LClambk{k}  }}
.
\end{equation}
The expectation of~(\ref{taylorresultrandom}) in~$\Pseqk{k}$ gives
$
\expectationcond{ \LClambk{k+1} }{\llambk{k},\eventink{k}}
=
\LCMRfx{\function}{\llambconv}  \LClambk{k}  
+ \zero{\norm{ \LClambk{k}  }}
$,
%
where $\LCMRfx{\function}{\llambconv}
\doteq
\tildeidentity - \diag{\probadistri{1}\tildeidentityi{1},...,\probadistri{\nn}\tildeidentityi{\nn}} \LCScalefx{\function}{\llambconv} \hessiantilde{\smoothfunction}{\llamb}$
and~$\eventink{k}$ symbolizes the event that~$\{\llambk{t}\}_{t=k}^{\infty}$ is confined in the reduced space at~$\llambconv$.
In order to derive asymptotic convergence rates for~$\{\mapRkfunction{k}\}$, we need to find out what happens when~$\eventink{k}$ is false, ideally making sure that $[1-\proba{\eventink{k}}] \expectationcond{ \genericresidual{\llambk{k+1}} }{\genericresidual{\llambk{k}} ,\eventoutk{k}}=\zero{\genericresidual{\llambk{k}}}$ for some  residual~$\genericresidual{\cdot}$.
This information can be partially inferred  from the following lemma, which extends to arbitrary distributions a convergence result derived in \cite[Theorem~5]{nesterov12}  for the algorithm known 	as UCDM, which is a version of~$\{\mapRkfunction{k}\}$ using fixed scaling in each direction and equal probabilities $\probadistri{i}=\frac{1}{\nn}$ for all directions $i\in\setofnodes$.
%
%
\obsolete{
This information can be partially  inferred from convergence results for the algorithm known 	as UCDM~\cite{nesterov12}, which is a version of~$\{\mapRkfunction{k}\}$ using fixed scaling in each direction and the probability distribution~$\probadistrifunction$ uniform for the coordinate directions ($\probadistri{i}=\frac{1}{\nn}$ for $i\in\setofnodes$).
Since we believe that the  uniform probability restriction to be be constraining in distributed contexts, these results are reconsidered in this study with arbitrary~$\probadistrifunction$.

In the case of arbitrary probabilities, \cite[Theorem~5]{nesterov12} can be stated as follows.
}%
\begin{lemma}[Convergence  of~$\{\mapRkfunction{k}\}$]\label{lemmalinearconvergence}
%
Assume that  Problem~\ref{initialproblem} has a unique solution~$\llambconv$ and that the feasible set is the Cartesian product $\setoflambdas=\setoflambdasi{1}\times...\times\setoflambdasi{\nn}$.
Consider the cyclic algorithm $\llambk{k+1}=\mapRkfx{k}{\function}{\llambk{k}}$  implemented with 
$\Pseqk{k}\follows\probadistrifunction=\vect{\probadistri{1},...,\probadistri{\nn}}$ at every step~$k$,
the step-size selection rule~(\ref{stocharmijoconditionglobalone}) where~$\Psig\leq\frac{1}{2}$, and fixed scaling~
$\Scalefx{\function}{\llamb}\equiv\constantScale=\diag{\constantScalei{1},...,\constantScalei{\nn}}$ with $\constantScalei{i}\matrixleq
\Lipschitzi{i}^{-1}
$  
for $i\in\setofnodes$.
Define
\begin{equation}\label{supernormfunction}
\supernormfunction:\llamb\in\setoflambdas\functionto
\supernorm{\llamb}
\doteq
\frac{\correctproba{1}{\nn\probamin}}{2}\specialnorm{\llamb-\llambconv}{\smartscale}^2
+  \smooth{\llamb}-\smooth{\llambconv} 
\in\REALplusone \,
,
\end{equation}
where \correctproba{}{$\probamin\doteq \min\{\probadistri{1},...,\probadistri{\nn}\}$ and }%
$\smartscale\doteq[\nn\, \diag{\probadistri{1}\constantScalei{1},...,\probadistri{\nn}\constantScalei{\nn}}]^{-1}$.
For any sequence~$\{\llambk{k}\}$ generated by the algorithm, we have
\begin{equation}\label{sublinearrate}
\expectation{
\smooth{\llambk{k}}-\smooth{\llambconv}
}
\leq
\correctproba{\frac{\nn}{\nn+k}}{\frac{1}{1+\probamin k}}
 \supernorm{\llambk{0}}
,\hspace{5mm}k=0,1,2,... \ .
\end{equation}
If, in addition, $\smoothfunction$ is strongly convex as in Assumption~\ref{assumptionstrongconvexity},
then
\begin{align}\label{linearrate}
&\hspace{10mm}
\expectationcond{
\supernorm{\llambk{k+1}}
}{\llambk{k}}
\leq
\left(1-
\correctproba{\frac{2\minK}{\nn(1+\minK)}}{\frac{2 \probamin\minK}{\minK+\nn\probamin}}
\right)
\supernorm{\llambk{k}}
,
&
k=0,1,2,... ,
\end{align}
where
%
the constant~$\minK>0$ 
satisfies  $\minK\smartscale\matrixleq\boundK$. 
%
%
\end{lemma}
The proof is similar to that of~\cite[Theorem~5]{nesterov12} and reported in the Appendix.
%
\obsolete{
\correctproba{%
}{
\begin{remark}
Tightness of the convergence rates is lost in the last line of~(\ref{lossoftightness}), where we use $
\smallexpectation{
{2
}/({\nn\probamin})
}
[\smooth{\llambk{k}}-\smooth{\llambk{k+1}}]
\leq
\smallexpectation{
{2
}/({\nn\probadistri{\Pseqk{k}}})
}
[\smooth{\llambk{k}}-\smooth{\llambk{k+1}}]
$,
the inequality being strict when~$\probadistrifunction$ is not uniform. The rates in~(\ref{sublinearrate}) and~(\ref{linearrate}) are thus overestimators  when~$\probamin\neq\frac{1}{\nn}$. 
\shorten{}{
In the sequel, tight asymptotic rates are computed for any~$\probadistrifunction$ in polyhedra. Only positive definiteness of~$\hessian{\smoothfunction}{\llambconv}$ and strict complementarity at~$\llambconv$ are then required for linear convergence of~$\smooth{\llambk{k}}$.
}
\end{remark}
}%
}%
\shorten{
We are now able to characterize the convergence of the algorithm in polyhedra.
}{}

\obsolete{
\begin{lemma}\label{finitetimecoordinate}
Let Problem~\ref{initialproblem} satisfy Assumptions~\ref{assumptionparallel} and~\ref{assumptiondifferentiability} and have a unique solution~$\llambconv$ where strict complementarity holds.
Consider a sequence~$\{\llambk{k}\}$ generated by the cyclic algorithm $\llambk{k+1}=\mapRkfx{k}{\function}{\llambk{k}}$  implemented with 
$\Pseqk{k}\follows\probadistrifunction=\vect{\probadistri{1},...,\probadistri{\nn}}$ at all~$k$.
If~$\eventink{k}$ symbolizes the event that~$\{\llambk{t}\}_{t=k}^{\infty}$ is confined in the reduced space at~$\llambconv$, i.e. $\activeaffinex{\llambk{t}}\equiv\activeaffinex{\llambconv}$ for~$t\geq k$, then
%
\begin{equation}\label{equationeventin}
\proba{\eventink{k}} \leq ... ,
\end{equation}
where $\proba{\eventink{k}}\to 0$.
\end{lemma}
\begin{proof}
...

\QEDA
\end{proof}
}%

\begin{proposition}[{Asymptotic convergence  of~$\{\mapRkfunction{k}\}$}]
%
%
Let Assumptions~\ref{assumptionparallel} and~\ref{assumptiondifferentiability} hold for Problem~\ref{initialproblem}.
Consider the cyclic algorithm $\llambk{k+1}=\mapRkfx{k}{\function}{\llambk{k}}$  implemented with 
$\Pseqk{k}\follows\probadistrifunction=\vect{\probadistri{1},...,\probadistri{\nn}}$ at all~$k$,
the step-size selection rule~(\ref{stocharmijoconditionglobalone}) where~$\Psig\leq\frac{1}{2}$, and fixed scaling~
$\Scalefx{\function}{\llamb}\equiv\constantScale=\diag{\constantScalei{1},...,\constantScalei{\nn}}$ with $\constantScalei{i}\matrixleq
\Lipschitzi{i}^{-1}
$  
for $i\in\setofnodes$.
For any sequence~$\{\llambk{k}\}$ generated by the algorithm, $\expectation{\smooth{\llambk{k}}}$ 
converges 
towards~$\smooth{\llambconv}$ with limiting asymptotic rate
\begin{equation}\label{NMRFrate}
\begin{array}{c}
{\NMRFfx{\function}{\llambconv}}\doteq
\spectralradius{
\sum_{i=1}^{\nn}
\probadistri{i} \LCMGifx{i}{\function}{\llambconv}^\trsp \hessiantilde{\smoothfunction}{\llambconv} \LCMGifx{i}{\function}{\llambconv} \hessiantilde{\smoothfunction}{\llambconv}^{-1}
}
\leq 1.
\end{array}
\end{equation}
If~$\hessian{\smoothfunction}{\llambconv}$ is positive definite, then~$\NMRFfx{\function}{\llambconv}<1$.

Moreover, if $\smoothfunction$ is strongly convex  as in Assumption~\ref{assumptionstrongconvexity},
then
$\expectation{\supernorm{\llambk{k}}}$
vanishes at least linearly with limiting asymptotic rate
\begin{equation}\label{NMRSrate}
\begin{array}{c}
{\NMRSfx{\function}{\llambconv}}\doteq
\max\left\{
\NMRFfx{\function}{\llambconv}
,
\spectralradius{
\sum\nolimits_{i=1}^{\nn}
\probadistri{i} \LCMGifx{i}{\function}{\llambconv}^\trsp \tildesmartscale \LCMGifx{i}{\function}{\llambconv} \tildesmartscale^{-1}
}
\right\}
<1,
\end{array}
\end{equation}
where $\tildesmartscale=\matrixnonviolatingx{\llambconv}^\trsp\smartscale\matrixnonviolatingx{\llambconv}$,
and~$\supernormfunction$ and~$\smartscale$ are defined as in~(\ref{supernormfunction}).
\end{proposition}
\begin{proof}
Let~$\{\llambk{k}\}$ be a sequence generated by the algorithm.
 Proposition~\ref{identificationofthesaturatednonnegativityconstraints} claims that one can find a $\ballradius >0$ such that 
 $\activeaffinex{\llambk{t}}=\activeaffinex{\llambconv}$ for $t\geq k+1$ when $\norm{\llambk{k}-\llambconv}<\ballradius$.
If $\smoothmin\doteq\max\{{\smooth{\llamb}}\setst{\norm{\llamb-\llambconv}<\ballradius,\llamb\in\setoflambdas}\}$ and~$\eventink{k}$ 
is defined as above as the event that $\activeaffinex{\llambk{t}}\equiv\activeaffinex{\llambconv}$ for~$t\geq k$,
it follows from Lemma~\ref{lemmalinearconvergence} that
\begin{equation}
\proba{\eventink{k+1}}\geq \proba{\norm{\llambk{k}-\llambconv}<\ballradius}\geq \proba{\smooth{\llambk{k}}<\smoothmin} \refeq{(\ref{sublinearrate})}{\geq} 1 - \frac{\supernorm{\llambk{0}}}{(1+\probamin k)(\smoothmin-\smooth{\llambconv})} 
.
\end{equation}
%
%
%
Hence $\proba{\eventink{k+1}}\to1$.
From 
 Proposition~\ref{propositionefficiency}, we also know that the step sizes are equal to~$1$, and from~(\ref{coordinateunconstrainedequation}) that, when~$\eventink{k}$ is true, then
$\llambk{k}=\llambconv+ \matrixnonviolatingx{\llambconv} \LClambk{k}$ and
$\llambk{k+1}=\llambconv+ \matrixnonviolatingx{\llambconv} \LClambk{k+1}$
for some $\LClambk{k},\LClambk{k+1} \in \REAL{\tildeMM}$ satisfying~(\ref{taylorresultrandom}).

Consider  the sequence of function values~$\{\functionx{\llambk{k}}\}$ and  a step~$k$.
If~$\eventink{k}$ is true, we have $\grad\smooth{\llambconv}^\trsp(\llambk{k}-\llambconv)=0$.
By setting
$
\Nlambk{k}
\doteq
\hessiantilde{\smoothfunction}{\llambconv}^{\frac{1}{2}}
\LClambk{k}
$
and
$
\NMGifx{i}{\function}{\llambconv}
\doteq
\hessiantilde{\smoothfunction}{\llambconv}^{\frac{1}{2}}
\LCMGifx{i}{\function}{\llambconv}
\hessiantilde{\smoothfunction}{\llambconv}^{-\frac{1}{2}}
$ for $i\in\setofnodes$,
and proceeding as in~(\ref{equationbegin})-(\ref{equationend}), we find
\begin{equation}
\smooth{\llambk{k+1}}
-
\smooth{\llambconv}
=
\frac{1}{2}(\Nlambk{k})^\trsp\NMGifx{\Pseqk{k}}{\function}{\llambconv}^\trsp\NMGifx{\Pseqk{k}}{\function}{\llambconv}\Nlambk{k}
+\zero{\norm{\LClambk{k}}^2}
.
\end{equation}
 Thus,
\begin{align}
\begin{array}{c}
\expectationcond{
\smooth{\llambk{k+1}}
-
\smooth{\llambconv}
}{\llambk{k},\eventink{k}}
\end{array}
\hspace{-17mm}&
\notag
\\
&
\begin{array}{c}
\refeq{(\ref{taylorresultrandom})}{=}
\frac{1}{2}(\Nlambk{k})^\trsp
\sum\nolimits_{i=1}^{\nn}
\probadistri{i}[
\NMGifx{i}{\function}{\llambconv}^\trsp\NMGifx{i}{\function}{\llambconv}
]
\Nlambk{k}
%
+\zero{\norm{\Nlambk{k}}^2}
\end{array}
\\
&
\label{equationNMRFfx}
\begin{array}{c}
\leq
\NMRFfx{\function}{\llambconv}[\smooth{\llambk{k}}-\smooth{\llambconv}]
+\zero{\smooth{\llambk{k}}-\smooth{\llambconv}}
\end{array}
\end{align}
where~$\NMRFfx{\function}{\llambconv}$ is given by~(\ref{NMRFrate}). 
If, however, $\eventink{k}$ is false, then $[\smooth{\llambk{k+1}}-\smooth{\llambconv}]\leq[\smooth{\llambk{k}}-\smooth{\llambconv}]$ by~(\ref{stocharmijoconditionglobalone}).
All in all, we have
\begin{equation}\label{equationdifference}
\expectationcond{
\smooth{\llambk{k+1}}-\smooth{\llambconv}
}{\llambk{k}}
\leq
\NMRFfx{\function}{\llambconv}
[\smooth{\llambk{k}}-\smooth{\llambconv}]
+\differencek{k}
+
\zero{\smooth{\llambk{k}}-\smooth{\llambconv}},
\end{equation}
where $\differencek{k}=[1-\proba{\eventink{k}}][1-\NMRFfx{\function}{\llambconv}]
[\smooth{\llambk{k}}-\smooth{\llambconv}]=
\zero{\smooth{\llambk{k}}-\smooth{\llambconv}}$,
and the rate~$\NMRFfx{\function}{\llambconv}$ is tight. 
\obsolete{
\begin{equation}
\NMRFfx{\function}{\llambconv}\doteq
\hessiantilde{\smoothfunction}{\llambconv}^{-\frac{1}{2}} 
\left[
\sum\nolimits_{i=1}^{\nn}
\probadistri{i}
\LCMGifx{i}{\function}{\llambconv}^\trsp \hessiantilde{\smoothfunction}{\llambconv} \LCMGifx{i}{\function}{\llambconv}
\right]
\hessiantilde{\smoothfunction}{\llambconv}^{-\frac{1}{2}}
.
\end{equation}
}%

Let Assumption~\ref{assumptionstrongconvexity} hold.
Similarly, one can write
$
\specialnorm{\llambk{k+1}-\llambconv}{\smartscale}^2
=
\specialnorm{\LClambk{k+1}}{\tildesmartscale}^2
$ when~$\eventink{k}$ is true. 
Using~(\ref{taylorresultrandom}) and~(\ref{equationNMRFfx}), one finds 
\begin{align}\label{expsupernorm}
\expectationcond{
\supernorm{\llambk{k+1}}
}{\llambk{k}}
&
\refeq{(\ref{linearrate})}{\leq}
\NMRSfx{\function}{\llambconv}
\supernorm{\llambk{k}}
+\differencetwok{k}
+\zero{\supernorm{\llambk{k}}},
\end{align}
where $\differencetwok{k}=[1-\proba{\eventink{k}}][1-
\correctproba{{2\minK}{(\nn+\minK\nn)^{-1}}}{{2 \probamin\minK}({\minK+\nn\probamin})^{-1}}
-\NMRSfx{\function}{\llambconv}]
\supernorm{\llambk{k}}=
\zero{\supernorm{\llambk{k}}}$,
and~$\NMRSfx{\function}{\llambconv}$ is given by~(\ref{NMRSrate}) and tight. 
It follows from Lemma~\ref{lemmalinearconvergence} and~(\ref{linearrate}) that $\NMRSfx{\function}{\llambconv}<1$, and thus $\NMRFfx{\function}{\llambconv}<1$.
Otherwise there would exist a vector $\mmuk{0}=\llambconv+ \matrixnonviolatingx{\llambconv} \epsilon\LCmu\in\setoflambdas$ such that $\NMRSfx{\function}{\llambconv}\geq1$ and~(\ref{expsupernorm}) holds with equality sign, which 
contradicts~(\ref{linearrate}) if we take~$\epsilon$ small enough.
%

Assume now that~$\smoothfunction$ is not necessarily strongly convex, yet~$\hessian{\smoothfunction}{\llambconv}$ is positive definite.
Because~$\NMRFfx{\function}{\llambconv}$ depends only on local properties of~$\smoothfunction$, applying the same algorithm within~$\setoflambdas$ to a strongly convex function~$\smoothtwofunction$ with the same derivative and Hessian as~$\smoothfunction$ in a neighborhood of~$\llambconv$ will see $\smallexpectation{\smoothtwo{\llambk{k+1}}-\smoothtwo{\llambconv}}$ converge with  asymptotic rate~$\NMRFfx{\function}{\llambconv}<1$. 
%
When~$\hessian{\smoothfunction}{\llambconv}$ is positive semi-definite, we find $\NMRFfx{\function}{\llambconv}\leq1$ by considering the function $\smoothfunction+\frac{\epsilon}{2}\norm{\llamb}^2$ and using the same continuity arguments as in the proof of Proposition~\ref{localconvergenceprop}.
%
\QEDA
\end{proof}
\begin{remark}
Let us compare the rates given in~(\ref{linearrate}) and~(\ref{NMRSrate}). For this purpose we  place ourselves in the conditions which optimize the precision of~(\ref{linearrate})
by supposing~(i) that the bound~$\minK$ is tight, in the sense that 
$\minK\smartscale\matrixleq\boundK$ is satisfied with equality sign (thus implying that~$\boundK$ is block diagonal), (ii) that~$\hessian{\smoothfunction}{\llambconv}$ is defined and equal to~$\boundK$ 
so that the strong convexity constant~$\boundK$ is itself tight and we have the constraint $\minK\leq\nn\probamin$ imposed by $\constantScale\matrixleq
\Lipschitz^{-1}\matrixleq\boundK^{-1}$, and (iii)  that $\activeaffinex{\llambconv}=\emptyset$.
After computations, we find $\NMRSfx{\function}{\llambconv}=1-\frac{\minK}{\nn}(2-\frac{\minK}{\nn\probamin})$. The ratio with the rate~(\ref{linearrate}) gives 
\begin{equation}\label{equationdifference}
\frac{1-\NMRSfx{\function}{\llambconv}}{1-\left(1-\frac{2 \probamin\minK}{\minK+\nn\probamin}\right)}
=
1+\frac{\minK(\nn\probamin-\minK)}{2(\nn\probamin)^2}\geq 1.
\end{equation}
It can be inferred from~(\ref{equationdifference}) that the rate~(\ref{linearrate}) is (for this problem) generally conservative, and equal to the asymptotic rate~$\NMRSfx{\function}{\llambconv}$ iff $\minK=\nn\probamin$ holds, i.e., 
when 
when we use a scaled version $\constantScalei{i}=\frac{\probamin}{\probadistri{i}}\hessiani{\smoothfunction}{ii}{\llambconv}^{-1}$ ($i\in\setofnodes$) of the asymptotic expression of coordinate-wise Newton scaling approach  previously discussed in Remark~\ref{sectionnewton}%
---in that case the convergence rate reduces to $1-\probamin$.
\end{remark}

\subsection{Non-twice differentiable cost functions}

In the previous sections we have assumed that~$\hessianfunction{\smoothfunction}$ existed 
at the point of convergence~$\llambconv$. 
Suppose instead that~$\hessian{\smoothfunction}{\llambconv}$ is not defined but that~$\smoothfunction$ satisfies a strong convexity condition at least locally in a neighborhood~$\nbhd$ of~$\llambconv$, i.e. there is a symmetric, positive definite matrix~$\boundK$ such that $
[\grad\smooth{\llamb}-\grad\smooth{\mmu}]^\trsp(\llamb-\mmu)
\geq
\specialnorm{\llamb-\mmu}{\boundK}^2
$ holds for $\llamb,\mmu\in\nbhd$.

Under Assumption~\ref{assumptionparallel}, consider any algorithm based on~$\{\maphatGi{i}\}_{i=1}^{\nn}$, such as those introduced in Section~\ref{sectionparalleloptimization}, and generate a
 sequence~$\{\llambk{k}\}$ with step-size selection rule~(\ref{stocharmijoconditionglobalone}) asymptotically efficient in the sense of Proposition~\ref{propositionefficiency}. 
Assume that strict complementarity holds at~$\llambconv$, so that convergence occurs in the reduced space at~$\llambconv$ and, for large~$k$, we have
$\llambk{k} = \llambconv + \matrixnonviolatingx{\llambconv} \LClambk{k}  $ with $  \LClambk{k} \in \REAL{\tildeMM} $.
In view of Remark~\ref{remarkswitchingsystem} and by considering directional derivatives of~$\grad\smoothfunction$ in the developments that lead to~(\ref{coordinateasymptotic}), we find
\begin{equation}
\maphatGifx{i}{\function}{\llambk{k}}-\llambconv=\NDGki{k}{i}  (\llambk{k}-\llambconv)+\zero{\norm{\llambk{k}-\llambconv}}
\end{equation}
where 
$\NDGki{k}{i}\doteq \matrixnonviolatingx{\llambconv} [ \tildeidentity - \diag{0,...,0,\tildeidentityi{i},0,...,0} \LCScalefx{\function}{\llambconv} 
\matrixnonviolatingx{\llambconv}^{\trsp}
\NDMki{k}{i}
\matrixnonviolatingx{\llambconv}
]
$
for some matrix $\NDMki{k}{i}\in\NDMset$,
where~$\NDMset$ denotes the set of all the symmetric matrices~$\genericnormscale$ satisfying $\boundK\matrixleq\genericnormscale\matrixleq\Lipschitz\}$.
Suppose now that, for any strongly convex function~$\smoothtwofunction$ which realizes its minimum on~$\setoflambdas$ at~$\llambconv$ and satisfies Assumption~\ref{assumptiondifferentiability}, the algorithm produces sequences~$\mmuk{k}$  linearly convergent towards~$\llambconv$ with respect to some residual~$\genericresidual{\mmuk{k}}$ and with rate~$\NDHx{\hessian{\smoothtwofunction}{\llambconv}}$, i.e.
\begin{equation}
\genericresidual{\mmuk{k+1}}\leq \NDHx{\hessian{\smoothtwofunction}{\llambconv}} \genericresidual{\mmuk{k}} + \zero{\genericresidual{\mmuk{k}}},
\end{equation}
where $\spectralradius{\NDHx{\genericnormscale}}<1$ for any $\genericnormscale\in\NDMset$, and~$\NDHx{\cdot}$ is a continuous mapping. 
It follows from the compactness of~$ \NDMset $
that we can find a matrix $\NDHmax\in\NDMset$ such that
$\spectralradius{\NDHmax}=
\max_{\genericnormscale\in\NDMset}\{\spectralradius{\NDHx{\genericnormscale}}\}<1$ 
and
\begin{equation}
\genericresidual{\llambk{k+1}}\leq \NDHmax \genericresidual{\llambk{k}} + \zero{\genericresidual{\llambk{k}}}.
\end{equation}

%
%
%
%
%
\obsolete{
For any matrix $\genericblock=(\genericblockij{i}{j})$ of~$ \REAL{\tildeMM\times\tildeMM}$ with $\genericblockij{i}{j}\in\REAL{\tildeMMi{i}\times\tildeMMi{j}}$ for $i,j\in\setofnodes$ and such that $\genericblock=\genericD-\genericL-\genericL^\trsp$
for 
$\genericD=\diag{\genericblockij{1}{1},...,\genericblockij{\nn}{\nn}}$  
block diagonal and~$\genericL$ strictly lower block diagonal, we define
\begin{align}
\label{genericLCMrate}
\LCMgeneric{\genericblock} 
&\doteq 
\left
\lbrack \LCScalefx{\function}{\llambconv}^{-1} -  \genericL \right\rbrack^{-1}  \left\lbrack \LCScalefx{\function}{\llambconv}^{-1} - \genericD + \genericL^\trsp \right\rbrack,
\\
\LCMRgeneric{\genericblock} 
&\doteq
\tildeidentity -
\diag{\probadistri{1}\tildeidentityi{1},...,\probadistri{\nn}\tildeidentityi{\nn}} \LCScalefx{\function}{\llambconv} \genericblock
\end{align}

}

\subsection{Stochastic optimization based on gradient projections}

Some stochastic optimization problems are  concerned with the minimization of a  cost function  unknown in closed form that  can only be estimated through measurement or simulation.
Assume  in Problem~\ref{initialproblem} that~$\smoothfunction$ is unknown, while a sequence~$\{\functionk{k}\}$ of estimates in~$\setoffunctionsn{\MM}$ is available for~$\function$ with common Lipschitz constant for every~$\grad\smoothkfunction{k}$,
and that~$\{\functionk{k}\}$  converges  almost surely towards~$\function$ 
in the sense that
$
\lbrack \sup\nolimits_{\llamb\in\compact}\singlenorm{\smoothkx{k}{\llamb}-\smooth{\llamb}}
+
\sup\nolimits_{\llamb\in\compact}\norm{\grad\smoothkx{k}{\llamb}-\grad\smooth{\llamb}}
\rbrack
$ 
vanishes almost surely
 for any compact set~$\compact\subset\setoflambdas$.
An approach for solving this problem consists of sequentially applying  an iterative optimization algorithm~$\mapfunction$ along the  sequence of function estimates, i.e.
\begin{equation}\label{genericalgorithm}
\llambk{k+1}=\mapfx{\functionk{k}}{\llambk{k}}
,\hspace{5mm}k=0,1,2,...\ .
\end{equation}
The bounded sequences~$\{\llambk{k}\}$ generated by~(\ref{genericalgorithm}) are known to converge almost surely towards a nonempty solution set  provided that~$\mapfunction$ is \emph{closed} and a \emph{descent algorithm} with respect to~$\function$ and the  set of solutions~\cite[Theorem~2.1]{shapiro96}.
Possible choices for~$\mapfunction$ include the gradient projection algorithm~$\mapGTXfunction{\Scale}{\setoflambdas}$ and (under Assumption~\ref{assumptionparallel}) the parallel implementations of Section~\ref{sectionparalleloptimization}, whose convergence  in stochastic settings is specifically addressed in~\cite{bilenne15parallel}. 
%

%
Consider such an algorithm~$\mapfunction$, and suppose that strict complementarity holds at~$\llambconv$ (Assumption~\ref{assumptiondifferentiability})  and 
that each function~$\smoothkfunction{k}$ 
has a unique minimizer~$\accuk{k}$  on~$\setoflambdas$  where $\hessianfunction{\smoothkfunction{k}}$ is defined, continuous  and positive definite at~${\accuk{k}}$.
%
By Lipschitz continuity of~$\function$, the sequence~$\{\accuk{k}\}$ is such that $\activeaffinex{\accuk{k}}\to\activeaffinex{\llambconv}$, i.e. $\accuk{k}=\accu+\matrixnonviolatingx{\llambconv}\LCaccuk{k}$ for $\LCaccuk{k}\in\REAL{\tildeMMi{i}}$ and, say, $k>\ksc$, with strict complementarity holding at~$\accuk{k}$ for~$\smoothkfunction{k}$.
Assume that the considered  algorithm~$\mapfunction$ produces,  when applied to any~$\smoothkfunction{k}$ with $k>\ksc$,  sequences in~$\setoflambdas$  converging towards~$\accuk{k}$ in the subspace at~$\llambconv$ with rate $\LCMgenfx{\functionk{k}}{\accuk{k}} <1$. 
Any bounded sequence~$\{\llambk{k}\}$ generated by~(\ref{genericalgorithm}) will then be such that
%
\obsolete{
, for any generating a  vector sequence~$\{\llambk{k}\}$ which,  almost surely, converges to~$\llambconv$ 
reaching in finite time the reduced space at~$\llambconv$.
By extension of Propositions~\ref{identificationofthesaturatednonnegativityconstraints} and~\ref{unconstraineddescent} and by application of  Taylor's theorem---recall~(\ref{taylorresult})---, we find  
}%
$\llambk{k} = \llambconv + \matrixnonviolatingx{\llambconv} \LClambk{k}  $ and $\llambk{k+1} = \llambconv + \matrixnonviolatingx{\llambconv} \LClambk{k+1}  $ for~$k$ large enough and for some
$  \LClambk{k}, \LClambk{k+1} \in \REAL{\tildeMM} $ satisfying
\begin{equation}
 \LClambk{k+1} -\LCaccuk{k}
  = \LCMgenfx{\functionk{k}}{\accuk{k}}  (\LClambk{k}-\LCaccuk{k})  + 
\fullremainderk{k}{\LClambk{k}}
, \label{equationM}
\end{equation}
where $\fullremainderk{k}{\LClamb}=\zero{\norm{\LClamb-\accuk{k}}}$ for  $k\geq\ksc$.
Further, if $\function$  and all~$\functionk{k}$ are  smooth, the scale~$\Scalefx{\function}{\cdot}$ of~$\mapfunction$ is continuously differentiable at~$\llambconv$, and, almost surely, $\hessianfunction{\functionk{k}}$ and its derivatives  converge uniformly on a neighborhood of~$\llambconv$ towards~$\hessianfunction{\function}$ and its derivatives, respectively, then
$\fullremainderk{k}{\LClambk{k}}\equiv\remainderk{k}{\LClambk{k}} (\LClambk{k}-\LCaccuk{k})(\LClambk{k}-\LCaccuk{k})^\trsp$,
where~$\remainderk{k}{\LClambk{k}}$ is a function of  derivatives at~$\accu$ of~$\hessianfunction{\function}$ and~$\Scalefx{\function}{\cdot}$ and uniformly bounded for all $k\geq\ksc$ on a neighborhood of~$\accu$ in accordance with Proposition~\ref{unconstraineddescent}.
\obsolete{
\obsolete{
where~$\tildeMM$ is the dimension of the reduced space at~$\llambconv$,
$\matrixnonviolatingx{\llambconv}$ is defined as in Section~\ref{sectionpolyhedra}, $\LCaccuk{k}\doteq\matrixnonviolatingx{\llambconv}^\trsp(\accuk{k}-\llambconv)$, $\LCMgenfx{\functionk{k}}{\accuk{k}} $ is asymptotic rate of convergence towards~$\accuk{k}$  when~$\mapfunction$ is applied to~$\functionk{k}$, and
}%
the remainder~$\remainderk{k}{\llambk{k}} 
$
%
 is a function of second derivatives of~$\functionk{k}$. 
If~$\{\hessianfunction{\functionk{k}}\}$ converges uniformly towards~$\hessianfunction{\function}$ on a {neighborhood} of~$\accu$,  then~$\remainderk{k}{\llambk{k}}$ is uniformly bounded  
}%
Then, (\ref{equationM}) rewrites (with probability one) as
\begin{equation}\label{markovchain}
\llambk{k+1} -\accu
  = 
\markovAk{k}
(\llambk{k} -\accu)
 +
\markovBk{k}
(\accuk{k} -\accu)
+
\zero{\norm{\llambk{k} -\accu }}
, 
\end{equation}
where 
$\markovAk{k}=\matrixnonviolatingx{\accu}\LCMgenfx{\functionk{k}}{\accuk{k}}  \matrixnonviolatingx{\accu}^\trsp$,
and
$\markovBk{k}=\matrixnonviolatingx{\accu}(\tildeidentity-\LCMgenfx{\functionk{k}}{\accuk{k}}) \matrixnonviolatingx{\accu}^\trsp$. 

The asymptotics  of~$\{\accuk{k}-\accu\}$ ensue from the nature of the  function sequence~$\{\functionk{k}\}$.
In many problems, $\function$ 
is an expectation function of the type
\begin{equation} \label{stochdualfunction}
\begin{array}{c}
\functionx{\llamb} = \int\nolimits_{\Stoch} { \stochg{\llamb}{\stoch} }   \distri{d\stoch},
\hspace{5mm} \forall \llamb \in \RMM, 
\end{array}
\end{equation}
%
%
where~$\stoch$ is a random parameter defined on a probability space~$\probabilityspace{\Stoch}{\mathcal{F}}{\distrifunction}$,  and~$\stochg{\cdot}{\stoch}$ serves as a random measurement of~$\function$, modeling for instance the optimal value of the second-stage problem of a two-stage stochastic program~\cite{shapiro09}.
%
%
%
Based on~(\ref{stochdualfunction}) and the simulation of a sequence of samples~$\{\{\stochkq{k}{l}\}_{l=\SA{0}{1}}^{\qk{k}\SA{-1}{}}\}$  of independent {realizations} of~$\stoch$, with $\qk{k}\to\infty$ as $k\to\infty$, it is common to consider the \emph{sample average approximation (SAA)} 
\begin{equation}\label{empiricestimator}
\functionkx{k}{\llamb}=\frac{1}{\qk{k}}\sum\nolimits_{l=\SA{0}{1}}^{\qk{k}\SA{-1}{}}\stochg{\llamb}{\stochkq{k}{l}}
,\hspace{5mm} k=0,1,2,...\ \shortlong{,}{.}
\end{equation}
which  converges almost surely and uniformly   towards~$\thefunction$ on any compact set~$\compact\subset\setoflambdas$ under certain continuity and integrability conditions for~$\stochgfunction$~\cite{rubinstein93}.
%
%
%
%
%
The sequence~$\{\accuk{k}\}$ is then known as the (SAA) estimator, 
and it follows from the central limit theorem that~(\ref{empiricestimator}) is asymptotically normal, 
i.e.
\begin{equation}\label{NRV}
\qk{k}^{-\frac{1}{2}}
[\functionkx{k}{\llamb}-\functionx{\llamb})]\distrito\NRVx{\llamb}
,\hspace{5mm} \forall \llamb\in\setoflambdas
,\end{equation}
where~$\distrito$ denotes convergence in distribution and~$\NRVx{\llamb}$ is a centered normal random variable with variance $\varstochg{\llamb}=\Var{\stochg{\llamb}{\stoch}}$.
%
Since the hypotheses of~\cite[Theorem~5.8]{shapiro09} are then satisfied at~$\accu$,  the first order asymptotics of the SAA estimator~$\accuk{k}$ can be inferred from the second order Taylor series expansion of~$\function$ at~$\accu$ and the Delta theorem, 
%
%
%
and we find
\begin{equation}\label{SAA}
\qk{k}^{-\frac{1}{2}}[\accuk{k} -\accu]\distrito
-\matrixnonviolatingx{\llambconv}\hessiantilde{\function}{\accu}^{-1}\matrixnonviolatingx{\llambconv}^\trsp\grad\NRVx{\accu},
 \end{equation}
where $\hessiantilde{\smoothfunction}{\llamb}\doteq \matrixnonviolatingx{\llambconv}^\trsp   \hessian{\smoothfunction}{\llamb} \matrixnonviolatingx{\llambconv}$.
%
%
%
%
\obsolete{
provided that
$\arg\inf_{\LCvariation\in\REAL{\tildeMM}}\{2 \LCvariation^\trsp\matrixnonviolatingx{\llambconv}^\trsp\grad\dirx{\accu}+\LCvariation^\trsp\hessiantilde{\function}{\accu}\LCvariation\}
$
yields a singleton~$\{\secondderivdir{\dir}\}$ for every $\dir\in\setoffunctionsn{\MM} $, where 
\obsolete{
$\critical{\accu}=\{\matrixnonviolatingx{\llambconv} \LCvariation\setst\LCvariation\in\REAL{\tildeMM}\}$
is the {critical cone} at~$\accu$.
}%
$\hessiantilde{\smoothfunction}{\llamb}\doteq \matrixnonviolatingx{\llambconv}^\trsp   \hessian{\smoothfunction}{\llamb} \matrixnonviolatingx{\llambconv}$.
In that case we have
\begin{equation}\label{SAA}
\qk{k}^{-\frac{1}{2}}[\accuk{k} -\accu]\distrito\matrixnonviolatingx{\llambconv}\secondderivdir{\NRV}.
 \end{equation}
%
Equation~(\ref{SAA}) holds in particular when~$\hessianfunction{\function}$ is positive definite. Then 
we find
$
\secondderivdir{\dir}
=-\hessiantilde{\function}{\accu}^{-1}\matrixnonviolatingx{\llambconv}^\trsp\grad\dirx{\accu}
$
for every $\dir\in\setoffunctionsn{\MM} $.
}%

We see from~(\ref{markovchain}) and~(\ref{SAA}) that the convergence of the sequence~$\{\llambk{k}\}$ 
is then asymptotically analogous to that of a discrete-time random dynamical system characterized by (i)~the affine mapping
sequence~$\{\markovAk{k}\}$, which converges almost surely towards the asymptotic convergence rate $\markovAk{\infty}=\matrixnonviolatingx{\accu}\LCMgenfx{\function}{\accu}  \matrixnonviolatingx{\accu}^\trsp$ of the (typically \emph{linearly} convergent) algorithm~$\mapfunction$, and (ii)~a random noise process with variance vanishing \emph{sublinearly} like~$\magnitude{\qk{k}^{-1}}$, thus hindering the whole optimization process and dictating its actual asymptotic performance. 
\begin{remark}
The impact of variance of the SAA estimator can be  lessened  using variance reduction~\cite{shapiro96conf,shapiro98,shapiro09} or scenario reduction techniques~\cite{romisch09}. Reducing the computational charge due to sample averaging is  possible for instance by  controlling the sample generation process~\cite{dupuis91}, or by synchronizing---possibly in parallel---the application of the descent algorithm~(\ref{genericalgorithm}) with the increasing precision of~$\{\functionk{k}\}$~\cite{bilenne15parallel}.
\end{remark}

\jotaIEEE{
\appendix
\section*{Appendix: proofs and auxiliary results}
}{
\appendices
\section{}
}

\begin{proof}[Proof of Proposition~\ref{propositionefficiency}]
Consider any~$\notx\in\genericset$ and  the gradient projection  $\mapGTXfx{\Scale}{\genericset}{\GPfunction}{\notx}$ with a tentative step size $\Pa\in(0,1]$. We have
\begin{align}
\GPfunctionx{\notx}-\GPfunctionx{\searchfxa{\GPfunction}{\notx}{\Pa}}
&
\refeq{(\ref{Lequation})}{\geq}
-\grad\GPfunctionx{\llamb}^\trsp(\searchfxa{\GPfunction}{\notx}{\Pa}-\notx)-\frac{1}{2}\specialnorm{\searchfxa{\GPfunction}{\notx}{\Pa}-\notx}{\Lipschitz}^2
\\
&
\refeq{(\ref{gradientdescentoptimalitycondition})}{\geq}
(\searchfxa{\GPfunction}{\notx}{\Pa}-\notx)^\trsp{
\genericnormscale
}
(\searchfxa{\GPfunction}{\notx}{\Pa}-\notx)
\end{align}
with $\genericnormscale\doteq[\Pa\Scalefx{\GPfunction}{\notx}]^{-1}-\frac{\Lipschitz}{2}$, and by~(\ref{conditionefficiencynonlinear})  the condition~(\ref{stocharmijoconditionglobalone}) is satisfied for $\Pa=1$.

Suppose now that~$\Scalefx{\GPfunction}{\cdot}$ is continuous and~$\GPfunction$ is twice continuously differentiable in a neighborhood~$\nbhd$ of~$\notxconv$.
%
 %
%
Taylor's theorem yields
\begin{equation}\label{taylorstepsize}
\GPfunctionx{\notx}-\GPfunctionx{\noty}
=
-\grad\GPfunctionx{\llamb}^\trsp(\noty-\notx)-\frac{1}{2}\specialnorm{\noty-\notx}{\hessian{\GPfunction}{\notx}}^2
+ \zero{\norm{\noty-\notx}^2}
,
\hspace{1mm}
\forall \notx,\noty\in\nbhd.
\end{equation}
Consider the sequence~$\{\notxk{k}\}$ converging to~$\notxconv$ and the sequence~$\{\notyk{k}\}$ such that $\notyk{k}=\mapGTXfx{\Scale}{\genericset}{\GPfunction}{\notxk{k}}$ with step size~$1$ for all~$k$.
Since $\mapGTXfx{\Scale}{\genericset}{\GPfunction}{\notxconv}=\notxconv$ for any step size by stationarity of~$\notxconv$, we find that $\notyk{k}\to\notxconv$ by continuity of~$\mapGTXfunction{\Scale}{\genericset}$. Thus, for~$k$ large enough, $\notxk{k},\notyk{k}\in\nbhd$, and
it follows from~(\ref{gradientdescentoptimalitycondition}) and~(\ref{taylorstepsize}) that
\begin{equation}
\begin{array}{c}
\GPfunctionx{\notxk{k}}-\GPfunctionx{\notyk{k}}
\geq
\specialnorm{\notyk{k}-\notxk{k}}{
\genericnormscaletwo
}^2
+\zero{\norm{\notyk{k}-\notxk{k}}^2}
\end{array}
\end{equation}
with $\genericnormscaletwo\doteq\Scalefx{\GPfunction}{\notxk{k}}^{-1}-\frac{1}{2}\hessian{\GPfunction}{\notxk{k}}$. By~(\ref{conditionefficiencynonlinear}) and continuity arguments, (\ref{stocharmijoconditionglobalone}) is satisfied at~$\notxk{k}$ for large~$k$ if $\Pa=1$, i.e. $\notxk{k+1}\equiv\notyk{k}$. Hence $\mapAfx{\GPfunction}{\notxk{k}}\to 1$.
\QEDA
\end{proof}

\begin{proof}[Proof of Proposition~\ref{identificationofthesaturatednonnegativityconstraints}]
%
%
By strict complementarity at~$\llambconv$ we know that~(\ref{KKTsc}) is satisfied with coefficients $\{\CLcoefj{j}\}_{j\in\activeaffinex{\llambconv}}$ all positive.
For any~$\ballradius>0$, denote by $\neibradius{\ballradius}\doteq\{\llamb\in\setoflambdas\setst\norm{\llamb-\llambconv}\leq\ballradius\}$ a neighborhood of~$\llambconv$ in~$\setoflambdas$.
We first show that one can find a $\ballradius>0$ such that
$\activeaffinex{\llamb} \subset \activeaffinex{\llambconv} $ for any $\llamb\in\neibradius{\ballradius}$.
Otherwise there would be $j\in\{1,...,\affinen\}\setminus \activeaffinex{\llambconv}$  and a 
sequence~$(\mmuk{k})$ in~$\setoflambdas$ converging towards~$\llambconv$ such that $\cstjx{j}{\mmuk{k}} = 0$ for all~$k$. By continuity of~$\cstj{j}$, we would find $\cstjx{j}{\llambconv} = 0$ and thus $\cstj{j}\in\activeaffinex{\llambconv}$, which is a contradiction.
Since the proposition becomes trivial if $\activeaffinex{\llambconv} = \emptyset$, we suppose in the rest of the proof that $\activeaffinex{\llambconv} \neq \emptyset$, and thus $\specialnorm{\gradi{i} \smooth{\llambconv}}{} \neq 0$ by strict complementarity at~$\llambconv$.

\obsolete{
First notice that there exists a $\kscbis<\infty$ such that $\activeaffinex{\llambk{k}} \subset \activeaffinex{\llambconv} $ for $k\geq\kscbis$.
Otherwise there would be $j\in\{1,...,\affinen\}\setminus \activeaffinex{\llambconv}$  and a subsequence~$(\llambk{k})_{k\in K}$ such that $\cstjx{j}{\llambk{k}} = 0$ for $k\in K$. 
Since $\llambk{k}\to \llambconv$ and by continuity of~$\cstj{j}$, we would find $\cstjx{j}{\llambconv} = 0$ and thus $\cstj{j}\in\activeaffinex{\llambconv}$, which is a contradiction.
Now, if $\activeaffinex{\llambconv} = \emptyset$, then $\activeaffinex{\llambk{k}}\to \emptyset$ and we are done.
Therefore we suppose in the rest of the proof that $\activeaffinex{\llambconv} \neq \emptyset$, and thus $\specialnorm{\gradi{i} \smooth{\llambconv}}{} \neq 0$ by strict complementarity at~$\llambconv$.
}%

Consider a  point $\llamb\in\setoflambdas$ where $\activeaffinex{\llamb} = \activeaffinespecial $ with $\activeaffinespecial \subset \activeaffinex{\llambconv}$ and $\activeaffinespecial \neq \activeaffinex{\llambconv}$.
The affine constraints can be rewritten as $ \cstjx{j}{\llamb} = \grad \cstjx{j}{\llambconv}^\trsp (\llamb-\llambconv)  $   for all $\llamb \in \setoflambdas$ and $j\in\activeaffinex{\llambconv}$. 
We find
\begin{equation}\label{eqcontradiction}
\begin{array}{c}
 \sum\nolimits_{j\in \activeaffinex{\llambconv} } \CLcoefj{j}   \cstjx{j}{\llamb}  
 = 
\left[ \sum\nolimits_{j\in \activeaffinex{\llambconv} }  \CLcoefj{j}  \grad \cstjx{j}{\llambconv}^\trsp \right] (\llamb-\llambconv) 
\refeq{(\ref{KKTsc})}{=}
- \grad \smooth{\llambconv}^\trsp (\llamb-\llambconv) 
.
\end{array}
\end{equation}
Assume that $\grad \smooth{\llambconv}$
is a linear combination of elements of~$\{ \grad \cstjx{j}{\llambconv} \}_{j\in\activeaffinespecial}$.
We have $\cstjx{j}{\llamb}   = 0  $ for $j \in \activeaffinespecial$ and the expression in~(\ref{eqcontradiction}) is equal to~$0$.
Since $   \cstjx{j}{\llamb}  < 0  $ for $j \in \activeaffinex{\llambconv} \setminus \activeaffinespecial$,
 we find $ \sum\nolimits_{j\in \activeaffinex{\llambconv} } \CLcoefj{j}   \cstjx{j}{\llamb}  = \sum\nolimits_{j\in \activeaffinex{\llambconv} \setminus \activeaffinespecial} \CLcoefj{j}   \cstjx{j}{\llamb} < 0   $, a contradiction.
Hence~$\grad \smooth{\llambconv} $
cannot be expressed as a linear combination of elements of~$\{ \grad \cstjx{j}{\llambconv} \}_{j\in\activeaffinespecial}$ and there exists a $\minimumdist>0$  independent of~$\llamb$ such that
\obsolete{
We find that $\grad \smooth{\llambconv} $
cannot be expressed as a linear combination of elements of~$\{ \grad \cstjx{j}{\llambconv} \}_{j\in\activeaffinespecial}$.
Indeed,
we have $\cstjx{j}{\llamb}   = 0  $ for $j \in \activeaffinespecial$ and $   \cstjx{j}{\llamb}  < 0  $ for $j \in \activeaffinex{\llambconv} \setminus \activeaffinespecial$,
 and thus $ \sum\nolimits_{j\in \activeaffinex{\llambconv} } \CLcoefj{j}   \cstjx{j}{\llamb}  = \sum\nolimits_{j\in \activeaffinex{\llambconv} \setminus \activeaffinespecial} \CLcoefj{j}   \cstjx{j}{\llamb} < 0   $.
Since the constraints are affine we can write $ \cstjx{j}{\llamb} = \grad \cstjx{j}{\llambconv}^\trsp (\llamb-\llambconv)  $   for all $\llamb \in \setoflambdas$ and $j\in\activeaffinex{\llambconv}$. 
It follows from~(\ref{KKTsc}) that 
\begin{equation}
 \sum\nolimits_{j\in \activeaffinex{\llambconv} } \CLcoefj{j}   \cstjx{j}{\llamb}  
 = \left[ \sum\nolimits_{j\in \activeaffinex{\llambconv} }  \CLcoefj{j}  \grad \cstjx{j}{\llambconv}^\trsp \right] (\llamb-\llambconv) 
 = - \grad \smooth{\llambconv}^\trsp (\llamb-\llambconv) 
\end{equation}
which is equal to~$0$ and leads to a contradiction if $\grad \smooth{\llambconv}$
is a linear combination of elements of~$\{ \grad \cstjx{j}{\llambconv} \}_{j\in\activeaffinespecial}$.
Hence we can find a $\minimumdist>0$  independent of~$\llamb$ such that
}%
\begin{equation} \label{minimumdistance}
\begin{array}{c}
\Specialnorm{ \grad \smooth{\llambconv} +  \sum\nolimits_{j\in  \activeaffinespecial} \CLcoefterj{j} \grad \cstjx{j}{\llambconv} }{} 
>
 \minimumdist 
, \hspace{5mm} \forall \{ \CLcoefterj{j} \}_{j\in  \activeaffinespecial}.
\end{array}
\end{equation}
For~$\ballradius>0$, consider the function 
$
\stuff{\ballradius}\doteq\max\left(\ballradius,\max\nolimits_{\llamb\in\neibradius{\ballradius}} \norm{ \mapGTXfx{\Scale}{\setoflambdas}{\function}{\llamb}-\llamb}\right)
$ with any bounded scaling strategy~$\Scale$ and step-size policy  in $[\Pamin,1]$ ($\Pamin>0$).
Since $\mapGTXfx{\Scale}{\setoflambdas}{\function}{\llambconv}=\llambconv$, we find  by Lipschitz continuity of~$\grad\function$ and other continuity arguments that $\stuff{\ballradius}\downarrow 0$ whenever $\ballradius\downarrow 0$. It follows that for any $\maxradius>0$, one can find a $\ballradius>0$ such that $\llamb\in\neibradius{\ballradius}$ yields both $\activeaffinex{\llamb} \subset \activeaffinex{\llambconv} $ and $\stuff{\ballradius}<\maxradius$.
By Lipschitz continuity of~$\grad \smoothfunction$, we also have
$\specialnorm{ \grad \smooth{\llamb} - \grad \smooth{\llambconv} }{} \leq \lipschitz \specialnorm{\llamb-\llambconv}{} <\lipschitz\maxradius $  for any $\llamb\in\setoflambdas$, where~$\lipschitz$ denotes the Lipschitz constant.
Suppose now that $\activeaffinex{\llamb} = \activeaffinespecial $ for some $\llamb\in\neibradius{\ballradius}$ and set $\mmu=\mapGTXfx{\Scale}{\genericset}{\function}{\llamb}$ with step size $\Pa\in[\Pamin,1]$.
From~(\ref{gradientdescentoptimalitycondition}) and~(\ref{KKTsc}),
we infer the existence of nonnegative coefficients~$\{ \CLcoefbisj{j} \}_{j\in  \activeaffinespecial}$ satisfying
\begin{equation} \label{geometricinterpretationofprojection}
\begin{array}{c}
\grad \smooth{\llamb} + [
\Pa
 \Scalefx{\function}{\llamb } ]^{-1} ( \mmu - \llamb )
= - \sum\nolimits_{j\in\activeaffinespecial} \CLcoefbisj{j} \grad \cstjx{j}{ \mmu} .
\end{array}
\end{equation}
Then,
\begin{align} \label{minimumdistance}
\begin{array}{c}
 \Specialnorm{ \grad \smooth{\llambconv} +  \sum\nolimits_{j\in  \activeaffinespecial} \CLcoefbisj{j} \grad \cstx{\llambconv} }{} 
\end{array}
&
\begin{array}{c}
\refeq{(\ref{geometricinterpretationofprojection})}{=}
\Specialnorm{  [ \grad \smooth{\llambconv}  - \grad \smooth{\llamb}   ] - [
\Pa
  \Scalefx{\function}{\llamb } ]^{-1} ( \mmu - \llamb )
}{}
\end{array}
\notag
\\
&
\leq
[ \lipschitz +   (\Pamin  \eigentbound )^{-1} ] \maxradius 
,
\notag
\end{align}
 which contradicts~(\ref{minimumdistance}) if, initially, 
$
\maxradius < {\minimumdist }/[{ \lipschitz +  1 /(\Pamin  \eigentbound ) } ]  
$. 
%
Hence $\activeaffinex{\mmu}  \neq \activeaffinespecial $, which proves the first statement considering that the number of constraints~$\affinen$ is finite. The second statement is then immediate.
\QEDA\end{proof}


\label{appendixprooflemma}
\begin{proof}[Proof of Lemma~\ref{lemmalinearconvergence}]
 We already know from Proposition~\ref{propositionefficiency} that the step sizes chosen by~(\ref{stocharmijoconditionglobalone}) are equal to~$1$.
\obsolete{

We first show that~(\ref{stocharmijoconditionglobalone}) holds and the step sizes are all equal to~$1$.
\obsolete{
Using 
$
\smooth{\mmu}
=
\smooth{\llamb}
+
\grad\smooth{\mmu}^\trsp(\mmu-\llamb)
+
\int_{0}^{1} [ 
\grad\smooth{\llamb+\xi(\mmu-\llamb)}
-
\grad\smooth{\mmu}
   ]^\trsp 
(\mmu-\llamb)  \, d\xi
$ twice
\obsolete{
$
\maplocalixy{\GPfunction}{i}{\llamb}{\mmui{i}}
=
\functionx{\llamb}
+
\grad\maplocalixy{\GPfunction}{i}{\llamb}{\mmui{i}}^\trsp(\mmui{i}-\llambi{i})
+
\int_{0}^{1} [ 
\grad\maplocalixy{\GPfunction}{i}{\llamb}{\llambi{i}+\xi(\mmui{i}-\llambi{i})}
-
\grad\maplocalixy{\GPfunction}{i}{\llamb}{\mmui{i}}
   ]^\trsp 
(\mmui{i}-\llambi{i})  \, d\xi
$,
}%
 at step~$k$, one finds,
\begin{align}
\gradi{\Pseqk{k}}\smooth{\llambk{k}}^\trsp(\llambki{k+1}{\Pseqk{k}}-\llambki{k}{\Pseqk{k}})
&\geq \label{Lequation}
\smooth{\llambk{k+1}}-\smooth{\llambk{k}}-\frac{\maxLi{\Pseqk{k}}}{2}\norm{\llambki{k}{\Pseqk{k}}-\llambki{k+1}{\Pseqk{k}}}^2
\\
&\geq  \label{Lequationtwo}
-\maxLi{\Pseqk{k}}\norm{\llambki{k}{\Pseqk{k}}-\llambki{k+1}{\Pseqk{k}}}^2.
\end{align}
}%
If~$\Pa$ denotes the step size used at step~$k$, then
\begin{align}
\smooth{\llambk{k}}-\smooth{\llambk{k+1}}
&
\refeq{(\ref{Lequation})}{\geq}
-\gradi{\Pseqk{k}}\smooth{\llambk{k}}^\trsp(\llambki{k+1}{\Pseqk{k}}-\llambki{k}{\Pseqk{k}})-\frac{\maxLi{i}}{2}\norm{\llambki{k}{\Pseqk{k}}-\llambki{k+1}{\Pseqk{k}}}^2
\\
&
\refeq{(\ref{gradientdescentoptimalitycondition})}{\geq}
(\llambki{k+1}{\Pseqk{k}}-\llambki{k}{\Pseqk{k}})^\trsp{\left[(\Pa\constantScalei{\Pseqk{k}})^{-1}-\frac{\maxLi{\Pseqk{k}}}{2}\identityi{\Pseqk{k}}\right]}
(\llambki{k+1}{\Pseqk{k}}-\llambki{k}{\Pseqk{k}})
\end{align}
which validates the step size selection rule~(\ref{stocharmijoconditionglobalone}) for $\Pa=1$ under the assumptions on~$\Psig$ and~$\constantScale$.
}%
%
%
The rest of the proof---herein provided for completeness and comparison---follows the lines of that of~\cite[Theorem~5]{nesterov12} with the difference that we reason with the norm~$\specialnorm{\cdot}{\smartscale}$. 
We have
\begin{align}
\specialnorm{\llambk{k+1}-\llambconv}{\smartscale}^2
&
=
\specialnorm{\llambk{k}-\llambconv}{\smartscale}^2
+
2(\llambk{k} - \llambconv)^\trsp \smartscale (\llambk{k+1}-\llambk{k})
+
\specialnorm{\llambk{k+1}-\llambk{k}}{\smartscale}^2
\notag
\\
&
=
\specialnorm{\llambk{k}-\llambconv}{\smartscale}^2
+
2(\llambk{k+1} - \llambconv)^\trsp \smartscale (\llambk{k+1}-\llambk{k})
-
\specialnorm{\llambk{k+1}-\llambk{k}}{\smartscale}^2
\notag
\\
&
\refeq{(\ref{gradientdescentoptimalitycondition})}{\leq}
\specialnorm{\llambk{k}-\llambconv}{\smartscale}^2
+
\frac{2}{\nn\probadistri{\Pseqk{k}}}\gradi{\Pseqk{k}}\smooth{\llambk{k}}^\trsp(\llambconvi{\Pseqk{k}}-\llambki{k+1}{\Pseqk{k}})
-
\specialnorm{\llambk{k+1}-\llambk{k}}{\smartscale}^2
\notag
\\
&
\refeq{(\ref{Lequation})}{\leq}
\specialnorm{\llambk{k}-\llambconv}{\smartscale}^2
+
\frac{2
\left[
\gradi{\Pseqk{k}}\smooth{\llambk{k}}^\trsp (\llambconvi{\Pseqk{k}}-\llambki{k}{\Pseqk{k}})
+
\smooth{\llambk{k}}-\smooth{\llambk{k+1}}
\right]
}{\nn\probadistri{\Pseqk{k}}}
\notag
\\
&
\refeq{(\ref{stocharmijoconditionglobalone})}{\leq}
\specialnorm{\llambk{k}-\llambconv}{\smartscale}^2
+
\frac{2
\gradi{\Pseqk{k}}\smooth{\llambk{k}}^\trsp (\llambconvi{\Pseqk{k}}-\llambki{k}{\Pseqk{k}})
}{\nn\probadistri{\Pseqk{k}}}
+
\correctproba{2\left[\smooth{\llambk{k}}-\smooth{\llambk{k+1}}\right]}{
\frac{2\left[\smooth{\llambk{k}}-\smooth{\llambk{k+1}}\right]
}{\nn\probamin}
}
\notag
\end{align}
which yields, by expectation in~$\Pseqk{k}$ and rearrangement of the terms,
\begin{equation}\label{loss}
\expectationcond{
\supernorm{\llambk{k+1}}
}{\llambk{k}}
\refeq{(\ref{supernormfunction})}{\leq}
\supernorm{\llambk{k}}
-
\correctproba{\frac{1}{\nn}}{\probamin}
\grad\smooth{\llambk{k}}^\trsp (\llambk{k}-\llambconv).
\end{equation}
Since $\grad\smooth{\llambk{k}}^\trsp (\llambk{k}-\llambconv)\geq\smooth{\llambk{k}}-\smooth{\llambconv}$ by convexity of~$\smoothfunction$, 
we find by computing  successive conditional expectations,
\begin{eqnarray}
\expectation{
\supernorm{\llambk{k+1}}
}
&
{\leq}
&
\supernorm{\llambk{0}}
-
\correctproba{\frac{1}{\nn}}{\probamin}
\sum\nolimits_{t=0}^{k}
\expectation{\smooth{\llambk{t}}-\smooth{\llambconv}}
\\
&
\refeq{(\ref{stocharmijoconditionglobalone})}{\leq}
&
\supernorm{\llambk{0}}
-
\correctproba{\frac{k+1}{\nn}}{\probamin(k+1)}
\expectation{\smooth{\llambk{k+1}}-\smooth{\llambconv}}
\end{eqnarray}
which shows~(\ref{sublinearrate}).

When Assumption~\ref{assumptionstrongconvexity} holds, we proceed as in~(\ref{Lequation}) and find, 
\begin{equation}
\grad\smooth{\llambk{k}}^\trsp(\llambconv-\llambk{k})
\leq \label{Kequation}
\smooth{\llambconv}-\smooth{\llambk{k}}-\frac{\minK}{2}\specialnorm{\llambk{k}-\llambconv}{\smartscale}^2
\leq  
-\minK\specialnorm{\llambk{k}-\llambconv}{\smartscale}^2
.
\end{equation}
Substituting the two inequalities~(\ref{Kequation}) into~(\ref{loss}) with relative weights $\coefineq=2 \minK(\minK+\correctproba{1}{\nn\probamin})^{-1}\in(0,1]$ and $1-\coefineq$ yields~(\ref{linearrate}).
\QEDA
\end{proof}

\begin{lemma}\label{lemmamatrices}
Let $\LmH=(\LmHij{i}{j})$ be a symmetric  block matrix of $\REAL{\genericn\times\genericn}$ such that $\LmH=\LmD-\LmL-\LmL^\trsp$, where $\LmD=\diag{\LmDi{1},...,\LmDi{\genericn}}$ is block diagonal and~$\LmL$ is strictly lower triangular. If $\LmT=\diag{\LmTi{1},...,\LmTi{\genericn}}$ is a symmetric, positive definite, block diagonal  matrix of $\REAL{\genericn\times\genericn}$ and $\LmGi{i}\doteq\identityi{\genericn}-\diag{0,...,0,\identityi{i},0,...,0}\LmT\LmH$ for $i=1,...,\genericn$, then $\LmGi{\genericn}\LmGi{\genericn-1}\cdots\LmGi{1}=(\LmT^{-1}-\LmL)^{-1}(\LmT^{-1}-\LmD+\LmL^\trsp)$.
\end{lemma}
\begin{proof}
Since the result is trivial for $\genericn=1$ we suppose that $\genericn\geq2$. We proceed by induction on~$\genericn$.
Let $\LmMi{1}\doteq\LmGi{1}$. For $2\leq i\leq\genericn$, define $\LmMi{i}\doteq\LmGi{i}\LmGi{i-1}\cdots\LmGi{1}$ and decompose~$\LmT$, $\LmD$, $\LmL$ and the $\genericn\times\genericn$ identity matrix~$\identityi{\genericn}$ into
\begin{equation}
\LmT
=
\makeblockmatrixnine{\LmTmsi{i}}{0}{0}{0}{\LmTi{i}}{0}{0}{0}{\LmTpsi{i}}, 
\ 
\LmD
=
\makeblockmatrixnine{\LmDmsi{i}}{0}{0}{0}{\LmDi{i}}{0}{0}{0}{\LmDpsi{i}}, 
\ 
\LmL
=
\makeblockmatrixnine{\LmLmsi{i}}{0}{0}{\Lmli{i}}{0}{0}{\Lmlpsi{i}}{\Lmlambi{i}}{\LmLpsi{i}}, 
\ 
\identityi{\genericn}
=
\makeblockmatrixnine{\identitymsi{i}}{0}{0}{0}{\identityi{i}}{0}{0}{0}{\identitypsi{i}}.
\end{equation}
We can write
\begin{equation}
\LmGi{i}
=
\makeblockmatrixnine{\identitymsi{i}}{0}{0}{\LmTi{i}\Lmli{i}}{\identityi{i}-\LmTi{i}\LmDi{i}}{\LmTi{i}\Lmlambi{i}^\trsp}{0}{0}{\identitypsi{i}}, 
\hspace{5mm} 2\leq i \leq\genericn.
\label{Gi}
\end{equation}
For some~$i\geq2$, notice that~$\LmTmsi{i}$ is nonsingular, as well as~$\LmTmsi{i}^{-1}-\LmLmsi{i}$, and  suppose  that
\begin{align}
\LmMi{i-1}
&=
\makeblockmatrixfour{(\LmTmsi{i}^{-1}-\LmLmsi{i})^{-1}(\LmTmsi{i}^{-1}-\LmDmsi{i}+\LmLmsi{i}^\trsp)}{(\LmTmsi{i}^{-1}-\LmLmsi{i})^{-1}\vect{\Lmli{i},\Lmlpsi{i}}^\trsp}{0}{\identitypsi{i-1}}
\label{Miminusoneinit}
\\
&=
\makeblockmatrixnine{\LmZmsi{i}(\LmTmsi{i}^{-1}-\LmDmsi{i}+\LmLmsi{i}^\trsp)}{\LmZmsi{i}\Lmli{i}^\trsp}{\LmZmsi{i}\Lmlpsi{i}^\trsp}{0}{\identityi{i}}{0}{0}{0}{\identitypsi{i}},
\label{Miminusone}
\end{align}
where  $\LmZmsi{i}\doteq(\LmTmsi{i}^{-1}-\LmLmsi{i})^{-1}$.
By block matrix inversion of~$\LmTmsi{i+1}^{-1}-\LmLmsi{i+1}$, we have
\begin{equation}\label{blockinversion}
\LmTmsi{i+1}^{-1}-\LmLmsi{i+1}
=
\makeblockmatrixfour{\LmZmsi{i}}{0}{\LmTi{i}\Lmli{i} \LmZmsi{i}}{\LmTi{i}}
,\
\LmTmsi{i+1}^{-1}-\LmDmsi{i+1}+\LmLmsi{i+1}^\trsp
=
\makeblockmatrixfour{\LmTmsi{i}^{-1}-\LmDmsi{i}+\LmLmsi{i}^\trsp}{\Lmli{i}^\trsp}{0}{\LmTi{i}^{-1}-\LmDi{i}}
.
\end{equation}
It follows from~(\ref{Gi}),~(\ref{Miminusone}) and 
$\LmMi{i}=\LmGi{i}\LmMi{i-1}
$
that
\begin{align}
\LmMi{i}
&=
\makeblockmatrixnine{\LmZmsi{i}(\LmTmsi{i}^{-1}-\LmDmsi{i}+\LmLmsi{i}^\trsp)}{\LmZmsi{i}\Lmli{i}^\trsp}{\LmZmsi{i}\Lmlpsi{i}^\trsp}
{\LmTi{i}\Lmli{i}    \LmZmsi{i}(\LmTmsi{i}^{-1}-\LmDmsi{i}+\LmLmsi{i}^\trsp)}
{\LmTi{i}\Lmli{i}  \LmZmsi{i}\Lmli{i}^\trsp +\identityi{i} -\LmTi{i}\LmDi{i}}{\LmTi{i}\Lmli{i} \LmZmsi{i}\Lmlpsi{i}^\trsp + \LmTi{i}\Lmlambi{i}^\trsp }{0}{0}{\identitypsi{i}},
\\
&\refeq{(\ref{blockinversion})}{=}
\makeblockmatrixfour{(\LmTmsi{i+1}^{-1}-\LmLmsi{i+1})^{-1}(\LmTmsi{i+1}^{-1}-\LmDmsi{i+1}+\LmLmsi{i+1}^\trsp)}{(\LmTmsi{i+1}^{-1}-\LmLmsi{i+1})^{-1}\vect{\Lmli{i+1},\Lmlpsi{i+1}}^\trsp}{0}{\identitypsi{i}},
\end{align}
where we have used $\vect{\Lmli{i+1},\Lmlpsi{i+1}}\equiv({\Lmlpsi{i}}  \ {\Lmlambi{i}})$.
Since~(\ref{Miminusoneinit}) holds for $i=1$, we find by induction
$
\LmGi{\genericn}\cdots\LmGi{1}
=
\LmMi{\genericn}
=
(\LmT^{-1}-\LmL)^{-1}(\LmT^{-1}-\LmD+\LmL^\trsp)$.
\end{proof}




\jotaIEEE{}{
\ifCLASSOPTIONcaptionsoff
  \newpage
\fi
}



\bibliographystyle{IEEEtran}
 \bibliography{IEEEabrv,\setpath%
bibli_ob_TUB}
%

%

\obsolete{
\begin{IEEEbiography}{Michael Shell}
Biography text here.
\end{IEEEbiography}

\begin{IEEEbiographynophoto}{John Doe}
Biography text here.
\end{IEEEbiographynophoto}


\begin{IEEEbiographynophoto}{Jane Doe}
Biography text here.
\end{IEEEbiographynophoto}




}

\end{document}